\documentclass{amsart}

\usepackage{amssymb}
\usepackage{overpic}
\usepackage{graphics}
\usepackage{epsfig}
\usepackage{amsmath}
\usepackage{amsfonts}
\usepackage{paralist}

\newtheorem{theorem}{Theorem}[section]
\newtheorem{corollary}{Corollary}

\newtheorem{lemma}[theorem]{Lemma}
\newtheorem{proposition}{Proposition}

\newtheorem{example}{Example}
\theoremstyle{definition}
\newtheorem{definition}[theorem]{Definition}
\newtheorem{remark}{Remark}

\newtheorem*{theorem1}{Theorem 3.8 and 3.9 (Existence and Uniqueness)}
\newtheorem*{theorem2}{Theorem 3.11 (Spectral Equivalence)}
\newtheorem*{theorem3}{Theorem 6.6 and 6.8 (Dynamical Network Expansions)}

\begin{document}

\title[Isospectral Transforms, Spectral Equivalence, and Global Stability]{Isospectral Graph Transformations, Spectral Equivalence, and Global Stability of Dynamical Networks}

\author{L. A. Bunimovich$^1$}

\author{B. Z. Webb$^2$}



\maketitle

\begin{center}
$^{1}$ \smaller{ABC Math Program and School of Mathematics, Georgia Institute of Technology, 686 Cherry Street, Atlanta, GA 30332, USA}\\
$^{2}$ Brigham Young University, Department of Mathematics, Provo, UT 84602, USA\\
E-mail: bunimovih@math.gatech.edu and bwebb@math.byu.edu
\end{center}

\begin{abstract}
In this paper we present a general procedure that allows for the reduction or expansion of any network (considered as a weighted graph). This procedure maintains the spectrum of the network's adjacency matrix up to a set of eigenvalues known beforehand from its graph structure. This procedure can be used to establish new equivalence relations on the class of all weighted graphs (networks) where two graphs are equivalent if they can be reduced to the same graph. Additionally, dynamical networks (or any finite dimensional, discrete time dynamical system) can be analyzed using isospectral transformations. By so doing we obtain stronger results regarding the global stability (strong synchronization) of dynamical networks when compared to other standard methods.
\end{abstract}

\section{Introduction}
Real world networks, i.e. those found in nature and technology, typically have a complicated irregular structure and consist of a large number of interconnected dynamical units. Coupled biological and chemical systems, neural networks, social interacting species, and the Internet are only a few such examples \cite{Albert02,Dorogovtsev03,Faloutsos99,Newman06,Strogatz03,Watts99}. Because of this complexity, the first approach to capturing the global properties of such systems has been to model them as graphs whose nodes represent elements of the network and the edges define a topology (graph of interactions) between these elements. This is principally a static approach to modeling networks in which only the structure of the network's interactions (and perhaps the interaction's strength) are analyzed.

A more general line of research is concerned with the dynamical properties of such networks. Mathematically, this has been done by modeling networks as interacting dynamical systems \cite{Afriamovich07,Blank06,Chazottes05}. Important processes that are studied within this framework include synchronization \cite{CHB07} or contact processes, such as opinion formation and epidemic spreading. These studies give strong evidence that the structure of a network can have a substantial impact on the network dynamics \cite{BLMC06}.

Regarding this connection, the impact of a dynamical network's underlying structure on its dynamics is of major interest. In this the spectrum of the network's adjacency matrix has naturally emerged as a key quantity in the study of a variety of dynamical networks. For example, systems of interacting dynamical units are known to synchronize depending only on the dynamics of the uncoupled dynamical systems and the spectral radius of the network's adjacency matrix \cite{HOR06,HOR05}. Moreover, the eigenvalues of a network are important for determining if the dynamics of a network are regular or chaotic. \cite{Afriamovich07,Blank06}.

To aid in understanding this interplay between the structure (graph of interactions) and dynamics of a network in this paper we introduce the concept of an \textit{isospectral graph transformation}. Such transformations allow one to modify the structure of a network, at the level of a graph, while maintaining properties related to the network's dynamics. For the sake of the reader, we will present several of our results regarding isospectral graph transformations here in the introduction. However, we will do so in a very informal manner as exact formulations require the introduction of many formal definitions.

Mathematically, an isospectral graph transformation is a graph operation that modifies the structure of a graph while preserving the eigenvalues of the graph's (weighted) adjacency matrix. Besides modifying interactions, such transforms can also reduce (or increase) the number of nodes in a network. As not to violate the fundamental theorem of algebra, isospectral graph transformations preserve the spectrum of the graph (in particular the number of eigenvalues) by allowing edges to be weighted by rational functions. In other words, a matrix of smaller size can have the same spectrum as a larger matrix if the entries of the smaller are allowed to be functions of a spectral parameter. This fundamental result is contained in theorem \ref{maintheorem}.

For a graph $G$ with vertex set $V$ this is done by reducing (or expanding) $G$ with respect to specific subsets of $V$ known as structural sets (see section 3.1 for exact definitions). As a typical graph has many different structural sets it is possible to consider different isospectral transformations of the same graph $G$ as well as sequences of such transformations. However, given the large number of possible reductions of a typical graph the results of theorems \ref{theorem-1} and \ref{theorem-3} are quite useful and central to the theory of isospectral graph reductions.

\begin{theorem1} If $G$ is a graph with vertex set $V$ then it is possible to sequentially reduce $G$ to a graph on any nonempty subset $\mathcal{V}\subseteq V$ of its original vertex set. Moreover, the resulting graph does not depend on the number or sequence of intermediate reductions but only on the choice of the set $\mathcal{V}$.
\end{theorem1}

Because it is possible to uniquely reduce a graph over a subset of its vertices, isospectral graph reductions can be used to induce new equivalence relations on the set of all graphs.

\begin{theorem2}
Suppose we have a rule $\tau$ that selects a \textit{unique} subset of vertices $\tau(G)\subseteq V$ for any graph $G$. Then the rule $\tau$ induces an equivalence relation (spectral equivalence) on the set of all graphs (or all networks considered as graphs). Namely, two graphs (networks) are spectrally equivalent if they can be isospectrally reduced to the same graph (network).
\end{theorem2}

Observe that in order to induce an equivalence relation we need a rule that selects a \textit{unique} subset of the vertices of a network. For instance, the rule ``choose any vertex'' does not work. Given this restriction, an expert in biology, sociology, etc. must determine appropriate vertex sets or rules over which to reduce e.g. minimal (maximal) degree, minimal (maximal) clustering coefficient, etc. Once this is done our procedure can then be used to reduce the network(s). However, the rule must be chosen by an expert in a corresponding applied field.

Importantly, if using a particular rule does or does not yield insight into a particular network (or set of networks) it is possible to try another rule. In this regard, our procedure provides a working tool for scientists, engineers, and others for the analysis of their respective networks. However, determining the correct rule to use requires an experts knowledge of his/her own field.

Network reductions themselves are motivated by the need to find ways of reducing a network's complexity while simultaneously maintaining some important network characteristic(s). This amounts to coarse graining or finding the right scale at which to view the network. Isospectral network transformations, specifically reductions, serve this need by allowing one to view a network as a smaller network with essentially the same spectrum. Finding the right scale here then amounts to finding the right structural set or rule over which to reduce the network, which is the job of the expert.

We note that by allowing the edge weights of the graphs to be functions it may appear that our procedure is trading the complexity of the graph's structure for complex edge weights. In fact, it is often possible to isospectrally transform a graph in a way that does not change edge weights of the graph or preserves the type of edge weights a graph has e.g. positive integers, real numbers, etc. (See theorems \ref{theorem-2} and \ref{theoremunital}.)

Moreover, it is worth noting that one can reduce the \textit{line graph} (often called an \textit{edge graph}) $L(G)$ of a graph $G$. As the line graph of $L(G)$ is formed by associating a vertex with each edge of $G$ then reducing over a structural set of $L(G)$ amounts to reducing $G$ over a particular edge set. From this point of view, isospectral graph reductions can be done either over vertex sets or edge sets.

Of primary interest is the fact that isospectral graph transformations suggest other useful transformations on interacting dynamical systems i.e. dynamical networks. An important example in this paper is a \textit{dynamical network expansion} in which a dynamical network is modified in a way that preserves its dynamics but alters its associated graph structure.

Such transforms provide a new tool for the study of the interplay between the structure (topology) and dynamics of dynamical networks. Much as general dynamical systems are investigated via change of coordinates, isospectral network transforms introduce a mechanism for rearranging the specific (graph) structure of a system while preserving the networks dynamics in an essential way. By so doing the original network's dynamics can be investigated by studying the transformed dynamical network. The main result of this technique is the following.

\begin{theorem3}
If an expansion of a dynamical network has a globally attracting fixed point then the network itself has a globally attracting fixed point. Moreover, determining whether a dynamical network has a global attractor is more easily done by analyzing an expansion rather than the original network.
\end{theorem3}

That is, the method of dynamical network expansion provides a new efficient tool for establishing the global stability of dynamical networks. Moroever, the result(s) theorem \ref{gafp} and theorem \ref{last} generalize the results given in \cite{Afriamovich07,Afriamovich10}. Moreover, as any finite dimension dynamical system with discrete time can be considered a dynamical network this technique can be applied to a very large class of dynamical systems.

The structure of the paper is as follows. In section 2 we introduce dynamical networks along with their graph of interactions. We then give sufficient conditions under which a dynamical network will have a globally attracting fixed point. Section 3 introduces the procedure of an isospectral graph reduction and results on sequences of such transformations. Graph transformations that preserve weight sets and spectral equivalence are treated in section 4. The proof of the theorems contained in sections 3 and 4 are then given in section 5. In section 6 the behavior of dynamical networks is investigated via dynamical network expansions. Section 7 contains some concluding remarks.

\section{Preliminaries}

Dynamical networks or networks of interacting dynamical systems are composed of (i) local dynamical systems which have their own (local intrinsic) dynamics, (ii) interactions between these (elements of the network) local systems, and (iii) the graph of interactions (topology of the network).

\subsection{Dynamical Networks and their Graph Structure}

Our first task is to establish a mathematical framework for the investigation of dynamical networks. This is done following the approach in \cite{Afriamovich07}.

Let $i\in \mathcal{I}=\{1,\dots,n\}$ and $T_i: X_i\rightarrow X_i$ be maps on the complete metric space $(X_i,d)$ where
\begin{equation}\label{eq0.1}
L_i=\sup_{x_i\neq y_i\in X_i}\frac{d(T_i(x_i),T_i(y_i))}{d(x_i,y_i)}<\infty.
\end{equation}
Let $(T,X)$ denote the direct product of the local systems $(T_i,X_i)$ over $\mathcal{I}$ on the complete metric space $(X,d_{max})$ where for $\textbf{x},\textbf{y}\in X$
$$d_{max}(\textbf{x},\textbf{y})=\max_{i\in\mathcal{I}}\{d(x_i,y_i)\}.$$

\begin{definition}\label{interaction}
A map $F:X \rightarrow X$ is called an \textit{interaction} if for every $j\in \mathcal{I}$ there exists a nonempty collection of indices $\mathcal{I}_j\subseteq\mathcal{I}$ and a continuous function
$$F_j:\bigoplus_{i\in \mathcal{I}_j} X_i\rightarrow X_j,$$
that satisfies the following Lipschitz condition for constants $\Lambda_{ij}\geq 0:$
\begin{equation}\label{eq2.3}
d\big(F_j(\{x_i\}),F_j(\{y_i\})\big)\leq \sum_{i\in \mathcal{I}_j} \Lambda_{ij} d(x_i,y_i)
\end{equation}
for all $\{x_i\},\{y_i\}\in \bigoplus_{i\in \mathcal{I}_j} X_i$ where $\{x_i\}$ is the restriction of $x\in X$ to $\bigoplus_{i\in \mathcal{I}_j} X_i$.
Then the (\textit{interaction}) map $F$ is defined as follows:
$$F(x)_j=F_j(\{x_i\}), \ \ j\in \mathcal{I}, \ \ i\in\mathcal{I}_j.$$
\end{definition}

The constants $\Lambda_{ij}$ in definition \ref{interaction} form the \textit{Lipschitz matrix} $\Lambda\in\mathbb{R}^{n\times n}$ where the entry $\Lambda_{ij}=0$ if $i\notin\mathcal{I}_j$.

\begin{definition}\label{netdef}
The superposition $\mathcal{F}=F\circ T$ generates the dynamical system $(\mathcal{F},X)$ which is a \textit{dynamical network}.
\end{definition}

To each dynamical network $(\mathcal{F},X)$ there is a corresponding \textit{graph of interactions}. A graph of interactions is an unweighted directed graph representing the structure of the interaction $F$ between the network elements.

An \textit{unweighted directed graph} $G$ is an ordered pair $G=(V,E)$. The sets $V$ and $E$ are the \textit{vertex set} and \textit{edge set} of $G$ respectively. If the vertex set $V=\{v_1,\dots,v_n\}$ then we denote the edge from $v_i$ to $v_j$ by $e_{ij}$.

\begin{definition}
Let $F:X\rightarrow X$ be an interaction. The graph $\Gamma_{F}=(V,E)$ with vertex set $V=\{v_1,\dots,v_n\}$ and edge set $E=\{e_{ij}: i\in \mathcal{I}_j, \ j\in\mathcal{I}\}$ is called the \textit{graph of interactions} of $F$.
\end{definition}

We note that each vertex $v_i\in V$ of the graph of interactions $\Gamma_F=(V,E)$ corresponds to the $i$th element of the dynamical network $(\mathcal{F},X)$. Moreover, there is an edge $e_{ij}\in E$ if and only if the $j$th coordinate of the interaction $F(\textbf{x})$ depends on the $i$th coordinate of $\textbf{x}$.

\begin{remark}
Suppose the interaction $F:X\rightarrow X$ is continuously differentiable and each $X_i\subset\mathbb{R}$. If $DF$ is the matrix of first partial derivatives of $F$ then the constants
$$\Lambda_{ij}=\max_{\textbf{x}\in X}|(DF)_{ji}(\textbf{x})|$$
satisfy condition (\ref{eq2.3}) for the interaction $F$.
\end{remark}

\begin{example}\label{ex1}
Let $F:[0,1]^3\rightarrow[0,1]^3$ be the interaction given by
$$F(\mathbf{x})=\left[
\begin{array}{c}
x_1(1-\alpha x_2)\\
\beta x_1x_3\\
x_3(1-\gamma x_2)
\end{array}
\right]  \ \ \text{for} \ \ 0\leq \alpha,\beta,\gamma \leq 1.
$$
Using the Lipschitz constants $\Lambda_{ij}=\max_{\textbf{x}\in X}|(DF)_{ji}(\textbf{x})|$ the interaction $F$ has Lipschitz matrix
$$\Lambda=\left[
\begin{array}{ccc}
1&\beta&0\\
\alpha&0&\gamma\\
0&\beta&1
\end{array}
\right].
$$
The graph of interactions $\Gamma_F$ of $F$ is the graph shown in figure \ref{fig1}.
\end{example}

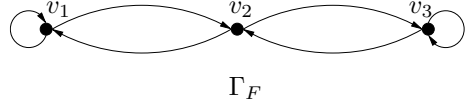
\begin{figure}
  \begin{center}
    \begin{overpic}[scale=.5]{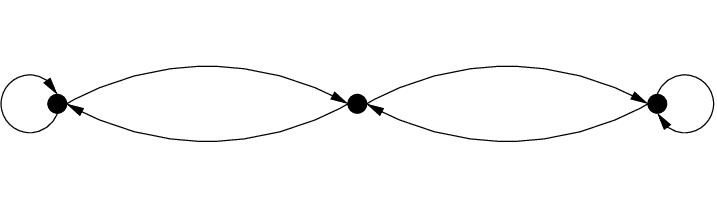}
    \put(48,0){$\Gamma_F$}
    \put(8,18){$v_1$}
    \put(48,18){$v_2$}
    \put(87.5,18){$v_3$}
    \end{overpic}
  \end{center}
  \caption{The graph $\Gamma_F$ corresponding to the interaction $F$ in example \ref{ex1}.}\label{fig1}
\end{figure}

\subsection{Stability of Dynamical Networks}
One of the goals of this paper is to find sufficient conditions under which a dynamical network $(\mathcal{F},X)$ has simple dynamics. By simple we mean that the system $(\mathcal{F},X)$ has a globally attracting fixed point. That is, that all trajectories of the dynamical network converge to a single point as time tends to infinity.

\begin{definition}
The dynamical network $(\mathcal{F},X)$ has a \textit{globally attracting fixed point} $\tilde{\textbf{x}}\in X$ if for any $\textbf{x}\in X$, $$\displaystyle{\lim_{k\rightarrow\infty}d_{max}\big(\mathcal{F}^k(\textbf{x}),\tilde{\textbf{x}}\big)=0}.$$ If $(\mathcal{F},X)$ has a globally attracting fixed point we say it is \textit{globally stable}.
\end{definition}

Recall that the constants $L_i$ and $\Lambda_{ij}$ come from (\ref{eq0.1}) and (\ref{eq2.3}) for the local systems $(T,X)$ and interaction $F$ respectively. For the dynamical network $(\mathcal{F},X)$ we define the matrix $M_\mathcal{F}=\Lambda^T\cdot diag[L_1,\dots,L_n]$.  That is,

$$M_\mathcal{F}=  \left( \begin{array}{cccc}
\Lambda_{11}L_1 & \dots & \Lambda_{n1}L_n \\
\vdots & \ddots & \vdots \\
\Lambda_{1n}L_1 & \dots & \Lambda_{nn}L_n \end{array} \right).$$

We let $\rho(M_\mathcal{F})$ denote the \textit{spectral radius} of the matrix $\Lambda$, i.e. if $\sigma(M_\mathcal{F})$ are the eigenvalues of $\Lambda$ then
$$\rho(M_\mathcal{F})=\max\{|\lambda|:\lambda\in\sigma(M_\mathcal{F})\}.$$
The following theorem gives a sufficient condition under which a dynamical network has a globally attracting fixed point.

\begin{theorem}\label{stability}
If $\rho\big(M_\mathcal{F}\big)<1$ then the dynamical network $(\mathcal{F},X)$ has a globally attracting fixed point.
\end{theorem}


Before proving theorem \ref{stability} we observe that if $\mathbf{v}\in\mathbb{C}^{n\times 1}$ then its $\ell^\infty$ norm $||\mathbf{v}||_\infty=\max_i|v_i|$. The $\ell^\infty$ norm of a matrix $A\in\mathbb{C}^{n\times n}$ is
$$|||A|||_\infty=\max_{1\leq i\leq n}\sum_{j=1}^n|a_{ij}|$$
or the maximum absolute row sum of $A$. Moreover, the $\ell^\infty$ matrix norm is \textit{sub-multiplicative}, i.e. $|||AB|||_{\infty}\leq|||A|||_{\infty}|||B|||_{\infty}$ for any $A,B\in\mathbb{C}^{n\times n}$. A proof of theorem \ref{stability} is the following.

\begin{proof}
For $\textbf{x},\textbf{y}\in X$ and $1\leq j\leq n$
\begin{align*}
d\big(\mathcal{F}(\textbf{x})_j,\mathcal{F}(\textbf{y})_j\big)=& d\big(F_j(\{T(\textbf{x})_i\}),F_j(\{T(\textbf{y})_i\})\big)\\
\leq& \sum_{i\in\mathcal{I}}\Lambda_{ij} d\big(T_i(x_i),T_i(y_i)\big)\\
\leq& \sum_{i\in \mathcal{I}}\Lambda_{ij}L_i d(x_i,y_i).
\end{align*}
Therefore, each entry of the column vector
$$\left[\begin{array}{c}
  d\big(\mathcal{F}(\textbf{x})_1,\mathcal{F}(\textbf{y})_1\big)\\
  \vdots\\
  d\big(\mathcal{F}(\textbf{x})_n,\mathcal{F}(\textbf{y})_n\big)
  \end{array}\right]\leq
  M_\mathcal{F}\left[\begin{array}{c}
  d\big(x_1,y_1\big)\\
  \vdots\\
  d\big(x_n,y_n\big)
  \end{array}\right].$$

$$\text{As} \ \ \left[\begin{array}{c}
  d\big(\mathcal{F}^2(\textbf{x})_1,\mathcal{F}^2(\textbf{y})_1\big)\\
  \vdots\\
  d\big(\mathcal{F}^2(\textbf{x})_n,\mathcal{F}^2(\textbf{y})_n\big)
 \end{array}\right]\leq
  M_\mathcal{F}\left[\begin{array}{c}
  d\big(\mathcal{F}(\textbf{x})_1,\mathcal{F}(\textbf{y})_1\big)\\
  \vdots\\
  d\big(\mathcal{F}(\textbf{x})_n,\mathcal{F}(\textbf{y})_n\big)
  \end{array}\right]\leq
  M^2_\mathcal{F}\left[\begin{array}{c}
  d\big(x_1,y_1\big)\\
  \vdots\\
  d\big(x_n,y_n\big)
  \end{array}\right]$$
by the same reasoning then inductively
\begin{equation}\label{theoremeq}
d_{max}\big(\mathcal{F}^k(\textbf{x}),\mathcal{F}^k(\textbf{y})\big)\leq\Big|\Big|M^k_\mathcal{F}
\left[\begin{array}{c}
d(x_1,y_1)\\
\vdots\\
d(x_n,y_n)
\end{array}\right]
\Big|\Big|_{\infty}
\end{equation}
for all $k>0$.

By the \textit{Jordan canonical form theorem} there is a non-singular matrix $A\in\mathbb{C}^{n\times n}$ and a block-diagonal matrix $J\in\mathbb{C}^{n\times n}$ such that $\mathcal{M}_{\mathcal{F}}=AJA^{-1}$. In particular,
$$J=\left[
\begin{array}{cccc}
J_{m_1}(\lambda_1)&0&\dots&0\\
0&J_{m_2}(\lambda_2)& &\vdots\\
\vdots& &\ddots&0\\
0&\dots&0&J_{m_t}(\lambda_t)\\
\end{array}
\right]$$
where
$$J_{m_i}(\lambda_i)=\left[
\begin{array}{ccccc}
\lambda_i&1&0&\dots&0\\
0&\lambda_i&1&&0\\
\vdots& &\ddots&\ddots&\vdots\\
0&&&\lambda_i&1\\
0&0&\dots&0&\lambda_i\\
\end{array}
\right]\in\mathbb{C}^{m_i\times m_i}$$
and $\lambda_i\in\sigma(\mathcal{M}_\mathcal{F})$ for $1\leq i\leq t$. Moreover, as $J$ is block diagonal then for $k\geq0$,
$$J^k=\left[
\begin{array}{cccc}
J_{m_1}^k(\lambda_1)&0&\dots&0\\
0&J_{m_2}^k(\lambda_2)& &\vdots\\
\vdots& &\ddots&0\\
0&\dots&0&J^k_{m_t}(\lambda_t)\\
\end{array}
\right].$$

As the $\ell^\infty$ norm of a matrix is its maximum absolute row sum then $$\displaystyle{|||J^k|||_{\infty}=\max_{1\leq i\leq t}|||J^k_{m_i}(\lambda_i)|||_{\infty}}.$$ A well-known result states that the $k$th power of any Jordan block $J_{m_i}(\lambda_i)$ is
$$J^k_{m_i}(\lambda_i)=\left[
\begin{array}{ccccc}
{k\choose 0}\lambda_{i}^k&{k\choose 1}\lambda_{i}^{k-1}&{k\choose 2}\lambda_{i}^{k-2}&\dots&{k\choose m_i-1}\lambda_{i}^{k-m_i+1}\\
0&{k\choose 0}\lambda_{i}^k&{k\choose 1}\lambda_{i}^{k-1}&\dots&{k\choose m_i-2}\lambda_{i}^{k-m_i+2}\\
\vdots&&\ddots&&\vdots\\
0&&&{k\choose 0}\lambda_{i}^k&{k\choose 1}\lambda_{i}^{k-1}\\
0&0&\dots&0&{k\choose 0}\lambda_{i}^k\\
\end{array}
\right]
$$
for any $k\geq m_i-1$. Hence, for $k\geq m_i-1$ the matrix norm
\begin{align*}
|||J_{m_i}^k(\lambda_i)|||_{\infty}&=\sum_{j=0}^{m_i-1}\Big|{k\choose j}\lambda_i^{k-j}\Big|\\
&=\Big[\big|{k\choose 0}\lambda_{i}^{m_i-1}\big|+\big|{k\choose 1}\lambda_{i}^{m_i-2}\big|+\dots+\big|{k\choose m_i-1}\big|\Big]\big|\lambda_i^{k-m_i+1}\big|\\
&\leq C_{m_i}\rho(\mathcal{M}_{\mathcal{F}})^{k-m_i+1}
\end{align*}
where $C_{m_i}=\sum_{j=0}^{m_i-1}\big|{k\choose j}\lambda_{i}^{m_i-j-1}\big|$.

Suppose $\rho(\mathcal{M}_{\mathcal{F}})<1$. Letting $\displaystyle{m=\max_{0\leq i\leq t}m_i}$ and $\displaystyle{C=\max_{1\leq i\leq t}C_{m_i}}$ then

$$|||J^k|||_{\infty}\leq C\rho(\mathcal{M}_{\mathcal{F}})^{k-m+1} \ \ \text{for all} \ \ k\geq m-1.$$

Via equation (\ref{theoremeq}),
\begin{align*}
\sum_{k=m}^{\infty}d_{max}\big(\mathcal{F}^k(\textbf{x}),\mathcal{F}^k(\textbf{y})\big)\leq
\sum_{k=m}^\infty\Big|\Big|M^k_\mathcal{F}
\left[\begin{array}{c}
d(x_1,y_1)\\
\vdots\\
d(x_n,y_n)
\end{array}\right]&
\Big|\Big|_{\infty}\\
\leq\sum_{k=m}^{\infty}\Big|\Big|\Big|AJ^kA^{-1}\Big|\Big|\Big|_{\infty}\Big|\Big|
\left[\begin{array}{c}
d(x_1,y_1)\\
\vdots\\
d(x_n,y_n)
\end{array}\right]&
\Big|\Big|_{\infty}\\
\leq\Big|\Big|\Big|A\Big|\Big|\Big|_{\infty}\Big|\Big|\Big|A^{-1}\Big|\Big|\Big|_{\infty}
\Big|\Big|
\left[\begin{array}{c}
d(x_1,y_1)\\
\vdots\\
d(x_n,y_n)
\end{array}\right]&
\Big|\Big|_{\infty}\sum_{k=m}^{\infty}\Big|\Big|\Big|J^k\Big|\Big|\Big|_{\infty}\\
\leq\Big|\Big|\Big|A\Big|\Big|\Big|_{\infty}\Big|\Big|\Big|A^{-1}\Big|\Big|\Big|_{\infty}
\Big|\Big|
\left[\begin{array}{c}
d(x_1,y_1)\\
\vdots\\
d(x_n,y_n)
\end{array}\right]&
\Big|\Big|_{\infty}\frac{C\rho(\mathcal{M}_{\mathcal{F}})}{1-\rho(\mathcal{M}_{\mathcal{F}})}<\infty.
\end{align*}
Letting $\mathbf{y}=F(\mathbf{x})$ it then follows that $\sum_{k=m}^{\infty}d_{max}\big(\mathcal{F}^k(\textbf{x}),\mathcal{F}^{k+1}(\textbf{x})\big)<\infty$. Hence, for $\epsilon>0$ there exists an $N$ such that $\sum_{k=N}^{\infty}d_{max}\big(\mathcal{F}^k(\textbf{x}),\mathcal{F}^{k+1}(\textbf{x})\big)<\epsilon$. Therefore,
$$d_{max}\big(F^k(\mathbf{x}),F^{\ell}(\mathbf{x})\big)<\epsilon$$
for any $k,\ell>N$ or the sequence $\{F^k(\mathbf{x})\}_{k\geq 1}$ is Cauchy. Since $X$ is complete then this sequence converges.

Suppose $\mathcal{F}^k(\textbf{x})\rightarrow \tilde{\textbf{x}}$. As $\mathcal{F}$ is continuous it follows that $\mathcal{F}(\tilde{\textbf{x}})=\tilde{\textbf{x}}$. Moreover, given that $\sum_{k=1}^{\infty}d_{max}\big(\mathcal{F}^k(\tilde{\textbf{x}}),\mathcal{F}^k(\textbf{y})\big)\leq\infty$ for any $\textbf{y}\in X$ then $\mathcal{F}^k(\textbf{y})\rightarrow\tilde{\textbf{x}}$ completing the proof.
\end{proof}

\begin{example}\label{ex2}
Let $T_i:[0,1]\rightarrow [0,1]$ be the logistic map $T_i(x_i)=4x_i(1-x_i)$ for $1\leq i\leq 3$. Let $F:[0,1]^3\rightarrow[0,1]^3$ be the interaction given by
$$F(\mathbf{x})=\left[
\begin{array}{c}
1-x_1x_2/9\\
1-x_1x_3/9\\
1-x_2x_3/9
\end{array}
\right].
$$
One can check that the constants $L_i=4$ and
$$\Lambda_{ij}=
\begin{cases}
2/9 \ &\text{if} \ \ i+j\neq 4\\
\ 0 \ &\text{otherwise}
\end{cases}
$$
satisfy (\ref{eq0.1}) and (\ref{eq2.3}) respectively. In particular, $\Lambda_{ij}=\max_{\textbf{x}\in X}|(DF)_{ji}(\textbf{x})|$. For these constants the matrix
$$M_{\mathcal{F}}=\left[
\begin{array}{ccc}
4/9&4/9&0\\
4/9&0&4/9\\
0&4/9&4/9
\end{array}
\right].
$$
As $\rho(M_\mathcal{F})=8/9<1$ then the dynamical network $(\mathcal{F},X)$ has a globally attracting fixed point.
\end{example}

The local systems $(T,X)$ are said to be \textit{stable} if $\max_{j\in\mathcal{I}}\{L_i\}<1$ and are said to be \textit{unstable} otherwise. We say the interaction $F:X\rightarrow X$ \textit{stabilizes} the local systems $(T,X)$ if the local systems are unstable but the dynamical network $(\mathcal{F},X)$ has a globally attracting fixed point. If the local systems $(T,X)$ are stable and $(\mathcal{F},X)$ has a globally attracting fixed point we say the interaction $F:X\rightarrow X$ \textit{maintains} the stability of $(T,X)$. The following is a corollary to theorem \ref{stability}.

\begin{corollary}\label{col2}
For the dynamical network $(\mathcal{F},X)$ let $\max_{j\in\mathcal{I}}\{L_i\}=L$. If $L\rho(\Lambda)<1$ then the interaction $F$ stabilizes (or maintains the stability of) the local systems $(T,X)$.
\end{corollary}

Before proving this we require the following. For $A,B\in\mathbb{R}^{n\times n}$ we write
\begin{align*}
A\leq& B \ \ \text{if} \ \ A_{ij}\leq B_{ij} \ \ \text{for} \ \ 1\leq i,j\leq n.
\end{align*}
If $0\leq A\leq B$ then a well known result in matrix theory states that $\rho(A)\leq\rho(B)$ (see chapter 8, \cite{Horn85}). We now give a proof of corollary \ref{col2}.

\begin{proof}
As $0\leq\Lambda^T\cdot diag[L_1,\dots,L_n]\leq L\Lambda^T$ then $\rho(M_\mathcal{F})\leq \rho(L\Lambda^T)$. Given that $L\Lambda^T=\{L\lambda:\lambda\in\sigma(\Lambda^T)\}$ then this implies $\rho(L\Lambda^T)= L\rho(\Lambda)$ which in turn implies that $\rho(M_\mathcal{F})\leq L\rho(\Lambda)$. Hence, if $L\rho(\Lambda)<1$ then by theorem \ref{stability} the interaction $F$ stabilizes or maintains the stability of the local systems $(T,X)$.
\end{proof}

For the dynamical network $\mathcal{F}=F\circ T$ in example \ref{ex2} note that $L=4$ and $\rho(\Lambda)=2/9$. As $L\rho(\Lambda)<8/9$ then corollary \ref{col2} implies that the interaction $F$ stabilizes the local systems $(T,X)$.

\subsection{Dynamical Networks Without Local Dynamics}

As defined in section 2.1 a dynamical network $\mathcal{F}=F\circ T$ is the composition of the network's local dynamics $T$ and the interaction $F$. However, if the system has no local dynamics, i.e. $T=id$ is the identity map, then the dynamical network $\mathcal{F}$ is simply the interaction $F$.

\begin{remark}
Note that any composition $\mathcal{F}=F\circ T$ can be considered to be an interaction. Writing $\mathcal{F}=(F\circ T)\circ id$ the dynamical network $(\mathcal{F},X)$ is simply the interaction $F\circ T$ with no local dynamics.
\end{remark}

The following is another corollary of theorem \ref{stability}.

\begin{corollary}\label{cor-1}
Suppose the constants $L_i$ and $\Lambda_{ij}$ satisfy (\ref{eq0.1}) and (\ref{eq2.3}) for the map $\mathcal{F}=F\circ T$. Let $\tilde{\mathcal{F}}=(F\circ T)\circ id$ be the map inducing the dynamical network $(\tilde{F},X)$ with interaction $F\circ T$ and no local dynamics. Then there exist constants $\tilde{L}_i$ and $\tilde{\Lambda}_{ij}$ satisfying (\ref{eq0.1}) and (\ref{eq2.3}) for $id$ and $F\circ T$ respectively such that $\rho(M_{\tilde{\mathcal{F}}})\leq\rho(M_{\mathcal{F}})$.
\end{corollary}

\begin{proof}
Let $L_i$ and $\Lambda_{ij}$ be as in the statement of corollary \ref{cor-1}. Suppose $\mathbf{x},\mathbf{y}\in X$. Then $$d\big(id(\mathbf{x})_i,id(\mathbf{y})_i\big)=d(x_i,y_i).$$ Hence, the constants $\tilde{L}_i=1$ satisfy (\ref{eq0.1}) for the local systems $(id,X)$. Moreover,

\begin{align*}
d\big((F\circ T)(\mathbf{x})_j,(F\circ T)(\mathbf{y})_j\big)=&d\big(F_j(\{T_i(x_i)\}),F_j(\{T_i(y_i)\})\big)\\
                                            \leq&\sum_{i\in\mathcal{I}}\Lambda_{ij}d\big(T_i(x_i),T_i(y_i)\big)\\
                                            \leq&\sum_{i\in\mathcal{I}}\Lambda_{ij}L_i d(x_i,y_i).
\end{align*}

Therefore, the constants $\tilde{\Lambda}_{ij}=\Lambda_{ij}L_i$ satisfy equation (\ref{eq2.3}) for the interaction $F\circ T$. For this choice of constants note that the matrix $\mathcal{M}_\mathcal{\tilde{F}}$ is equal to the matrix $\mathcal{M}_\mathcal{F}$. This implies the result of the corollary.
\end{proof}

\begin{example}\label{local}
Let $T_i:[0,1]\rightarrow [0,1]$ be the map $T_i(x)=\sin(\pi x_i)$ for $1\leq i\leq 2$. Suppose $F:[0,1]^2\rightarrow[0,1]^2$ is the interaction given by
$$F(\mathbf{x})=\left[
\begin{array}{c}
1-x_2^2/4\\
1-x_1^2/4
\end{array}
\right].
$$

Note that the constants $L_i=\pi$ and $\Lambda_{ij}=\max_{\textbf{x}\in X}|(DF)_{ji}(\textbf{x})|$ satisfy (\ref{eq0.1}) and (\ref{eq2.3}) for $T$ and $F$ respectively. If the dynamical network $\mathcal{F}=F\circ T$ then
$$M_\mathcal{F}=\left[
\begin{array}{cc}
0&\pi/2\\
\pi/2&0
\end{array}
\right].
$$
As $\rho(M_\mathcal{F})=\pi/2>1$ then theorem \ref{stability} does not directly apply to the dynamical network $\mathcal{F}$.

However, suppose $\tilde{\mathcal{F}}=(F\circ T)\circ id$ is the dynamical network with interaction $F\circ T$ and no local dynamics. Let $L_i=1$ and $\tilde{\Lambda}_{ij}=\max_{\textbf{x}\in X}|\big(D(F\circ T)\big)_{ji}(\textbf{x})|$ which satisfy (\ref{eq0.1}) and (\ref{eq2.3}) for $T=id$ and $F\circ T$ respectively. With this choice of constants
$$M_{\tilde{\mathcal{F}}}=\left[
\begin{array}{cc}
0&\pi/4\\
\pi/4&0
\end{array}
\right].
$$
As $\rho(M_{\tilde{\mathcal{F}}})=\pi/4<1$ then the dynamical network $(\tilde{\mathcal{F}},X)$ has a globally attracting fixed point. Importantly, given that $\mathcal{F}=\tilde{\mathcal{F}}$ then the dynamical $(\mathcal{F},X)$ must also have a globally attracting fixed point.
\end{example}

Note that if $L_i=1$ for all $i$ then $\Lambda\cdot diag[L_1,\dots,L_n]=\Lambda$. This observation leads to the following remark.

\begin{remark}
If $(\mathcal{F},X)$ is considered as a dynamical network with no local dynamics then $\mathcal{M}_\mathcal{F}=\Lambda$ for any constants $\Lambda_{ij}$ satisfying (\ref{eq2.3}). Hence, $(\mathcal{F},X)$ has a globally attracting fixed point if $\rho(\Lambda)<1$.
\end{remark}

By considering a dynamical network in a modified form, i.e. as an interaction, it is therefore possible to obtain improved estimates of the network's global stability. The result of corollary \ref{cor-1} is then the first step toward answering the following general question: How can a dynamical network (equivalently interaction) be modified to get improved estimates of a network's global stability?

Additionally, in this paper we consider whether it is possible to transform (simplify) a dynamical network at the level of its graph of interactions while maintaining the dynamic properties of the network. Both questions serve to motivate the graph transformations introduced in the following sections.

\section{Isospectral Graph Reductions}

In this section we formally describe the isospectral reduction process of a graph. We then give specific examples of this method and present some results regarding sequences of isospectral reductions and spectral equivalence of graphs.

\subsection{The class of graphs $\mathbb{G}$}

We consider the following class of graphs, namely those graphs that are weighted, directed, with edge weights in the set $\mathbb{W}[\lambda]$ (defined below). Such graphs form the class $\mathbb{G}$. A graph $G\in\mathbb{G}$ is then an ordered triple $G=(V,E,\omega)$ where $V$ and $E$ are again the \textit{vertex set} and \textit{edge set} of $G$ respectively. The function $\omega:E\rightarrow\mathbb{W}[\lambda]$ gives the \textit{edge weights} of $G$. We adopt the standard convention that each edge of the weighted graph $G$ has a nonzero weight.

\begin{remark}
The class of graphs $\mathbb{G}$ is very general as it contains, for instance, the class of undirected graphs with numerical weighs.
\end{remark}

For notational convenience we assume that each vertex set $V$ is given the labelling $V=\{v_1,\dots,v_n\}$. Also, any graph $G$ written as an ordered triple $G=(V,E,\omega)$ will be implicitly assumed to be a graph in $\mathbb{G}$.

Let $\mathbb{C}[\lambda]$ be the set of polynomials in the complex variable $\lambda$ with complex coefficients i.e. the polynomial ring over $\mathbb{C}$ with indeterminate $\lambda$. We denote by $\mathbb{W}[\lambda]$ the field of fractions of $\mathbb{C}[\lambda]$ or the set of rational functions of the form $p/q$ where $p,q\in\mathbb{C}[\lambda]$, and $q\neq 0$. Our first task is to extend the definition of the eigenvalues of a matrix $A\in\mathbb{C}^{n\times n}$ to the eigenvalues of a matrix $A(\lambda)\in\mathbb{W}[\lambda]^{n\times n}$.

The element $\alpha$ of the set $A$ that includes multiplicities has \textit{multiplicity} $m$ if there are $m$ elements of $A$ equal to $\alpha$. Suppose $A$ and $B$ are sets that include multiplicities. If $\alpha\in A$ with multiplicity $m$ and $\alpha\in B$ with multiplicity $n$ then\\
(i) the \textit{union} $A\cup B$ is the set in which $\alpha$ has multiplicity $m+n$; and\\
(ii)the \textit{difference} $A-B$ is the set in which $\alpha$ has multiplicity $m-n$ if $m-n>0$ and multiplicity $0$ otherwise.

If a square matrix $A\in\mathbb{C}^{n\times n}$ then its \textit{characteristic polynomial} $\det\big(A-\lambda I\big)\in\mathbb{C}[\lambda]$. However, if $A(\lambda)\in\mathbb{W}[\lambda]^{n\times n}$ then
\begin{equation}\label{char}
\det\big(A(\lambda)-\lambda I\big)\in\mathbb{W}[\lambda].
\end{equation}
i.e. this determinant may be a rational function of $\lambda$. Despite this we will adopt the convention of referring to (\ref{char}) as the characteristic polynomial of $A(\lambda)$.

\begin{definition}\label{def1.1}
For a matrix $A(\lambda)\in\mathbb{W}[\lambda]^{n\times n}$ suppose the characteristic polynomial $\det(A(\lambda)-\lambda I)=p/q$ where $p,q\in\mathbb{C}[\lambda]$. Define the sets
$$P=\{\lambda\in\mathbb{C}:p=0\} \ \ \text{and} \ \ Q=\{\lambda\in\mathbb{C}:q=0\}$$
where these sets includes multiplicities. We call the sets $$\sigma\big(A(\lambda)\big)=P-Q \ \text{and} \ \sigma^{-1}\big(A(\lambda)\big)=Q-P$$ the \textit{spectrum} (or set of \textit{eigenvalues}) of $A(\lambda)$ and the \textit{inverse spectrum} of $A(\lambda)$ respectively.
\end{definition}

Of primary importance is the fact that the representation of $p/q\in\mathbb{W}[\lambda]$ is not unique. That is, $p/q$ is equivalent to $r/s$ for $p,q,r,s\in\mathbb{C}[\lambda]$ if $ps=qr$. It is therefore necessary to show that the spectrum and inverse spectrum of a matrix $A(\lambda)\in\mathbb{W}[\lambda]^{n\times n}$ are well defined.

That is, $p/q$ is equivalent to $r/s$ in $\mathbb{W}^{n\times n}$ for $p,q,r,s\in\mathbb{C}[\lambda]$. Let $P$ and $Q$ be defined as in definition \ref{def1.1}. Similarly define
$$R=\{\lambda\in\mathbb{C}:r=0\} \ \ \text{and} \ \ S=\{\lambda\in\mathbb{C}:s=0\}.$$
As $ps=qr$ it follows that $P\cup S=Q\cup R$. Moreover, as $P\cup S-Q=Q\cup R-Q=R$ then $(P\cup Q)-S=R-S$.

From the fact that $(P\cup S-Q)-S=(P\cup S-S)-Q$ it follows that $P-Q=R-S$. Similarly, one can show that $Q-P=S-R$. That is, the spectrum and inverse spectrum given by definition \ref{def1.1} are well defined.

Before proceeding we note the following. According to definition \ref{def1.1} any matrix $A\in\mathbb{C}^{n\times n}$ has a well defined spectrum $\sigma(A)$ and inverse spectrum $\sigma^{-1}(A)$. In particular, if $A\in\mathbb{C}^{n\times n}$ the characteristic polynomial of $A$ is the ratio $p/q\in\mathbb{W}[\lambda]$ where $p=\det(A-\lambda I)$ and $q=1$.
Hence, $$\sigma(A)=\{\lambda\in\mathbb{C}:\det(A-\lambda I)=0\}, \ \ \sigma^{-1}(A)=\emptyset.$$ This can be summarized as follows.

\begin{remark}
The eigenvalues of a square matrix with entries in $\mathbb{W}[\lambda]$ is a generalization of the standard notion of the eigenvalues of a complex valued square matrix.
\end{remark}

Suppose $G=(V,E,\omega)$ with vertex set $V=\{v_1,\dots,v_n\}$. The matrix $M(G)\in\mathbb{W}[\lambda]^{n\times n}$ defined entrywise by $$M(G)_{ij}=\omega(e_{ij})$$
is called the \textit{weighted adjacency matrix} of $G$. For the graph $G$ we denote by $\sigma(G)$ and $\sigma^{-1}(G)$ the spectrum of and inverse spectrum of $M(G)$ respectively.

\subsection{Isospectral Graph Reductions}

Our next task is to develop the mathematical framework, i.e. definitions and notation, needed to describe an isospectral reduction of a graph. This requires the following standard terminology.

A \textit{path} $P$ in  the graph $G=(V,E,\omega)$ is an ordered sequence of distinct vertices $v_1,\dots,v_m\in V$ such that $e_{i,i+1}\in E$ for $1\leq i\leq m-1$. We call the vertices $v_2,\dots,v_{m-1}$ of $P$ the \textit{interior} vertices of $P$. If the vetices $v_1$ and $v_m$ are the same then $P$ is a \textit{cycle}. A cycle $v_1\dots,v_m$ is called a \textit{loop} if $m=1$. Note that as $v_i,v_i$ is a loop of $G$ if and only if $e_{ii}\in E$ we may refer to the edge $e_{ii}$ as the loop. If $S\subseteq V$ where $V$ is the vertex set of a graph we will write $\bar{S}=V-S$.

The main idea behind an isospectral reduction of a graph $G=(V,E,\omega)$ is that we reduce $G$ to a smaller graph on some subset $S\subset V$. The sets $S$ for which this is possible are defined as follows.

\begin{definition}\label{def1}
Let $G=(V,E,\omega)$. A nonempty vertex set $S\subseteq V$ is a \textit{structural set} of $G$ if\\
(i) each cycle of $G$, that is not a loop, contains a vertex in $S$; and\\
(ii) $\omega(e_{ii})\neq \lambda$ for each $v_i\in \bar{S}$.
\end{definition}

Observe that assumption (i) of definition \ref{def1} implies that a structural set $S$ of $G$ depends intrinsically on the structure of $G$. On the other hand, part (ii) of definition \ref{def1} is the formal assumption that the loops of the vertices in $\bar{S}$, i.e. the complement of $S$, do not have weight equal to $\lambda\in\mathbb{W}[\lambda]$. For $G\in\mathbb{G}$ let $st(G)$ denote the set of all structural sets of the graph $G$.

\begin{definition}
Suppose $G=(V,E,\omega)$ with structural set $S=\{v_1,\dots,v_m\}$. Let $\mathcal{B}_{ij}(G;S)$ be the set of paths or cycles from $v_i$ to $v_j$ with no interior vertices in $S$. We call a path or cycle $\beta\in\mathcal{B}_{ij}(G;S)$ a \textit{branch} of $G$ with respect to $S$. We let
$$\mathcal{B}_S(G)=\bigcup_{1\leq i,j \leq m} \mathcal{B}_{ij}(G;S)$$
denote the set of all branches of $G$ with respect to $S$.
\end{definition}

If $\beta=v_1,\dots,v_m$ is a branch of $G$ with respect to $S$, i.e. $\beta\in\mathcal{B}_S(G)$, and $m>2$ then let
\begin{equation}\label{eq0.9}
\mathcal{P}_{\omega}(\beta)=\omega(e_{12})\prod_{i=2}^{m-1}\frac{\omega(e_{i,i+1})}{\lambda-\omega(e_{ii})}.
\end{equation}
For $m=1,2$ let $\mathcal{P}_{\omega}(\beta)=\omega(e_{1m})$. We call $\mathcal{P}_{\omega}(\beta)$ the \textit{branch product} of $\beta$. Note that assumption (ii) in definition \ref{def1} implies that the branch product of any $\beta\in\mathcal{B}_S(G)$ is always defined.

In a procedure we term an \textit{isospectral graph reduction} we replace the branches of a graph with edges. The following definition specifies the weights of these edges.

\begin{definition}\label{IR}
Let $G=(V,E,\omega)$ with structural set $S=\{v_1\,\dots,v_m\}$. Define the edge weights
\begin{equation}\label{eq1.0}
\mu(e_{ij})=\begin{cases}
\displaystyle{\sum_{\beta\in\mathcal{B}_{ij}(G;S)}\mathcal{P}_\omega(\beta)} & \text{if} \ \ \ \mathcal{B}_{ij}(G;S)\neq\emptyset\\
\ \ \ \ \ 0 & \text{otherwise}
            \end{cases} \ \ \ \text{for} \ \ \ 1\leq i,j\leq m.
\end{equation}
The graph $\mathcal{R}_S(G)=(S,\mathcal{E},\mu)$ where $e_{ij}\in \mathcal{E}$ if $\mu(e_{ij})\neq 0$
is the \textit{isospectral reduction} of $G$ over $S$.
\end{definition}

Observe that $\mu(e_{ij})$ in definition \ref{IR} is the weight of the edge $e_{ij}$ in $\mathcal{R}_S(G)$. Moreover, as $W[\lambda]$ is closed under both addition and multiplication then the edge weights $\mu(e_{ij})$ of $\mathcal{R}_S(G)$ are also in the set $W[\lambda]$. Hence, the isospectral reduction $\mathcal{R}_S(G)$ is again a graph in $\mathbb{G}$.

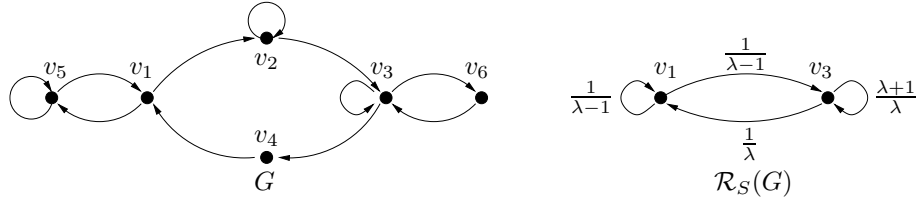
\begin{figure}
  \begin{center}
    \begin{overpic}[scale=.5]{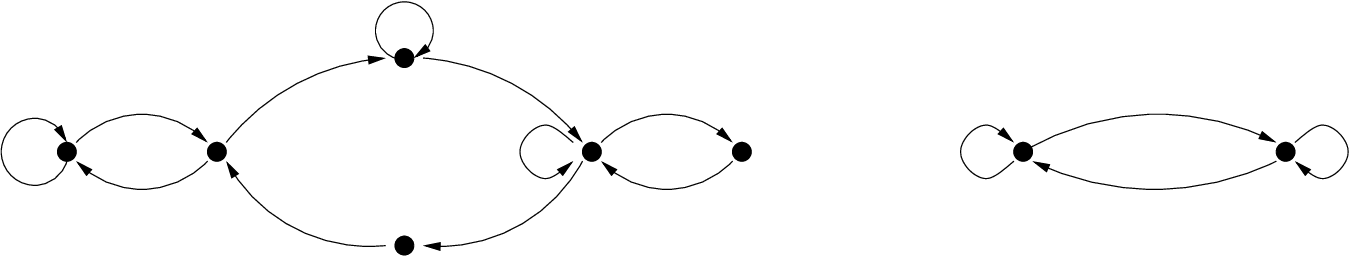}
    \put(28.5,-3){$G$}
    \put(4,10.5){$v_5$}
    \put(14,10.5){$v_1$}
    \put(75,10.5){$v_1$}
    \put(28.5,11.5){$v_2$}
    \put(28.5,2.5){$v_4$}
    \put(42,10.5){$v_3$}
    \put(93,10.5){$v_3$}
    \put(53,10.5){$v_6$}
    \put(83,12){{$\frac{1}{\lambda-1}$}}
    \put(85,1.5){$\frac{1}{\lambda}$}
    \put(65,7){$\frac{1}{\lambda-1}$}
    \put(100.5,7){$\frac{\lambda+1}{\lambda}$}
    \put(82,-3){$\mathcal{R}_S(G)$}
    \end{overpic}
  \end{center}
  \caption{Reduction of $G$ over $S=\{v_1,v_3\}$ where each edge in $G$ has unit weight.}\label{fig2}
\end{figure}

\begin{example}\label{ex3}
Consider the graph $G=(V,E,\omega)$ given in figure \ref{fig2} (left) in which each edge of $G$ is assumed to have unit weight. The vertex set $S=\{v_1,v_3\}\subset V$ is a structural set of of $G$ as\\
(i) the three nonloop cycles of $G$, namely $v_1,v_2,v_3,v_4,v_1$; $v_1,v_5,v_1$; and $v_3,v_6,v_3$ each contain a vertex in $S$; and\\
(ii) the loop weights vertices in $\bar{S}=\{v_2,v_4,v_5,v_6\}$ are $\omega(e_{22})=1$, $\omega(e_{44})=0,$ $\omega(e_{55})=1$, and $\omega(e_{66})=0$. That is, $\omega(e_{ii})\neq\lambda$ for each $v_i\in\bar{S}$.

Conversely, the vertex set $T=\{v_1,v_5\}$ is not a structural set of $G$.This follows from the fact that the (nonloop) cycle $v_3,v_6,v_3$ does not contain a vertex of $T$.

The branches of $G$ with respect to $S$ can be seen to be $\mathcal{B}_{11}(G;S)=\{v_1,v_5,v_1\}$, $\mathcal{B}_{13}(G;S)=\{v_1,v_2,v_3\}$, $\mathcal{B}_{31}(G;S)=\{v_3,v_4,v_1\}$, and the two branches in the set $\mathcal{B}_{33}(G;S)=\{v_3,v_3; v_3,v_6,v_3\}$.

Using equation (\ref{eq0.9}) the branch product of each of these branches are given by
$$\mathcal{P}_\omega(v_1,v_5,v_1)=\mathcal{P}_\omega(v_1,v_2,v_3)=\frac{1}{\lambda-1},$$
$$\mathcal{P}_\omega(v_3,v_4,v_1)=\mathcal{P}_\omega(v_3,v_6,v_3)=\frac{1}{\lambda}, \ \text{and} \ \ \mathcal{P}_\omega(v_3,v_3)=1.$$

The weight of each edge of $\mathcal{R}_S(G)=(S,\mathcal{E},\mu)$ given by (\ref{eq1.0}) can then be computed to be
$$\mu(e_{11})=\mu(e_{13})=\frac{1}{\lambda-1},$$
$$\mu(e_{31})=\frac{1}{\lambda}, \ \text{and} \ \ \mu(e_{33})=1+\frac{1}{\lambda}=\frac{\lambda+1}{\lambda}.$$
Note that as each edge weight is nonzero then the edge set of $\mathcal{R}_S(G)$ is given by $\mathcal{E}=\{e_{11}$, $e_{13}$, $e_{31}$, $e_{33}\}$. The graph $\mathcal{R}_S(G)$ is shown in figure \ref{fig2} (right).
\end{example}

Recall that if $S$ is a structural set of the graph $G\in\mathbb{G}$ then the isospectral reduction $\mathcal{R}_S(G)$ is also a graph in $\mathbb{G}$. Hence, both $G$ and $\mathcal{R}_S(G)$ have well-defined spectra. The relation between the spectrum $\sigma(G)$ and $\sigma(\mathcal{R}_S(G))$ is given in the following theorem.

\begin{theorem}\label{maintheorem}
Let $S$ be a structural set of the graph $G\in\mathbb{G}$. Then $$\sigma\big(\mathcal{R}_S(G)\big)=\big(\sigma(G)\cup\sigma^{-1}(G|_{\bar{S}})\big)-\sigma(G|_{\bar{S}}).$$
\end{theorem}

The proof of theorem \ref{maintheorem} will be postponed until section 5. However, we note here that the characteristic polynomial of the matrix $M(G|_{\bar{S}})$ has a particularly simple form.

Since the only cycles the graph $G|_{\bar{S}}$ contains are loops then the following holds. The vertices of $G|_{\bar{S}}$ can be ordered such that the matrix $M(G|_{\bar{S}})$ is \textit{triangular} i.e. either all the entries below or all the entries above the main diagonal of $M(G|_{\bar{S}})$ are zero (see \textit{Frobenius normal form} \cite{Brualdi91}). Hence, the matrix $M(G|_{\bar{S}})-\lambda I$ is also triangular.

As the determinant of a triangular matrix is the product of its diagonal entries then
\begin{equation}\label{eq0.2}
 \det\big(M(G|_{\bar{S}})-\lambda I\big)=\prod_{v_i\in \bar{S}}\big(\omega(e_{ii})-\lambda\big).
\end{equation}
Because of the simplicity of equation (\ref{eq0.2}) both the spectrum $\sigma(G|_{\bar{S}})$ and inverse spectrum $\sigma^{-1}(G|_{\bar{S}})$ can be calculated with relative ease. Moreover, in light of equation (\ref{eq0.2}) theorem \ref{maintheorem} has the following corollary.

\begin{corollary}\label{cor1}
Let $S$ be a structural set of the graph $G\in\mathbb{G}$. If $M(G)\in\mathbb{C}^{n\times n}$ then\\
(i) $\sigma(G|_{\bar{S}})=\{\omega(e_{ii}): v_i\in \bar{S}\}$;\\
(ii) $\sigma^{-1}(G|_{\bar{S}})=\emptyset$; and\\
(iii) $\sigma(\mathcal{R}_S(G))=\sigma(G)-\sigma(G|_{\bar{S}})$.
\end{corollary}

\begin{figure}
  \begin{center}
    \begin{overpic}[scale=.5]{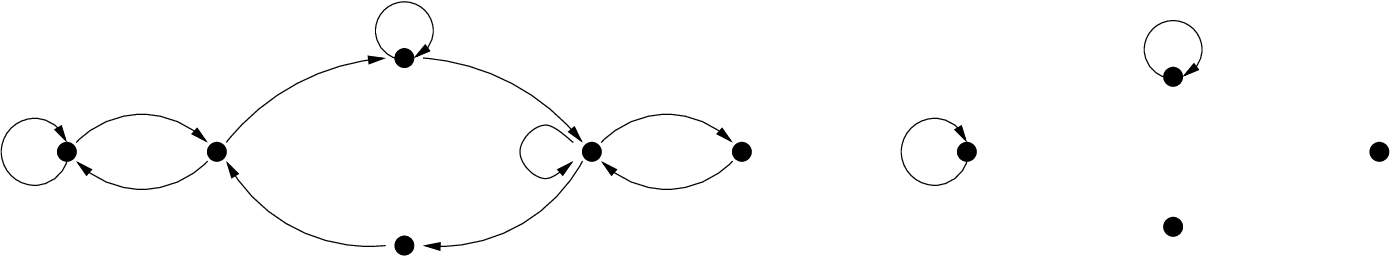}
    \put(28,-3){$G$}
    \put(4,10.5){$v_5$}
    \put(14,10.5){$v_1$}
    \put(71,8){$v_5$}
    \put(28.5,11.5){$v_2$}
    \put(28.5,2.5){$v_4$}
    \put(41,10.5){$v_3$}
    \put(95,8){$v_6$}
    \put(51,10.5){$v_6$}
    \put(83,10){$v_2$}
    \put(83,4){$v_4$}
    \put(82,-2.5){$G|_{\bar{S}}$}
    \end{overpic}
  \end{center}
  \caption{Restriction of the graph $G$ to $\bar{S}=\{v_2,v_4,v_5,v_6\}$.}\label{fig3}
\end{figure}
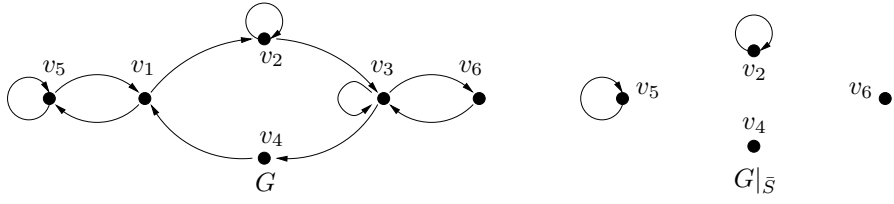

In applications the graphs (matrices) that are used often have real or positive weights (entries). If $G=(V,E,\omega)$ has complex valued weights and $S\in st(G)$ then corollary \ref{cor1} states the following. The spectrum of $\mathcal{R}_S(G)$ and $G$ differ at most by the spectrum of $G|_{\bar{S}}$. Moreover, the spectrum $\sigma(G|_{\bar{S}})$ can be described as follows.

For $v_i\in \bar{S}$ the weight $\omega(e_{ii})\in\mathbb{C}$ is the weight of the loop $e_{ii}$ if $e_{ii}$ is an edge of the graph $G$. In the case that $e_{ii}\notin E$ then $\omega(e_{ii})=0$. That is, $\sigma(G|_{\bar{S}})$ consists of all the weights of  the loops off the structural set $S$. This includes the weight(s) $0$ for all $v_i\in\bar{S}$ such that $e_{ii}\notin E$. Because of this, the set $\sigma(G|_{\bar{S}})$ can be easily read off the graph $G$. As an illustration we give the following example.

\begin{example}
Let $G$ be the graph considered in example \ref{ex3}. As previously shown the vertex set $S=\{v_1,v_3\}$ is a structural set of $G$. Moreover, $M(G)\in\mathbb{C}^{6\times 6}$. Hence, corollary \ref{cor1} allows us to quickly compute the eigenvalues of the reduced graph $\mathcal{R}_S(G)$ once the eigenvalues of $G$ are known.

As one can calculate, the eigenvalues of the graph $G$ are $\sigma(G)=\{2,-1,1,1,0,0\}$. The restriced graph $G|_{\bar{S}}$ shown in figure \ref{fig3} has loop weights $\omega(e_{22})=1$, $\omega(e_{44})=0$, $\omega(e_{55})=1$, and $\omega(e_{66})=0$. Then corollary \ref{cor1} implies that $\sigma(G|_{\bar{S}})=\{1,0,1,0\}$ (and $\sigma^{-1}(G|_{\bar{S}})=\emptyset$). Additionally, as $\sigma(\mathcal{R}_S(G))=\sigma(G)-\sigma(G|_{\bar{S}})$ then the spectrum of the reduced graph is $\sigma\big(\mathcal{R}_S(G)\big)=\{2,-1\}$.

That is, in the reduction process of $G$ over $S$ we lose the eigenvalues $\{0,0,1,1\}$. However, even if $\sigma(G)$ is unknown we still know the following. The set of eigenvalues $\sigma(G|_{\bar{S}})=\{0,0,1,1\}$ is the most by which $\sigma(\mathcal{R}_S(G))$ and $\sigma(G)$ can differ.
\end{example}

\begin{remark}
We note that it is possible when reducing a graph to maintain the graph's entire spectrum including multiplicities.
\end{remark}

\begin{example}\label{ex4}
Let $K_n=(V,E,\omega)$ be the complete graph on $n$ vertices with unit edge weights. That is, $V=\{v_1,\dots,v_n\}$ and
$$\omega(e_{ij})=
\begin{cases}
1 &i\neq j\\
0 &\text{otherwise}
\end{cases} \ \ \text{for} \ \ 1\leq i,j \leq n.
$$ (See figure \ref{fig4}, left for $n=6$.)

For any fixed $v_k\in V$ note that $V-\{v_k\}\in st(K_n)$. This follows from the fact that condition (i) of definition \ref{def1} is trivially satisfied by the vertex set $V-\{v_k\}$. Condition (ii) follows from the fact that $\omega(e_{kk})=0\neq\lambda$. Therefore, the reduced graph $\mathcal{R}_{V-\{v_k\}}(K_n)=(V-\{v_k\},\mathcal{E},\mu)$ is well defined having edges $e_{ij}$ with edges weights
$$\mu(e_{ij})=
\begin{cases}
1+1/\lambda &i,j\neq k\\
0 &\text{otherwise}
\end{cases} \ \ \text{for} \ \ 1\leq i,j \leq n.
$$ (See figure \ref{fig4}, right for $n=6$.)

To compute the spectrum $\sigma(K_n)$ note the following. As the matrix $M(K_n)+I$ is the $n\times n$ matrix of ones then $\sigma(M(K_n)+I)=\{n,0,\dots,0\}$ where $0$ has multiplicity $n-1$. Therefore, the graph $K_n$ has spectrum $\sigma(K_n)=\{n-1,-1,\dots,-1\}$ where $-1$ has multiplicity $n-1$. In particular, $0\notin\sigma(K_n)$.

Via equation (\ref{eq0.2}) the sets $\sigma(K_n|_{\{v_k\}})=\{0\}$ and $\sigma^{-1}(K_n|_{\{v_k\}})=\emptyset$ since the weight $\omega(e_{kk})=0$. As $\sigma(K_n|_{\{v_k\}})\cap\sigma(K_n)=\emptyset$ then theorem \ref{maintheorem} (or corollary \ref{cor1}) implies $\sigma(K_n)=\sigma(\mathcal{R}_{V-\{v_k\}}(K_n))$. That is, no eigenvalues are lost in reducing the graph $K_n$ to the graph $\mathcal{R}_{V-\{v_k\}}(K_n)$.
\end{example}

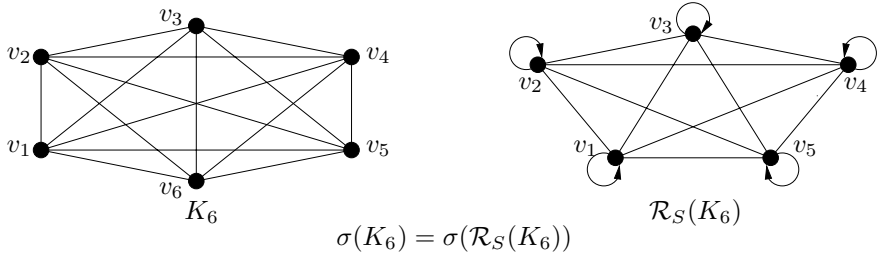
\begin{figure}
  \begin{center}
    \begin{overpic}[scale=.65]{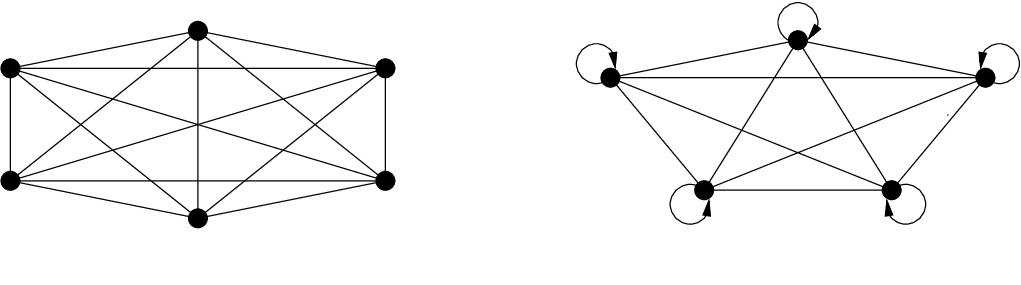}
    \put(18,2){$K_6$}
    \put(73,2){$\mathcal{R}_S(K_6)$}
    \put(-3,10){$v_1$}
    \put(64,10){$v_1$}
    \put(-3,21){$v_2$}
    \put(57.5,17.5){$v_2$}
    \put(15,25.5){$v_3$}
    \put(73,24.5){$v_3$}
    \put(39.5,21){$v_4$}
    \put(96,17.5){$v_4$}
    \put(39.5,10){$v_5$}
    \put(90,10){$v_5$}
    \put(15,5){$v_6$}
    \put(36,-1){$\sigma(K_6)=\sigma(\mathcal{R}_S(K_6))$}
    \end{overpic}
  \end{center}
  \caption{Reduction of $K_6$ over $S=\{v_1,\dots,v_5\}$. The edges of $K_6$ and $\mathcal{R}_S(K_6)$ have weight $1$ and $1+1/\lambda$ respectively.}\label{fig4}
\end{figure}

We note that both $\sigma(G|_{\bar{S}})$ and $\sigma^{-1}(G|_{\bar{S}})$ are easily calculated via equation (\ref{eq0.2}). Therefore, theorem \ref{maintheorem} offers a quick way of computing the eigenvalues of a reduced graph if the spectrum of the original unreduced graph is known.

More importantly, however, theorem \ref{maintheorem} tells us the eigenvalues we will gain and possibly lose if we reduce a graph. In section 6 this will allow us to modify the dynamical network $(\mathcal{F},X)$ while maintaining the network's stability.

\subsection{Sequential Reductions}
In section 3.2 we observed that any reduction $\mathcal{R}_S(G)$ of a graph $G\in\mathbb{G}$ is again a graph in $\mathbb{G}$. It is therefore natural to consider sequential reductions of a graph $G\in\mathbb{G}$. This requires that we first extend our notation to sequences of isospectral reductions.

For $G=(V,E,\omega)$ suppose $S_m\subseteq S_{m-1}\subseteq\dots\subseteq S_1\subseteq V$ such that $S_1\in st(G)$, $\mathcal{R}_1(G)=\mathcal{R}_{S_1}(G)$ and $$S_{i+1}\in st(\mathcal{R}_i(G)) \ \text{where} \ \mathcal{R}_{S_{i+1}}(\mathcal{R}_i(G))=\mathcal{R}_{i+1}(G), \  1\leq i\leq m-1.$$ If this is the case we say $S_1,\dots,S_m$ \textit{induces a sequence of reductions} on $G$ with \textit{final vertex set} $S_m$. By way of notation we write $\mathcal{R}_m(G)=\mathcal{R}(G;S_1,\dots,S_m)$ where $\mathcal{R}(G;S_1,\dots,S_m)$ denotes the graph $G$ reduced over the vertex set $S_1$ then $S_2$ and so on until $G$ is reduced over the final vertex set $S_m$.

\begin{theorem}\label{theorem3}
For $G=(V,E,\omega)$ suppose $S\in st(G)$. If there are vertex sets $S\subseteq S_{m-1}\subseteq\dots\subseteq S_{1}\subseteq V$ then $S_1,\dots,S_{m-1},S$ induces a sequence of reductions on $G$. Moreover, $\mathcal{R}(G;S_1,\dots,S_{m-1},S)=\mathcal{R}_{S}(G)$.
\end{theorem}

Let $S$ be a structural set of the graph $G=(V,E,\omega)$. In light of theorem \ref{theorem3} any sequence $S_1,\dots,S_{m-1},S$ where
\begin{equation}\label{star}
S\subseteq S_{m-1}\subseteq\dots\subseteq S_{1}\subseteq V
\end{equation}
induces a sequence of reductions on $G$. To see how many possible sequences satisfy (\ref{star}) we note the following.

Each such sequence satisfying (\ref{star}) corresponds to a partition of the set $V-S$. To see this suppose $V=S_0$ and $S=S_m$. Then the sequence $S_1,\dots,S_{m-1},S_m$ corresponds to the partition $\{P_1,\dots,P_m\}$ of $V-S$ where $P_i=S_{i-1}-S_i$ for $1\leq i\leq m$.

Note that the number of ways to partition a set is exponential in its number of elements. Hence, there are often a large number of ways to sequentially reduce a graph to one of its structural sets. However, despite the potentially large number of such sequential reductions the following corollary of theorem \ref{theorem3} offers the following uniqueness result.

\begin{corollary}
Suppose $G=(V,E,\omega)$ has structural set $S$. If $S\subseteq S_{m-1}\subseteq\dots\subseteq S_{1}\subseteq V$ and $S\subseteq T_{n-1}\subseteq\dots\subseteq T_{1}\subseteq V$ then it follows from lemma \ref{theorem3} that $$\mathcal{R}(G;S_1,\dots,S_{m-1},S)=\mathcal{R}_{S}(G)=\mathcal{R}(G;T_1,\dots,T_{n-1},S).$$
\end{corollary}

That is, the graph $\mathcal{R}_{S}(G)$ resulting from a sequence of reductions depends only on the final vertex set $S$ if $S\in st(G)$.

\begin{example}
Let $G$ again be the graph considered in example \ref{ex3}. If we let both $S_1=\{v_1,v_2,v_3,v_4\}$ and $S_2=\{v_1,v_3\}$ then the following holds.

The vertex set $S_2$ is a structural set of $G$. In particular $S_2=S$ where $S$ is the structural set considered in example \ref{ex3}. Moreover, $S_2\subseteq S_1$. Hence, lemma \ref{theorem3} implies the sequence $S_1,S_2$ induces a sequence of reductions on $G$. Additionally, as $S_2=S$ then $\mathcal{R}(G;S_1,S_2)=\mathcal{R}_S(G)$ (see figures \ref{fig2} and \ref{fig5}).

We note that as $\bar{S}$ has four elements then there are fourteen ways to reduce $G$ to the graph $\mathcal{R}_S(G)$. This follows from the fact that there are fourteen ways to partition a set with four elements. However, each of these reductions will result in the same graph.
\end{example}

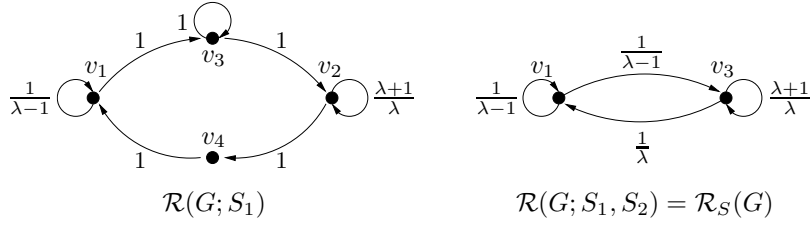
\begin{figure}
  \begin{center}
    \begin{overpic}[scale=.5]{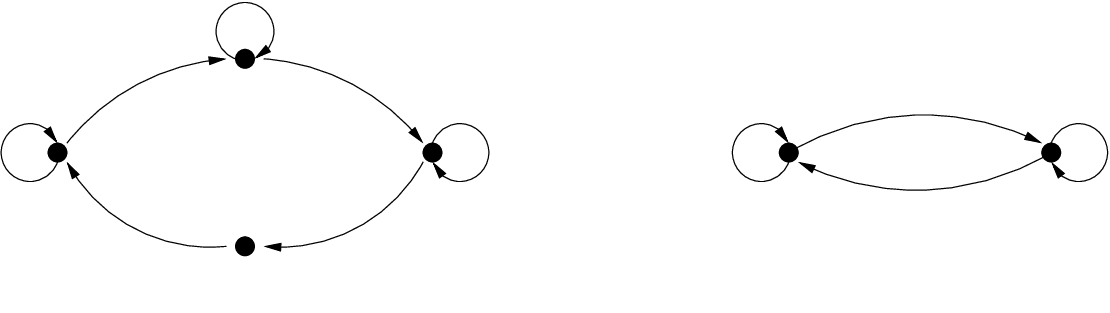}
    \put(15,-1.5){$\mathcal{R}(G;S_1)$}
    \put(-7,13.5){$\frac{1}{\lambda-1}$}
    \put(44.5,13.5){$\frac{\lambda+1}{\lambda}$}
    \put(4,18){$v_1$}
    \put(67,18){$v_1$}
    \put(37,18){$v_2$}
    \put(20.5,8){$v_4$}
    \put(20.5,19.5){$v_3$}
    \put(92.5,18){$v_3$}
    \put(79.5,20){{$\frac{1}{\lambda-1}$}}
    \put(81.5,6.5){$\frac{1}{\lambda}$}
    \put(59,13.5){$\frac{1}{\lambda-1}$}
    \put(100.5,13.5){$\frac{\lambda+1}{\lambda}$}
    \put(65,-1.5){$\mathcal{R}(G;S_1,S_2)=\mathcal{R}_S(G)$}
    \put(17,24){\small$1$}
    \put(11,21.5){\small$1$}
    \put(11,4.5){\small$1$}
    \put(31,21.5){\small$1$}
    \put(31,4.5){\small$1$}
    \end{overpic}
  \end{center}
  \caption{Sequential reduction of the graph $G$ from figure \ref{fig2}.}\label{fig5}
\end{figure}

Let $G=(V,E,\omega)$ and suppose $S\notin st(G)$. It is a natural question to ask whether the result(s) of lemma \ref{theorem3} holds for the vertex set $S$. Specifically, (i) do any sets satisfying $S\subseteq S_{m-1}\subseteq\dots\subseteq S_{1}\subseteq V$ induce a sequence of reductions on $G$; and (ii) can $G$ be reduced to a graph with vertex set $S$ via some sequence of reductions? If so, is this reduction unique?

The answer to (i) is \textit{negative} since $S_1$ need not be a structural set of $G$. To answer question (ii) we note the following. If the weight $\omega(e_{ii})\neq\lambda$ for some $v_i\in V$ then the vertex set $S=V-\{v_i\}$ is a structural set of $G$. This follows from the fact that $\bar{S}=\{v_i\}$. Hence, the graph $G|_{\bar{S}}$ is the graph restricted to the single vertex $v_i$. In particular, this implies that any cycle of $G|_{\bar{S}}$ is a loop.

Therefore, any graph $G\in\mathbb{G}$ can be reduced over the structural set $S=V-\{v_i\}$ if it is known that $\omega(e_{ii})\neq\lambda$. Another way to state this is that it is possible to remove the vertex $v_i$ from $G$ via an isospectral reduction if $\omega(e_{ii})\neq\lambda$ without knowing anything about the graph structure of $G$. This has the following important implication.

Suppose it is known that no loop of $G$ or any loop of any sequential reduction of $G$ has weight $\lambda$. If this is the case then it is possible to remove any sequence of single vertices from $G$ via a sequence of isospectral reductions. Therefore, $G$ can be sequentially reduced to a graph on any subset of its vertex set. This idea is the motivation behind the following.

For any polynomial $p\in\mathbb{C}[\lambda]$ let $deg(p)$ denote the degree of $p$. If $w=p/q\in\mathbb{W}[\lambda]$ where $p,q\in\mathbb{C}[\lambda]$ let $$\pi(w)=deg(p)-deg(q).$$

Again it is required to show that the function $\pi:\mathbb{W}[\lambda]\rightarrow\mathbb{Z}$ is well defined. To do so suppose $p/q$ and $r/s$ are equivalent in $\mathbb{W}[\lambda]$ i.e. $ps=rq$. We then have that $deg(ps)=deg(rq)$. As $deg(ps)=deg(p)+deg(s)$ and similarly $deg(rq)=\deg(r)+deg(q)$ then $$deg(p)-deg(q)=deg(r)-deg(s).$$ That is, $\pi(p/q)=\pi(r/s)$ implying the function $\pi$ is well defined.

Let $\mathbb{W}_\pi[\lambda]$ be the subset of $\mathbb{W}[\lambda]$ given by
$$\mathbb{W}_\pi[\lambda]=\big\{w\in\mathbb{W}[\lambda]:\pi(w)\leq 0\big\}.$$
That is, $\mathbb{W}_\pi[\lambda]$ is the set of rational functions in which the degree of the numerator is less than or equal to the degree of the denominator. Let $\mathbb{G}_\pi$ be the graphs in $\mathbb{G}$ with edge weights in the set $\mathbb{W}_\pi[\lambda]$.

\begin{lemma}\label{lemma1}
If $G\in\mathbb{G}_\pi$ and $S\in st(G)$ then $\mathcal{R}_S(G)\in\mathbb{G}_\pi$. In particular, no loop of $G$ and no loop of any reduction of $G$ can have weight $\lambda$.
\end{lemma}

By the reasoning above, if $G\in\mathbb{G}_\pi$ then $G$ can be (sequentially) reduced to a graph on any subset of its vertex set. This result is stated in the following theorem.

\begin{theorem}\label{theorem-1}\textbf{(Existence of Isospectral Reductions Over any Vertex Set)} Let $G=(V,E,\omega)$ be graph in $\mathbb{G}_\pi$ and suppose $\mathcal{V}$ is a nonempty subset of $V$. Then there exist sets $\mathcal{V}\subseteq S_{m-1}\subseteq\dots\subseteq S_1\subseteq V$ such that $S_1,\dots, S_{m-1},\mathcal{V}$ induces a sequence of reductions on $G$.
\end{theorem}

For $G\in\mathbb{G}_\pi$ it is therefore possible to reduce a graph $G\in\mathbb{G}_{\pi}$ to a graph on any (nonempty) subset of vertex set via some sequence of isospectral reductions. Moreover, such sequences have the following uniqueness property.

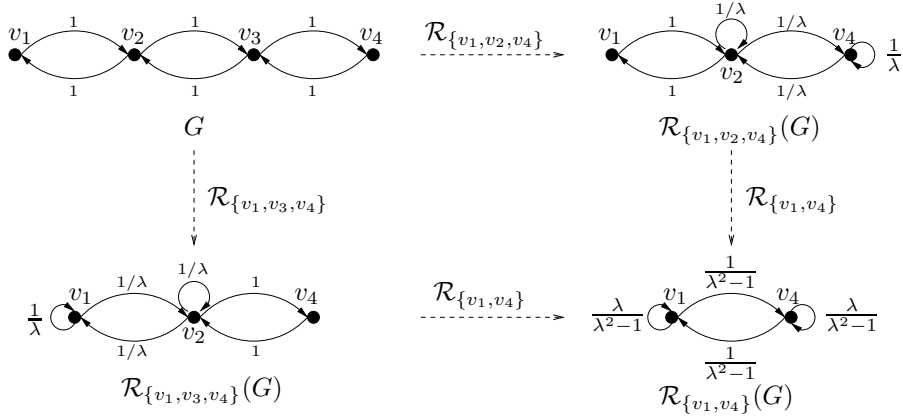
\begin{figure}
  \begin{center}
    \begin{overpic}[scale=.5]{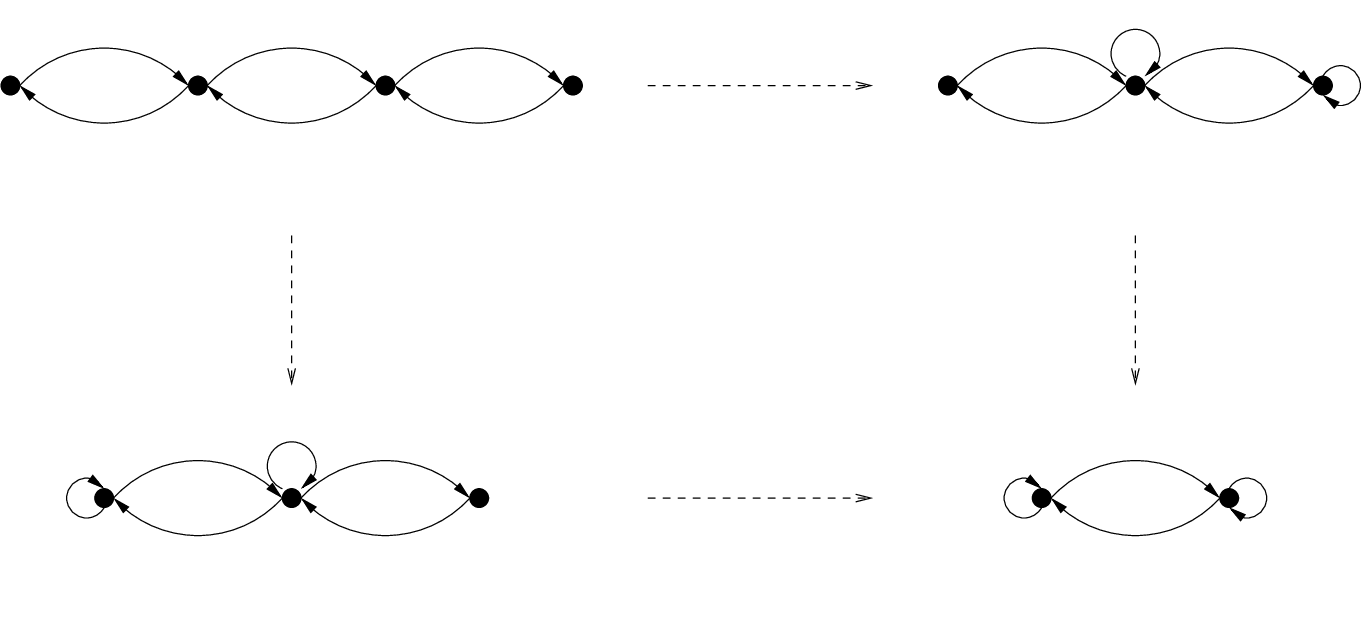}
    \put(20.25,30){$G$}
    \put(7,42.5){\tiny$1$}
    \put(21,42.5){\tiny$1$}
    \put(34.5,42.5){\tiny$1$}
    \put(7,34.5){\tiny$1$}
    \put(21,34.5){\tiny$1$}
    \put(34.5,34.5){\tiny$1$}
    \put(0,41){$v_1$}
    \put(13,41){$v_2$}
    \put(26.5,41){$v_3$}
    \put(40.5,41){$v_4$}

    \put(68,41){$v_1$}
    \put(82,36.5){$v_2$}
    \put(95,41){$v_4$}
    \put(76,42.5){\tiny$1$}
    \put(76,34.5){\tiny$1$}
    \put(81.5,44){\tiny$1/\lambda$}
    \put(88.5,42.5){\tiny$1/\lambda$}
    \put(88.5,34.5){\tiny$1/\lambda$}
    \put(101,38.5){$\frac{1}{\lambda}$}
    \put(75,30){$\mathcal{R}_{\{v_1,v_2,v_4\}}(G)$}

    \put(7,11){$v_1$}
    \put(20,6.5){$v_2$}
    \put(33,11){$v_4$}
    \put(12.5,12.5){\tiny$1/\lambda$}
    \put(12.5,4.5){\tiny$1/\lambda$}
    \put(19.5,14){\tiny$1/\lambda$}
    \put(28,12.5){\tiny$1$}
    \put(28,4.5){\tiny$1$}
    \put(2,8){$\frac{1}{\lambda}$}
    \put(13,0){$\mathcal{R}_{\{v_1,v_3,v_4\}}(G)$}

    \put(75.5,11){$v_1$}
    \put(88.5,11){$v_4$}
    \put(67,8.5){$\frac{\lambda}{\lambda^2-1}$}
    \put(94,8.5){$\frac{\lambda}{\lambda^2-1}$}
    \put(80,3){$\frac{1}{\lambda^2-1}$}
    \put(80,13.5){$\frac{1}{\lambda^2-1}$}
    \put(75,-1){$\mathcal{R}_{\{v_1,v_4\}}(G)$}

    \put(48,41){$\mathcal{R}_{\{v_1,v_2,v_4\}}$}
    \put(23,22){$\mathcal{R}_{\{v_1,v_3,v_4\}}$}
    \put(49,11){$\mathcal{R}_{\{v_1,v_4\}}$}
    \put(85,22){$\mathcal{R}_{\{v_1,v_4\}}$}

    \end{overpic}
  \end{center}
  \caption{Distinct sequences of isospectral reductions with the same outcome.}\label{fig6}
\end{figure}

\begin{theorem}\label{theorem-3}\textbf{(Uniqueness of Isospectral Reductions Over any Vertex Set)} Let $G=(V,E,\omega)$ be graph in $\mathbb{G}_\pi$ and suppose $\mathcal{V}$ is a nonempty subset of $V$. If $S_1,\dots, S_{m-1},\mathcal{V}$ and $T_1,\dots, T_{n-1},\mathcal{V}$ both induce a sequence of reductions on $G$ then $\mathcal{R}(G;S_1,\dots, S_{m-1},\mathcal{V})=\mathcal{R}(G;T_1,\dots, T_{n-1},\mathcal{V})$.
\end{theorem}

The results of theorem \ref{theorem-1} and theorem \ref{theorem-3} allows us to give the following definition.

\begin{definition}\label{note}
Let $G=(V,E,\omega)$ be graph in $\mathbb{G}_\pi$. If $\mathcal{V}\subseteq V$ is nonempty define
$$\mathcal{R}_{\mathcal{V}}[G]=\mathcal{R}(G;S_1,\dots, S_{m-1},\mathcal{V})$$
where $S_1,\dots, S_{m-1},\mathcal{V}$ is any sequence that induces a sequence of reductions on $G$ with final vertex set $\mathcal{V}$.
\end{definition}

The graph $\mathcal{R}_{\mathcal{V}}[G]$ is well defined as a result of theorems \ref{theorem-1} and \ref{theorem-3}. The notation $\mathcal{R}_\mathcal{V}[G]$ given in definition \ref{note} is intended to emphasize the fact that $\mathcal{V}$ need not be a structural set of $G$.

\begin{remark}
Note that $\pi(c)=0$ for any $c\in\mathbb{C}$. Hence, if $M(G)\in\mathbb{C}^{n\times n}$ then $G\in\mathbb{G}_\pi$. Therefore, any graph with complex weights can be uniquely reduced to a graph on any nonempty subset of its vertex set. This is of particular importance for the estimation of spectra of matrices with complex entries in \cite{BW09}.
\end{remark}

\begin{example}
Let $G=(V,E,\omega)$ be the graph shown in figure \ref{fig6}. Our goal is to reduce $G$ over its vertex set $\{v_1,v_4\}\subset V$. Note that as $G\in\mathbb{G}_\pi$ theorem \ref{theorem-1} guarantees that there is at least one sequence of reductions that reduces $G$ to the graph $\mathcal{R}_{\{v_1,v_4\}}[G]$.

In fact there are exactly two. This follows from the fact that $\{v_1,v_4\}\notin st(G)$. Hence, $G$ cannot be reduced over $\{v_1,v_4\}$ with a single reduction. However, any (nontrivial) reduction of $G$ removes at least one vertex from $G$.

Therefore, the two possible ways of reducing $G$ to the vertex set $\{v_1,v_4\}$ are
\begin{align}\label{Al1}
\mathcal{R}_{\{v_1,v_4\}}[G]&=\mathcal{R}(G;\{v_1,v_2,v_4\},\{v_1,v_4\}); \ \ \text{and}\\
\mathcal{R}_{\{v_1,v_4\}}[G]&=\mathcal{R}(G;\{v_1,v_3,v_4\},\{v_1,v_4\}).\label{Al2}
\end{align}

Both of the reductions given in (\ref{Al1}) and (\ref{Al2}) are shown in figure \ref{fig6}. The dashed arrows labeled $\mathcal{R}_T$ in this figure represent the reduction of a graph over some structural set $H$. This notation is meant to emphasize that this diagram commutes. That is,
$$\mathcal{R}_{\{v_1,v_4\}}\big(\mathcal{R}_{\{v_1,v_2,v_4\}}(G)\big)=\mathcal{R}_{\{v_1,v_4\}}\big(\mathcal{R}_{\{v_1,v_3,v_4\}}(G)\big)$$
as guaranteed by theorem \ref{theorem-3}.
\end{example}

\subsection{Equivalence Relations}

Theorem \ref{theorem-1} and theorem \ref{theorem-3} assert that a graph $G\in\mathbb{G}_\pi$ has a unique reduction to any (nonempty) subset of its vertex set via some sequence of isospectral reductions. In this section this property will allow us to define various equivalence relations on the graphs in $\mathbb{G}_{\pi}-\emptyset$.

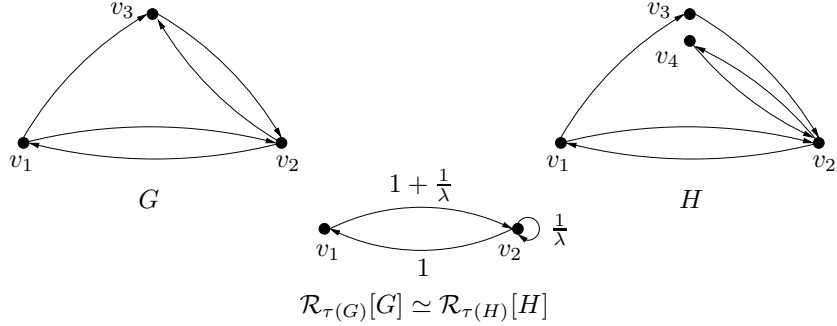
\begin{figure}
  \begin{center}
    \begin{overpic}[scale=.45]{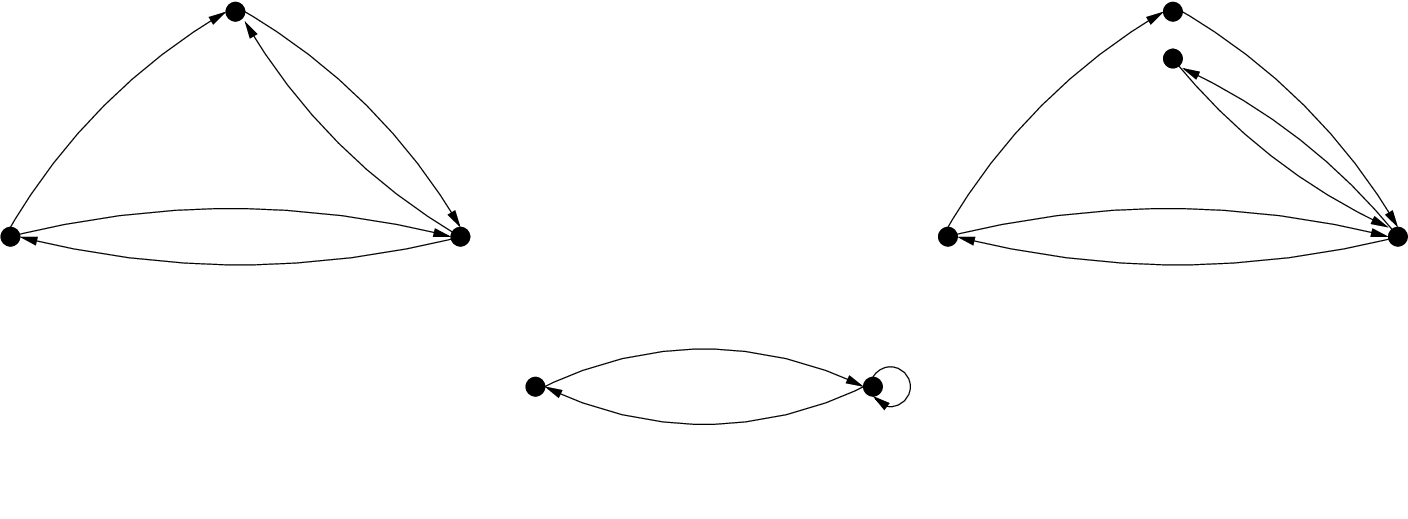}
    \put(-1,17){$v_1$}
    \put(32,17){$v_2$}
    \put(46,14){$1+\frac{1}{\lambda}$}
    \put(15,12){$G$}
    \put(82,12){$H$}
    \put(65.5,17){$v_1$}
    \put(98.5,17){$v_2$}
    \put(11.5,36){$v_3$}
    \put(78,36){$v_3$}
    \put(79,30){$v_4$}
    \put(66,8.5){$\frac{1}{\lambda}$}
    \put(49.5,3.5){$1$}
    \put(35,-1){$\mathcal{R}_{\tau(G)}[G]\simeq\mathcal{R}_{\tau(H)}[H]$}
    \put(37,6){$v_1$}
    \put(59.5,6){$v_2$}
    \end{overpic}
  \end{center}
  \caption{$G$ and $H$ are equivalent under the relation induced by the rule $\tau$ given in example \ref{ex0}.}\label{fig7}
\end{figure}

Two weighted digraphs $G_1=(V_1,E_1,\omega_1)$, and $G_2=(V_2,E_2,\omega_2)$ are \textit{isomorphic} if there is a bijection $b:V_1\rightarrow V_2$ such that there is an edge $e_{ij}$ in $G_1$ from $v_i$ to $v_j$ if and only if there is an edge $\tilde{e}_{ij}$ between $b(v_i)$ and $b(v_j)$ in $G_2$ with $\omega_2(\tilde{e}_{ij})=\omega_1(e_{ij})$. If the map $b$ exists it is called an \textit{isomorphism} and we write $G_1\simeq G_2$.

An isomorphism is essentially a relabeling of the vertices of a graph. Therefore, if two graphs are isomorphic then their spectra are identical.

\begin{theorem}\textbf{(Spectral Equivalence)}\label{theorem0.1}
Suppose for any graph $G=(V,E,\omega)$ in $\mathbb{G}_\pi-\emptyset$ that $\tau$ is a rule that selects a unique nonempty subset $\tau(G)\subseteq V$. Then $\tau$ induces an equivalence relation $\sim$ on the set $\mathbb{G}_\pi-\emptyset$ where $G\sim H$ if the graph $\mathcal{R}_{\tau(G)}[G]\simeq\mathcal{R}_{\tau(H)}[H]$.
\end{theorem}

\begin{example}\label{ex0}
Let $G=(V,E,\omega)$. The \textit{out degree} of a vertex $v_i\in V$ is the number of outgoing edges incident to $v_i$ i.e. the number $|\{j:\omega(e_{ij})\neq 0\}|$. If $G\in\mathbb{G}_\pi$ let $\tau(G)\subseteq V$ be the set of vertices of maximal out degree.

Observe that for each graph $G\in\mathbb{G}_\pi-\emptyset$ the set $\tau(G)$ both exists and is unique. Thus the relation of having an isomorphic reduction with respect to this rule induces an equivalence relation on $\mathbb{G}_\pi-\emptyset$.

In figure \ref{fig7} the graphs $G$ and $H$ have the vertex set $\tau(G)=\{v_1,v_2\}=\tau(H)$ of maximal out degree. As shown in the figure, the graph $\mathcal{R}_{\tau(G)}[G]\simeq\mathcal{R}_{\tau(H)}[H]$. Hence, $G\sim H$ under the relation $\sim$ induced by the rule $\tau$.
\end{example}

Note that the relation of simply having isomorphic reductions is not transitive. That is, if $\mathcal{R}_S[G]\simeq\mathcal{R}_T[H]$ and $\mathcal{R}_U[H]\simeq\mathcal{R}_V[K]$ it is not necessarily the case that there are sets $X$ and $Y$, subsets of the vertex sets of $G$ and $K$ respectively, such that $\mathcal{R}_X[G]\simeq\mathcal{R}_Y[K]$.

As an example, in figure \ref{fig100} both $\mathcal{R}_S[G]\simeq\mathcal{R}_S[H]$ and $\mathcal{R}_T[H]\simeq\mathcal{R}_T[K]$ where $S=\{v_1,v_2\}$ and $T=\{v_3,v_4\}$. However, one can quickly check that for no subsets $X\subseteq S$ and $Y\subseteq T$ are $\mathcal{R}_X[G]\simeq\mathcal{R}_Y[K]$. Overcoming this intransitivity requires some rule $\tau$ that selects a unique set of vertices from each graph in $\mathbb{G}$ (see theorem \ref{theorem0.1}). If $\tau$ does not, e.g. $\tau:=$\textit{remove any vertex}, then $\tau$ does not induce an equivalence relation on $\mathbb{G}-\emptyset$.

\begin{figure}
  \begin{center}
    \begin{overpic}[scale=.55]{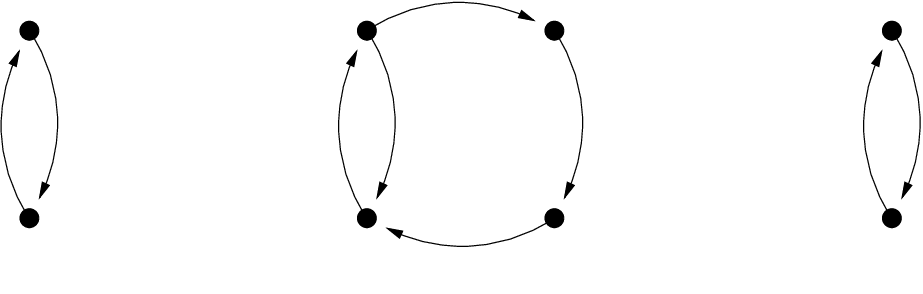}
    \put(1.5,30.5){$v_1$}
    \put(1.5,4){$v_2$}
    \put(-3.5,17){$1$}
    \put(8.5,17){$1+\frac{1}{\lambda^2}$}
    \put(1.5,-2){$G$}
    \put(49,-2){$H$}
    \put(38.5,30.5){$v_1$}
    \put(38.5,4){$v_2$}
    \put(60.5,30.5){$v_3$}
    \put(60.5,4){$v_4$}
    \put(95.5,30.5){$v_3$}
    \put(95.5,4){$v_4$}
    \put(95.5,-2){$K$}
    \put(102,17){$1$}
    \put(83,17){$\frac{1}{\lambda^2-1}$}
    \end{overpic}
  \end{center}
  \caption{$G=\mathcal{R}_S[H]$ and $\mathcal{R}_T[H]=K$ for $S=\{v_1,v_2\}$ and $T=\{v_3,v_4\}$ but the graphs $G$ and $K$ do not have isomorphic reductions.}\label{fig100}
\end{figure}
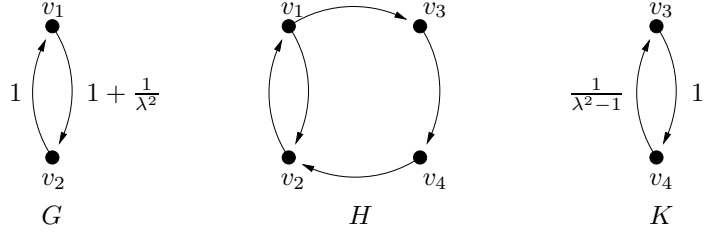

Importantly, choosing a rule that does select a unique vertex set allows one to study the graphs in $\mathbb{G}_\pi-\emptyset$ modulo some particular graph feature. For instance, in example \ref{ex0} this graph feature (or vertex set) is the vertices that do not have maximal out degree.

\section{Isospectral Transformations Over Fixed Weight Sets}

The isospectral graph reductions of the previous section modify not only the graph structure but also the weight set of a graph. That is, if $\mathcal{R}_S(G)=(S,\mathcal{E},\mu)$ is any reduction of $G=(V,E,\omega)$ then typically $\omega(E)\neq \mu(\mathcal{E})$, i.e.
$$\{\omega(e_{ij}):e_{ij}\in E\}\neq\{\mu(e_{ij}):e_{ij}\in \mathcal{E}\}.$$
This may lead one to assume that our procedure simply shifts the complexity of the graph's structure to its set of edge weights. However, this is not the case.

In this section we introduce different procedures of reducing and expanding a graph that restrict the weights of the transformed graph's edges to particular subsets of $\mathbb{W}[\lambda]$. As before, the procedure preserves the spectrum of the graph up to a known set. Such transformations are of particular importance in section 6 where dynamical network expansions are discussed.

\subsection{Branch Expansions}
The idea behind an isospectral graph transformation that preserves a graph's edge weights is simple enough. If two graphs $G,H\in\mathbb{G}$ have the same branch structure (including weights) then they should have similar spectra. If this is the case, then we say that $G$ is an \textit{isospectral transformation} of $H$ over its weight set and vice-versa.

To make this precise suppose $G=(V,E,\omega)$ and $S\in st(G)$. If the branch $\beta=v_{i_1},\dots,v_{i_m}\in\mathcal{B}_S(G)$ let $\Omega_G(\beta)$ be the ordered sequence $$\Omega_G(\beta)=\omega(e_{i_1i_2}),\dots,\omega(e_{i_{j-1},i_j}),\omega(e_{i_ji_j}),\omega(e_{i_j,i_{j+1}}),\dots,\omega(e_{i_{m-1},i_m}).$$ for $m>1$ and $\omega(e_{i_1i_1})$ if $m=1$. We call $\Omega_G(\beta)$ the \textit{weight sequence} of the branch $\beta$. Moreover, we let $\Omega_G(\beta)_{j,j+1}=\omega(e_{i_j,i_{j+1}})$ and $\Omega_G(\beta)_{jj}=\omega(e_{i_ji_j})$ for $1\leq j\leq m-1$.

Let $G,H\in\mathbb{G}$. Suppose $S=\{v_1,\dots,v_m\}$ is a structural set of both $G$ and $H$. The branch set $\mathcal{B}_{ij}(G;S)$ is \textit{isomorphic} to $\mathcal{B}_{ij}(H;S)$ if there is a bijection
$$b:\mathcal{B}_{ij}(G;S)\rightarrow\mathcal{B}_{ij}(H;S)$$
such that $\Omega_G(\beta)=\Omega_H(b(\beta))$ for each $\beta\in\mathcal{B}_{ij}(G;S)$. If such a map exists we write $\mathcal{B}_{ij}(G;S)\simeq\mathcal{B}_{ij}(H;S)$. If
$$\mathcal{B}_{ij}(G;S)\simeq\mathcal{B}_{ij}(H;S) \ \ \text{for each} \ \ 1\leq i,j\leq m$$
we say $\mathcal{B}_S(G)$ is \textit{isomorphic} to $\mathcal{B}_S(H)$ and write $\mathcal{B}_S(G)\simeq\mathcal{B}_S(H)$.

\begin{example}\label{ex9}
Suppose $H$ and $G$ are the graphs in figure \ref{fig9} with unit edge weights. The vertex set $S=\{v_1,v_3\}$ can be seen to be a structural set of both $H$ and $G$. Moreover,
\begin{align*}
\mathcal{B}_{11}(H;S)=\{v_1,v_2,v_1\},& \hspace{0.15in} \mathcal{B}_{11}(G;S)=\{v_1,u_5,v_1\};\\
\mathcal{B}_{13}(H;S)=\{v_1,v_2,v_3\},& \hspace{0.15in} \mathcal{B}_{13}(G;S)=\{v_1,u_2,v_3\};\\
\mathcal{B}_{31}(H;S)=\{v_3,v_4,v_1\},& \hspace{0.15in} \mathcal{B}_{31}(G;S)=\{v_1,u_4,v_3\};\\
\mathcal{B}_{33}(H;S)=\{v_3,v_3;v_3,v_4,v_3\},& \hspace{0.15in} \mathcal{B}_{33}(G;S)=\{v_3,v_3;v_3,v_u,v_3\}.
\end{align*}

Note that the branch $\beta=\{v_1,u_5,v_1\}$ in $\mathcal{B}_{11}(G;S)$ has the weight sequence given by $\Omega_G(\beta)=1,1,1$. Similarly, the branch $\gamma=\{v_1,v_2,v_1\}$ in $\mathcal{B}_{11}(H;S)$ has the weight sequence $\Omega_H(\gamma)=1,1,1$. Hence, $\mathcal{B}_{11}(H;S)\simeq\mathcal{B}_{11}(G;S)$. Continuing in this manner, one can check that each $\mathcal{B}_{ij}(H;S)\simeq\mathcal{B}_{ij}(G;S)$ for each $i,j\in\{1,3\}$. Therefore, $\mathcal{B}_S(H)\simeq\mathcal{B}_S(G)$.
\end{example}

\begin{remark}
We note that if $\mathcal{B}_S(H)\simeq\mathcal{B}_S(G)$ the graphs $G$ and $H$ need not be isomorphic (see example \ref{ex9}).
\end{remark}

For $G\in\mathbb{G}$ and $S\in st(G)$ suppose $\alpha=v_1,\dots,v_m$ and $\beta=u_1,\dots,u_n$ are branches in $\mathcal{B}_S(G)$. These branches are said to be \textit{independent} if $$\{v_2,\dots,v_{m-1}\}\cup\{u_2,\dots,u_{n-1}\}\neq \emptyset.$$
That is, $\alpha$ and $\beta$ are independent if they share no interior vertices.

\begin{definition}\label{def10}
Let $G,H\in\mathbb{G}$ and $S\in st(G),st(H)$. Suppose\\
(i) $\mathcal{B}_S(G)\simeq\mathcal{B}_T(H)$;\\
(ii) the branches of $\mathcal{B}_S(H)$ are independent;\\
(iii) each vertex of $G$ and $H$ belongs to a branch of $\mathcal{B}_S(G)$ and $\mathcal{B}_S(H)$ respectively.
Then we call $H$ a \textit{branch expansion} of $G$ with respect to $S$.
\end{definition}

\begin{proposition}\label{propnew}
Let $G\in\mathbb{G}$ and $S\in st(G)$. If $H$ and $K$ are branch expansions of $G$ with respect to $S$ then $H\simeq K$.
\end{proposition}

\begin{proof}
Suppose $H=(V_1,E_1,\omega_1)$ and $K=(V_2,E_2,\omega_2)$ are branch expansions of $G$ with respect to $S$. Then by assumption (i) of definition \ref{def10} it follows that there is a bijection $$b:\mathcal{B}_S(H)\rightarrow\mathcal{B}_S(G)$$
such that $\Omega_H(\beta)=\Omega_K(b(\beta))$ for all $\beta\in\mathcal{B}_S(H)$.

For the branch $\beta=v_{i_1},\dots,v_{i_m}$ let $\beta(j)=v_{i_j}$ for $1\leq j\leq m$. We call $j$ the index of $v_{i_j}$. By assumption each vertex $v\in V_1$ belongs to a branch of $\mathcal{B}_S(H)$ and the branches of $\mathcal{B}_S(H)$ are pairwise independent. Hence, either $v\in S$ or $v=\beta(i)$ for exactly one $\beta\in\mathcal{B}_S(H)$ and index $i$. Let $B:V_1\rightarrow V_2$ be the map
$$B(v)=
\begin{cases}
v \ \ &\text{if} \ \ v\in S\\
\big(b(\beta)\big)(i) \ \ &\text{if} \ \ v\notin S \ \ \text{and} \ \ v=\beta(i)
\end{cases}.
$$

By assumption each $v\in V_2$ is similarly either in $S$ or equal to $\beta(i)$ again for exactly one $\beta\in\mathcal{B}_S(K)$ and index $i$. Therefore, $B$ has the inverse
$$B^{-1}(v)=
\begin{cases}
v \ \ &\text{if} \ \ v\in S\\
\big(b^{-1}(\beta)\big)(i) \ \ &\text{if} \ \ v\notin S \ \ \text{and} \ \ v=\beta(i)
\end{cases}
$$ implying the map $B:V_1\rightarrow V_2$ is a bijection.

Let $e_{ij}\in E_1$. If $v_i,v_j\in S$ then $v_i,v_j\in\mathcal{B}_S(H)$ implying
$$\omega_1(e_{ij})=\Omega_H(v_i,v_j)=\Omega_K(v_{B(i)},v_{B(j)})=\omega_2(e_{B(i)B(j)}).$$
Suppose $v_i,v_j\notin S$. If $v_i\neq v_j$ then there are branches $\alpha_1=v_1,\dots,v_i,\dots,v_s$ and $\alpha_2=u_1,\dots,v_j,\dots,v_t$ in $\mathcal{B}_S(H)$ containing $v_i$ and $v_j$ respectively. Therefore, the sequence $v_1,\dots,v_i,v_j,\dots,u_t
\in\mathcal{B}_S(H)$. As the branches of $\mathcal{B}_S(H)$ are pairwise independent then $\alpha_1=\alpha_2$ and $j=i+1$. Hence,
$$\omega_1(e_{ij})=\Omega_H(\alpha_1)_{i,i+1}=\Omega_K(b(\alpha_1))_{i,i+1}=\omega_2(e_{B(i)B(j)}).$$
If $v_i=v_j$ then there is a unique $\beta\in\mathcal{B}_S(H)$ and index $k$ such that $v_i=\beta(k)$. Hence, $$\omega_1(e_{ij})=\Omega_H(\beta)_{kk}=\Omega_K(b(\beta))_{kk}=\omega_2(e_{B(i)B(j)}).$$

Supposing $v_i\in S$, $v_j\notin S$ or $v_i\notin S$, $v_j\in S$ a similar argument implies that $\omega_1(e_{ij})=\omega_2(e_{B(i)B(j)})$. Hence, $H\simeq K$.
\end{proof}

It follows that a branch expansion of $G$ with respect to $S$ is unique up to a labeling of vertices. Therefore, any two expansions of $G$ with respect to $S$ are isomorphic. In what follows we let $\mathcal{X}_S(G)$ be a representative of the class of branch expansions. By slightly abusing our terminology we call any representative $\mathcal{X}_S(G)$ \textit{the branch expansion} of $G$ with respect to $S$.

The principle idea behind a branch expansion is the following. If $G\in\mathbb{G}$ and $S\in st(G)$ then the set of branches $\mathcal{B}_S(G)$ is uniquely defined. However, there are typically many other graphs $H$ with the same branch structure as $G$, i.e. $S\in st (H)$ such that $\mathcal{B}_S(H)\simeq\mathcal{B}_S(G)$.

A branch expansion of $G$ over $S$ is then a graph $H=\mathcal{X}_S(G)$ with identical branch structure but with the following restriction: The branches of $\mathcal{B}_S(H)$ are pairwise independent and every vertex of $H$ belongs to a branch in $\mathcal{B}_S(H)$. That is, any vertex of $\bar{S}$ in $H$ is part of exactly one branch in $\mathcal{B}_S(H)$.

Hence, given a graph $G$ and structural set $S$ we can algorithmically construct the expansion $\mathcal{X}_S(G)$ as follows. Start with the vertices $S$. If $\beta\in\mathcal{B}_{ij}(G;S)$ then both $v_i,v_j\in S$. Construct a path (or cycle) from $v_i$ to $v_j$ with weight sequence $\Omega(\beta)$ with \textit{new} interior vertices. By new we mean vertices that do not already appear on the graph we are constructing. Repeat this for each $\beta\in\mathcal{B}_{ij}(G;S)$. The resulting graph is the branch expansion $\mathcal{X}_S(G)$. A branch expansion is our first example of an isospectral graph transformation that preserves the weight set of a graph.

\begin{theorem}\label{theorem-2}
Let $G=(V,E,\omega)$ with structural set $S$. Then the graph $G$ and its branch expansion $\mathcal{X}_S(G)$ have the same set of edge weights. Moreover,  $$\det\big(M(\mathcal{X}_S(G))-\lambda I \big)=\det\big(M(G)-\lambda I\big)\prod_{v_i\in V-S}\big(\omega(e_{ii})-\lambda\big)^{n_i-1}$$ where $n_i$ is the number of branches in $\mathcal{B}_S(G)$ containing $v_i$.
\end{theorem}

\begin{example}\label{ex10}
Consider the graph $G=(V,E,\omega)$ and $H=(\mathcal{V},\mathcal{E},\mu)$ in figure \ref{fig9} with unit edge weights. As demonstrated in example \ref{ex9}, if $S=\{v_1,v_3\}$ then the branch set $\mathcal{B}_S(G)\simeq\mathcal{B}_S(H)$. Moreover, it can be seen from figure \ref{fig9} the five branches of $\mathcal{B}_S(G)$ share interior vertices. That is, the branches of $\mathcal{B}_S(G)$ are pairwise independent. Lastly, each vertex of $G$ belongs to at least one branch of $\mathcal{B}_S(G)$.

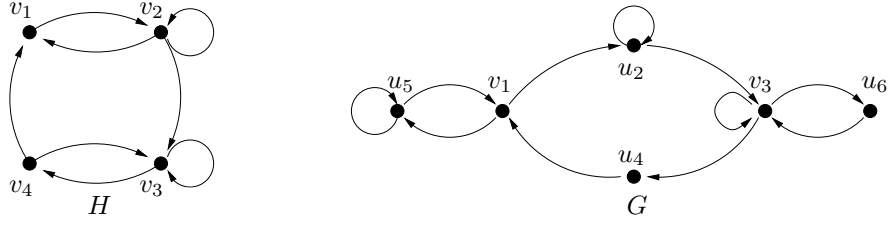
\begin{figure}
  \begin{center}
    \begin{overpic}[scale=.55]{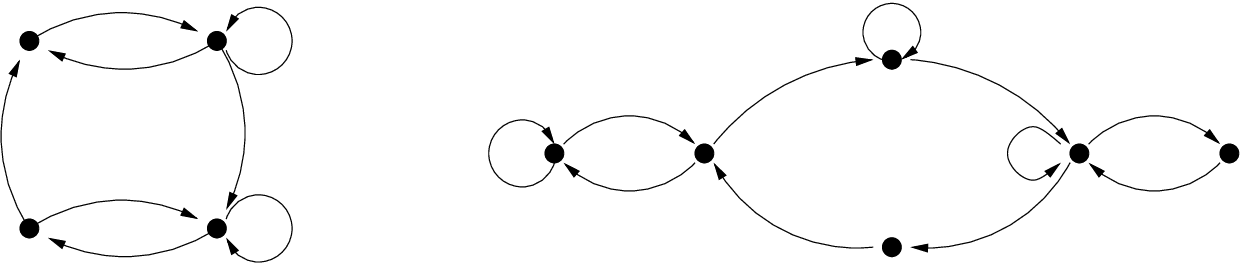}
    \put(0,20){$v_1$}
    \put(15,20){$v_2$}
    \put(15,-0.5){$v_3$}
    \put(0,-0.5){$v_4$}
    \put(9,-3){$H$}
    \put(43.5,11.5){$u_5$}
    \put(55,11.5){$v_1$}
    \put(70,13){$u_2$}
    \put(70,3){$u_4$}
    \put(85,11.5){$v_3$}
    \put(98,11.5){$u_6$}
    \put(71,-3){$G$}
    \end{overpic}
  \end{center}
  \caption{The graph $G$ is a branch expansion of the graph $H$ with respect to the structural set $S=\{v_1,v_3\}$.}\label{fig9}
\end{figure}

Therefore, the graph $G$ is a branch expansion of $H$ with respect to $S$, i.e. $G=\mathcal{X}_S(H)$. Importantly, note that the edge weights of $H$ and its expansion $G$ are identical, i.e. the sets $\omega(E)$ and $\mu(\mathcal{E})$ are both $\{0,1\}$.

Observe that the vertex $v_2$ of $H$ is an interior vertex of the branches $v_1,v_2,v_1;$ $v_1,v_2,v_3\in\mathcal{B}_S(H)$. Similarly, the vertex $v_4$ of $H$ is an interior vertex of the branches $v_3,v_4,v_3;$ $v_3,v_4,v_1\in\mathcal{B}_S(H)$. As $\omega(e_{22})=1$ and $\omega(e_{44})=0$ theorem \ref{theorem-2} implies that $\sigma\big(\mathcal{X}_S(H)\big)=\sigma(H)\cup\{1,0\}.$

We note that $G=\mathcal{X}_S(H)$ in figure \ref{fig9} is isomorphic to the graph $G$ in figure \ref{fig2}. Hence, $\sigma(G)=\{2,-1,1,1,0,0\}$ implying $\sigma(H)=\{2,-1,1,0\}$.
\end{example}

\subsection{Isospectral Graph Transformations Over Modified Weight Sets}
The branch expansions of the previous section are a method of graph transformation that separates the various branches of a graph. In this section we consider the reversal of this process. Specifically, we introduce a method of transforming a graph that incorporates a new technique of merging branches.

This type of isospectral transformation will have the additional property that it keeps the graph's weights in a fixed subset $\mathbb{U}\subset\mathbb{W}[\lambda]$. The particular type of subsets for which this will hold are semirings of $\mathbb{W}$.

The set $\mathbb{U}$ is a \textit{semiring} of $\mathbb{W}[\lambda]$ if it has the following properties. Both $0,1\in\mathbb{U}$ and for any $u_1,u_2\in\mathbb{U}$ both the product $u_1u_2$ and the sum $u_1+u_2$ are in $\mathbb{U}$. Moreover, as $\mathbb{U}\subset\mathbb{W}[\lambda]$ we are implicitly assuming that if $u\in\mathbb{U}$ then any other representation of $u$ is also in $\mathbb{U}$. Examples of such semirings include $\mathbb{C}[\lambda]$, $\mathbb{R}$, and $\mathbb{Z}^+=\{0,1,2,\dots\}$.

\begin{definition}
For $G=(V,E,\omega)$ let $S\in st(G)$. The set $S$ is a \textit{complete structural set} of $G$ if\\
(i) each cycle of $G$, including loops, contains a vertex in $S$; and\\
(ii) $\omega(e_{ii})\neq\lambda$ for each $v_i\in\bar{S}$.
\end{definition}

The difference between a structural set and a complete structural set of a graph $G$ is the following. Recall that if $S$ is a complete structural set of $G$ then every cycle of $G$ contains a vertex in $S$. If $S$ is simply a structural set of $G$ then loops of $G$ need not contain a vertex of $S$.

For $G\in\mathbb{G}$ let $st_0(G)$ denote the set of all complete structural sets of $G$. Additionally, let $\sigma_0(G)$ be the nonzero elements of $\sigma(G)$ including multiplicities. Hence, we refer to $\sigma_0(G)$ as the \textit{nonzero spectrum} of $G$. Moreover, we denote by $\rho(G)$ the \textit{spectral radius} of $G$. That is,
$$\rho(G)=\max_{\ell\in\sigma(G)}|\ell|.$$
The following is a corollary of theorem \ref{maintheorem}.

\begin{corollary}\label{cor100}
Let $G=(V,E,\omega)$ and suppose $S\in st_0(G)$. Then $$\sigma_0\big(\mathcal{R}_S(G)\big)=\sigma_0(G).$$
In particular, $\rho(\mathcal{R}_S(G)\big)=\rho(G)$.
\end{corollary}

\begin{proof}
Suppose $G=(V,E,\omega)$ and $S\in st_0(G)$. As no vertex in $\bar{S}$ can have a loop then $\omega(e_{ii})=0$ for each $v_i\in\bar{S}$. Equation (\ref{eq0.2}) then implies that
$$\det(M(G|_{\bar{S}})-\lambda I)=\prod_{v_i\in\bar{S}}\lambda.$$
Hence, $\sigma^{-1}(G|_{\bar{S}})=\emptyset$ and $\sigma(G|_{\bar{S}})=\{0,\dots,0\}$ in which $0$ has multiplicity $|\bar{S}|$. The corollary then follows from theorem \ref{maintheorem}.
\end{proof}

To simplify the discussion in what follows, suppose $G\in\mathbb{G}$ and $S\in st(G)$. If $\beta=v_{i_1},\dots,v_{i_m}\in\mathcal{B}_S(G)$ we will say $e_{i_j,i_{j+1}}$ is the $j$th edge \textit{belonging} to the branch $\beta$ for each $1\leq i\leq m-1$.

\begin{lemma}\label{newest}
Let $G=(V,E,\omega)$ and $S\in st_0(G)$. Then $\mathcal{X}_S(G)=(\mathcal{V},\mathcal{E},\mu)$ has the following properties.\\
(i) If $e_{ij}\in\mathcal{E}$ then $e_{ij}$ belongs to exactly one branch of $\mathcal{B}_S(\mathcal{X}_S(G))$.\\
(ii) If $\beta=v_1,\dots,v_m$ is a branch in $\mathcal{B}_S(\mathcal{X}_S(G))$ then
$$\Omega_{\mathcal{X}_S(G)}(\beta)=\mu(e_{12}),\dots,\mu(e_{k,k-1}),0,\mu(e_{k,k+1}),\dots,\mu(e_{m-1,m}).$$
\end{lemma}

\begin{proof}
As each vertex of $\mathcal{X}_S(G)$ belongs to at least one branch of $\mathcal{B}_S(\mathcal{X}_S(G))$ then the same holds any $e_{ij}\in\mathcal{E}$. On the other hand, suppose $e_{ij}$ belongs to both $\beta_1,\beta_2\in\mathcal{B}_S(\mathcal{X}_S(G))$. Then neither $v_i$ or $v_j$ can be interior vertices of $\beta_1$ or $\beta_2$ as these branches, if distinct, are independent. Hence, $\beta_1=\beta_2=v_i,v_j$ or $e_{ij}$ belongs to at most one branch of $\mathcal{B}_S(\mathcal{X}_S(G))$. This verifies property (i)

Since $S$ is a complete structural set of $G$ then $\omega(e_{ii})=0$ for each $v_i\in V-S$. This implies that, for any $\beta=v_1,\dots,v_m\in\mathcal{B}_S(G)$ the weight sequence $\Omega_{G}(\beta)$ has the form
$$\Omega_{G}(\beta)=\omega(e_{12}),\dots,\omega(e_{k,k-1}),0,\omega(e_{k,k+1}),\dots,\omega(e_{m-1,m}).$$
As $\mathcal{B}_S(G)\simeq\mathcal{B}_S(\mathcal{X}_S(G))$ property (ii) holds.
\end{proof}

\begin{center}
\textbf{The $\mathcal{Y}$-Construction: Branch Reweighting}
\end{center}

Given a branch expansion $\mathcal{X}_S(G)$ we construct a new graph $\mathcal{Y}_S(G)$ by reweighting the branches $\mathcal{B}_S(\mathcal{X}_S(G))$ of $\mathcal{X}_S(G)$. The idea behind this construction is to reweight the branches $\mathcal{B}_S(\mathcal{X}_S(G))$ in such a way that it preserves their branch products.

Suppose the expansion $\mathcal{X}_S(G)=(\mathcal{V},\mathcal{E},\mu)$ where $S\in st_0(G)$. Let the graph $\mathcal{Y}_S(G)=(\mathcal{V},\mathcal{E},\nu)$. That is, $\mathcal{Y}_S(G)$ has the same vertex and edge set as $\mathcal{X}_S(G)$ but possibly different edge weights. We note this implies that $S$ is a complete structural set of $\mathcal{Y}_S(G)$ and moreover that the branch set $\mathcal{B}_S(\mathcal{Y}_S(G))$ is identical to $\mathcal{B}_S(\mathcal{X}_S(G))$.

For the branch $\beta=v_1,\dots,v_m\in\mathcal{B}_S(\mathcal{Y}_S(G))$ define the weight sequence of $\beta$ to be
\begin{equation}\label{eq.new}
\displaystyle{\Omega_{\mathcal{Y}_S(G)}(\beta)=\prod_{k=1}^{m-1}\mu(e_{k,k+1}),0,\dots,1,0,1,\dots,0,1}
\end{equation}
if $m>1$. If $m=1$ let $\Omega_{\mathcal{Y}_S(G)}(\beta)=\mu(e_{11})$. As $\mathcal{X}_S(G)$ and $\mathcal{Y}_S(G)$ have the same vertex and edge set then lemma \ref{newest} implies that each edge of $\mathcal{Y}_S(G)$ belongs to exactly one branch of $\mathcal{B}_S(\mathcal{Y}_S(G))$. Therefore, equation (\ref{eq.new}) completely specifies the edge weights of the graph $\mathcal{Y}_S(G)$.

In particular, the edge $e_{ij}\in\mathcal{E}$ in $\mathcal{Y}_S(G)$ has weight $\nu(e_{ij})=1$ unless $e_{ij}$ is the first edge of a branch in $\mathcal{B}_S(\mathcal{Y}_S(G))$. If $e_{ij}$ happens to be the first edge of the branch $\beta\in\mathcal{B}_S(\mathcal{Y}_S(G))$ then its weight is the product of the nonzero entries of $\Omega_{\mathcal{X}_S(G)}(\beta)$.

Considering part (ii) of lemma \ref{newest}, $\mathcal{Y}_S(G)$ is effectively the graph $\mathcal{X}_S(G)$ in which the first edge weight of each branch is the product of that branches original weights in $\mathcal{X}_S(G)$. Every other edge of $\mathcal{Y}_S(G)$ is given weight 1. (Figure \ref{fig10} gives an example of this branch reweighting.)

\begin{center}
\textbf{The $\mathcal{Z}$-Construction: Branch Merging}
\end{center}

From the branch reweighting $\mathcal{Y}_S(G)$ we construct the graph $\mathcal{Z}_S(G)$. The major idea behind this construction is that the branches $\mathcal{B}_S(\mathcal{Y}_S(G))$ of the graph $\mathcal{Y}_S(G)$ can be merged together in a way that maintains the weight set of the graph as well as its nonzero spectrum.

For $G=(V,E,\omega)$ suppose $S=\{v_1,\dots,v_m\}$ is a structural set of $G$. Let
$$\mathcal{B}_j(G;S)=\bigcup_{1\leq i\leq m}\mathcal{B}_{ij}(G;S).$$ That is, $\mathcal{B}_j(G;S)$ are the branches in $\mathcal{B}_S(G)$ terminating at the vertex $v_j$. For any $\beta=u_1,\dots,u_k$ in $\mathcal{B}_S(G)$ let $|\beta|=k$, i.e. the number of vertices in $\beta$. Moreover, let $\{\beta\}_{int}$ denote the set of interior vertices of $\beta$.

Suppose the graph $\mathcal{Y}_S(G)=(\mathcal{V},\mathcal{E},\nu)$ and $S=\{v_1,\dots,v_m\}\in st_0(G)$. For each $v_j\in S$ select a branch $\beta^j\in\mathcal{B}_j(\mathcal{Y}_S(G))$ with the property that $|\beta^j|\geq |\beta|$ for all branches $\beta\in\mathcal{B}_S(\mathcal{Y}_S(G))$. We set $\beta^j=\emptyset$ if $\mathcal{B}_j(\mathcal{Y}_S(G))=\emptyset$.

Denoting $B=\mathcal{B}_S(\mathcal{Y}_S(G))-\{\beta^1,\dots,\beta^m\}$ let
\begin{equation}\label{eqeq}
\mathcal{U}=\bigcup_{\beta\in B}\{\beta\}_{int}.
\end{equation}
That is, $\mathcal{U}$ is the set of interior vertices of the branches $\beta\neq\{\beta^1,\dots,\beta^m\}$ in $\mathcal{B}_S(\mathcal{Y}_S(G))$. Let the graph $$\mathcal{Z}^{\prime}_S(G)=\mathcal{Y}_S(G)|_{\mathcal{V}-\mathcal{U}}.$$
Recall that the branches of $\mathcal{B}_S(\mathcal{Y}_S(G))$ are pairwise independent. Therefore, the graph $\mathcal{Z}^{\prime}_S(G)=(\mathcal{V}-\mathcal{U},\mathcal{E}^\prime,\nu^\prime)$ where $e\in\mathcal{E}^\prime$ if and only if $e$ belongs to some $\beta^j$. Furthermore, the edge weights of $\mathcal{Z}^\prime_S(G)$ are given by the restriction $\nu^\prime=\nu|_{\mathcal{E}^\prime}$.

If $e$ is an edge from the vertex $a$ to the vertex $b$ we will denote this by $e=(a,b)$. Suppose the branch $\beta^j=v_1^j,\dots,v_k^j$. For each $$\beta\in\mathcal{B}_{ij}(\mathcal{Y}_S(G);S)-\beta^j$$
we add an edge $(v_i,v_{k-|\beta|+2}^j)$ to the graph $\mathcal{Z}_S^\prime(G)$. The edge $(v_i,v_{k-|\beta|+2}^j)$ is given the weight of the first edge belonging to $\beta$ in $\mathcal{Y}_S(G)$. If this is done over all $1\leq i,j\leq m$ we call the resulting graph $\mathcal{Z}^{\prime\prime}_S(G)$.

It is important to note that the graph $\mathcal{Z}^{\prime\prime}_S(G)$ may have parallel edges. By \textit{parallel edges} we mean that there may be multiple edges in the edges set of $\mathcal{Z}^{\prime\prime}_S(G)$ of the form $(a,b)$. In particular, if two branches $\beta_1,\beta_2\in\mathcal{B}_{ij}(\mathcal{Y}_S(G);S)$ have the same length, i.e. $|\beta_1|=|\beta_2|=\ell$, then there are (at least) two edges in $\mathcal{Z}^{\prime\prime}_S(G)$ of the form $(v_i,v_{\ell-|\beta|+2}^j)$.

Suppose $\mathcal{Z}_S^{\prime\prime}(G)$ has parallel edges $e_1,\dots,e_N$ of the form $(v_i,v_j)$ with weights $w_1,\dots,w_N$. We replace the edges $e_1,\dots,e_N$ in $\mathcal{Z}_S^{\prime\prime}(G)$ with the single edge $e_{ij}$ having weight $w_1+\dots+w_N$. If this is done for each set of parallel edges in $\mathcal{Z}_S^{\prime\prime}(G)$ we denote the resulting graph by $\mathcal{Z}_S(G)$.

Note that our construction of $\mathcal{Z}_S(G)$ depends on the initial choice of each $\beta_j$. We therefore write $\mathcal{Z}_S(G)=\mathcal{Z}_S(G;\beta^1,\dots,\beta^m)$.

\begin{theorem}\label{theoremunital}
Let $G=(V,E,\omega)$ and $S\in st_0(G)$. Suppose the edge weights of $G$ are in the semiring $\mathbb{U}\subseteq\mathbb{W}[\lambda]$. Then $\mathcal{Z}_S(G)$ has edge weights in $\mathbb{U}$. Moreover, $\sigma_0(\mathcal{Z}_S(G))=\sigma_0(G)$
\end{theorem}

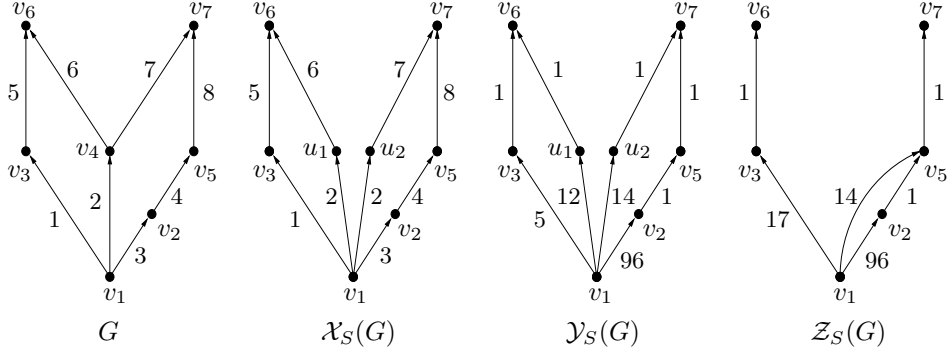
\begin{figure}
  \begin{center}
    \begin{overpic}[width=4.75in, height=1.65in]{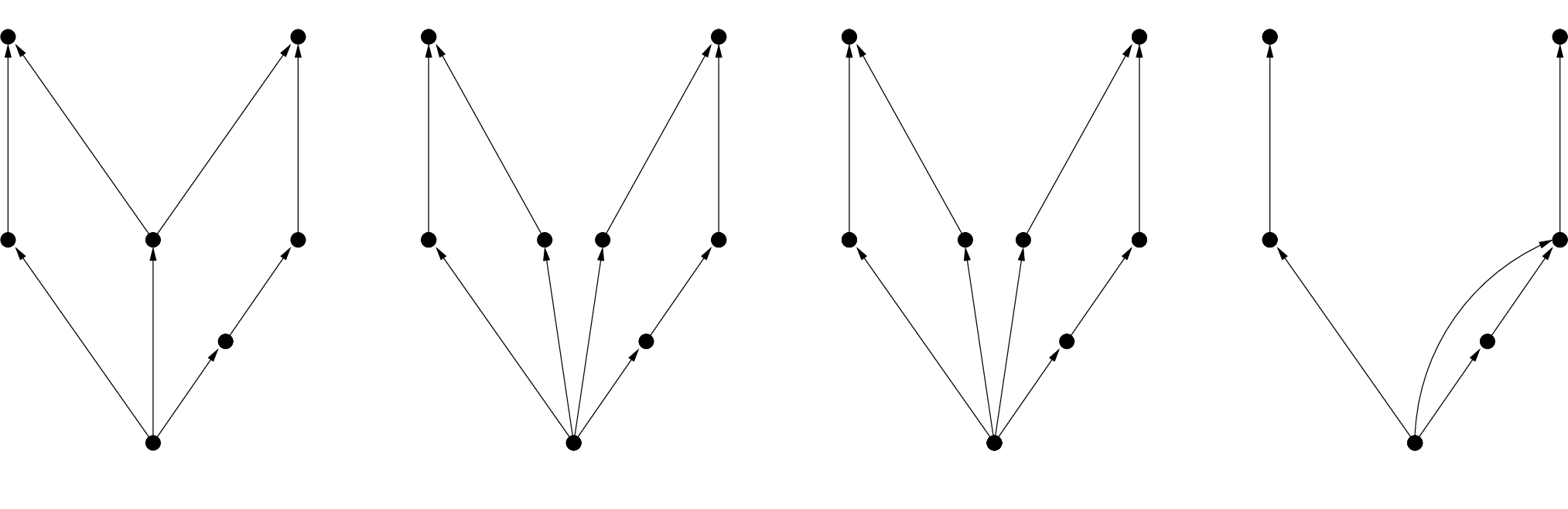}
    \put(9,2){$v_1$}
    \put(35.5,2){$v_1$}
    \put(62.5,2){$v_1$}
    \put(89.5,2){$v_1$}

    \put(15,9){$v_2$}
    \put(41.5,9){$v_2$}
    \put(68.5,9){$v_2$}
    \put(95.5,9){$v_2$}

    \put(-1.5,15.5){$v_3$}
    \put(25.5,15.5){$v_3$}
    \put(52.5,15.5){$v_3$}
    \put(79,15.5){$v_3$}

    \put(6,18){$v_4$}

    \put(19,15.5){$v_5$}
    \put(45.5,15.5){$v_5$}
    \put(72.5,15.5){$v_5$}
    \put(99.5,15.5){$v_5$}

    \put(-1,33.5){$v_6$}
    \put(25.5,33.5){$v_6$}
    \put(52.5,33.5){$v_6$}
    \put(80,33.5){$v_6$}

    \put(18.5,33.5){$v_7$}
    \put(45,33.5){$v_7$}
    \put(72,33.5){$v_7$}
    \put(99,33.5){$v_7$}

    \put(31,18){$u_1$}
    \put(39.5,18){$u_2$}
    \put(58,18){$u_1$}
    \put(66.5,18){$u_2$}

    \put(3,10){\small$1$}
    \put(7.5,12){\small$2$}
    \put(12.5,6){\small$3$}
    \put(16.5,12.5){\small$4$}
    \put(-1.5,24){\small$5$}
    \put(5,26.5){\small$6$}
    \put(13.5,26.5){\small$7$}
    \put(20,24){\small$8$}

    \put(29.5,10){\small$1$}
    \put(33.5,12.5){\small$2$}
    \put(38.5,12.5){\small$2$}
    \put(39.5,6){\small$3$}
    \put(43,12.5){\small$4$}
    \put(25,24){\small$5$}
    \put(31.5,26.5){\small$6$}
    \put(41,26.5){\small$7$}
    \put(46.5,24){\small$8$}

    \put(56.5,10){\small$5$}
    \put(59,12.5){\small$12$}
    \put(65,12.5){\small$14$}
    \put(66,5.5){\small$96$}
    \put(70.5,12.5){\small$1$}
    \put(52,24){\small$1$}
    \put(58.5,26.5){\small$1$}
    \put(67.5,26.5){\small$1$}
    \put(73.5,24){\small$1$}

    \put(82,10){\small$17$}
    \put(93,5.5){\small$96$}
    \put(97.5,12.5){\small$1$}
    \put(79,24){\small$1$}
    \put(100.5,24){\small$1$}
    \put(89.5,12.5){\small$14$}

    \put(8.5,-2.5){$G$}
    \put(33,-2.5){$\mathcal{X}_S(G)$}
    \put(60,-2.5){$\mathcal{Y}_S(G)$}
    \put(87,-2.5){$\mathcal{Z}_S(G)$}
    \end{overpic}
  \end{center}
  \caption{The isospectral transformation $\mathcal{Z}_S(G)$ of $G$ over $S=\{v_1,v_2,v_3\}$.}\label{fig10}
\end{figure}

\begin{example}\label{ex12}
Let $G=(V,E,\omega)$ be the graph shown in figure \ref{fig10} (far left). Note that the vertex set $S=\{v_1,v_2,v_3\}$ is a complete structural set of $G$ since $G$ has no cycles (including loops). The branches of $G$ with respect to $S$ are then $\beta_1=v_1,v_3,v_6$, $\beta_2=v_1,v_4,v_6$, $\beta_3=v_1,v_4,v_7$, and $\beta_4=v_1,v_2,v_5,v_7$. Observe that only $\beta_2$ and $\beta_3$ share an interior vertex.

The graph $\mathcal{X}_S(G)$, shown in figure \ref{fig10} (middle left), is then the graph in which the branches $\beta_2$ and $\beta_3$ have been replaced with the independent branches $\tilde{\beta}_2=v_1,u_1,v_6$ and $\tilde{\beta}_3=v_1,u_2,v_7$ respectively. Thus, the expansion $\mathcal{X}_S(G)$ has the vertex set $\mathcal{V}=\{v_1,v_2,v_3,v_5,v_6,v_7,u_1,u_2\}$.

Note that the products of the weights along the branches $\beta_1$, $\tilde{\beta}_2$, $\tilde{\beta}_3$, $\beta_4$ are $5, 12, 14, 96$ respectively. Hence, these are the first edge weights of each of the branches $\beta_1$, $\tilde{\beta}_2$, $\tilde{\beta}_3$, $\beta_4$ in $\mathcal{Y}_S(G)$ respectively (see figure \ref{fig10}, middle right). Each other edge weight of $\mathcal{Y}_S(G)$ is $1$.

To construct $\mathcal{Z}^\prime_S(G)$ note that $\beta^1=\emptyset$, $\beta^7=\beta_4$, and $\beta^6$ is either $\beta_1$ or $\tilde{\beta}_2$. Here we make the arbitrary choice of letting $\beta^6=\beta_1$. As the sets $\{\tilde{\beta}_2\}_{int}=\{u_1\}$ and $\{\tilde{\beta}_3\}_{int}=\{u_2\}$ then $\mathcal{Z}^\prime_S(G)$ has vertex set $\mathcal{V}-\{u_1,u_2\}=\{v_1,v_2,v_3,v_5,v_6,v_7\}$. Since the branch set $$\mathcal{B}_S(\mathcal{Z}^{\prime}_S(G))-\{\beta^6,\beta^7\}=\{\tilde{\beta}_2,\tilde{\beta}_3\}$$ then constructing $\mathcal{Z}_S^{\prime\prime}(G)$ amounts to adding two edges to $\mathcal{Z}_S^{\prime}(G)$; one for $\tilde{\beta}_2$ and one for $\tilde{\beta}_3$.

Note that $\tilde{\beta}_2\in\mathcal{B}_{16}(\mathcal{Y}_S(G);S)$ with first edge weight $12$ and $|\tilde{\beta}_2|=3$. As $|\beta^6|=3$ then to the graph $\mathcal{Z}_S^{\prime}(G)$ we add a parallel edge from $v_1$ to $v_3$ with weight $12$. For the branch $\tilde{\beta}_3\in\mathcal{B}_{17}(\mathcal{Y}_S(G);S)$ note that it has first edge weight $14$ and $|\tilde{\beta}_2|=3$. Since $|\beta^7|=4$ then to the graph $\mathcal{Z}_S^{\prime}(G)$ we add an edge from $v_1$ to $v_5$ with weight $14$. The result is the graph $\mathcal{Z}_S^{\prime\prime}(G)$.

Note that the two parallel edges from $v_1$ to $v_3$ in $\mathcal{Z}_S^{\prime\prime}(G)$ have weights $5$ and $12$ respectively. Hence, in $\mathcal{Z}_S(G)=\mathcal{Z}_S(G;\beta_1,\beta_4)$ shown in figure \ref{fig10} (far right) the edge $e_{13}$ has weight $17$. Moreover, the edge weights of $G$ are in the semiring of nonnegative integers $\mathbb{Z}^+$. Therefore, as guaranteed by theorem \ref{theoremunital}, the graph $\mathcal{Z}_S(G)$ also has edge weights in $\mathbb{Z}^+$. One can also compute that $\sigma_0(\mathcal{Z}_S(G)=\sigma_0(G)$.
\end{example}

For a graph $G=(V,E,\omega)$ let $|G|=|V|$, i.e. the number of vertices in $V$.

\begin{proposition}\label{prop1}
For $G=(V,E,\omega)$ let $S=\{v_1,\dots,v_m\}$ be a complete structural set of $G$. Then $$|\mathcal{Z}_S(G)|=|S|+\sum_{i=1}^m(|\beta^j|-2)$$
where $|\beta^j|=2$ if $\beta^j=\emptyset$.
\end{proposition}

\begin{proof}
Let $\mathcal{Y}_S(G)=(\mathcal{V},\mathcal{U},\nu)$. Suppose $\alpha,\beta\in\mathcal{B}_S(\mathcal{Y}_S(G))$ and $\alpha\neq\beta$. Note that the branches of $\mathcal{B}_S(\mathcal{Y}_S(G))$ are pairwise independent and each vertex of $\mathcal{V}$ belongs to a branch of $\mathcal{B}_S(\mathcal{Y}_S(G))$. Hence, $\{\alpha\}_{int}\cap\{\beta\}_{int}=\emptyset$ and $\{\beta\}_{int}\cap S=\emptyset$. Therefore, $\mathcal{V}$ is the disjoint union
$$\mathcal{V}=S\cup\big(\bigcup_{\beta\in\mathcal{B}_S(\mathcal{Y}_S(G))}\{\beta\}_{int}\big).$$

Recall that the set $\mathcal{U}$, given by (\ref{eqeq}), is the set of interior vertices of the branches $\mathcal{B}_S(\mathcal{Y}_S(G))-\{\beta_1,\dots,\beta_m\}$.  Thus, the vertex set
$$\mathcal{V}-\mathcal{U}=S\cup\big(\bigcup_{i=1}^m\{\beta^j\}_{int}\big).$$
where each $\{\beta^i\}_{int}\cap\{\beta^j\}_{int}=\emptyset$ and $\{\beta^i\}_{int}\cap S=\emptyset$ for $i\neq j$. Therefore, $$|\mathcal{V}-\mathcal{U}|=|S|+\sum_{i=1}^m(|\beta^j|-2)$$
since each branch $\beta^j$ has $|\beta^j|-2$ interior vertices. If $\beta^j=\emptyset$ then it has no interior vertices which is compensated for by assuming $|\beta^j|=2$.

Note that $|\mathcal{Z}_S^\prime S(G)|=|\mathcal{Y}_S(G)|_{\mathcal{V}-\mathcal{U}}|=|\mathcal{V}-\mathcal{U}|$.
As constructing $\mathcal{Z}^{\prime\prime}_S(G)$ from $\mathcal{Z}^\prime_G(S)$ only involves the addition of edges and constructing $\mathcal{Z}_S(G)$ from $\mathcal{Z}_S^{\prime\prime}G(S)$ the removal of parallel edges it follows that $|\mathcal{Z}_S(G)|=|\mathcal{V}-\mathcal{U}|$. This completes the proof.
\end{proof}

In example \ref{ex12} we consider the $\mathcal{Z}$-construction of the graph $G$ in figure \ref{fig10} over the complete structural set $S=\{v_1,v_6,v_7\}$. From proposition \ref{prop1} it follows that $|\mathcal{Z}_S(G)|=|S|+|\beta^1|+|\beta^6|+|\beta^7|$. As $\beta^1=\emptyset$, $\beta^6=v_1,v_3,v_6$, and $\beta^7=v_1,v_2,v_5,v_7$ then $|\mathcal{Z}_S(G)|=6$.

Note that in this example $|G|>|\mathcal{Z}_S(G)|$. Hence, $\mathcal{Z}_S(G)$ can be considered to be a \textit{reduction} of $G$ over the weight set $\mathbb{Z}^+$ having the same nonzero spectrum.

\section{Proofs}
In this section we give proofs of the theorems and lemmas in sections 3 and 4. Let $G\in\mathbb{G}$. If $S\in st(G)$ then, as previously noted, the vertices of $G$ can be ordered such that $M(G)-\lambda I$ has the block form
\begin{equation}\label{eq3.1}
M(G)-\lambda I=\left( \begin{array}{cc}
A & B\\
C & D
\end{array}
\right)
\end{equation}
where  $A$ corresponds to the vertices in $\bar{S}$ and is a triangular matrix with nonzero diagonal (see \textit{Frobenius normal form} in \cite{Brualdi91}). The matrix $A$ can be made triangular from the fact that $G|_{\bar{S}}$ contains no (nonloop) cycles. The nonzero diagonal follows from the assumption that $\omega(e_{ii})\neq \lambda$ for $v_i\in \bar{S}$. As a triangular matrix with nonzero diagonal has a nonzero determinant then the matrix $A$ is invertable.

Using the matrix identity
\begin{equation}\label{eq1.2}
 \det\left( \begin{array}{cc}
 A & B\\
 C & D
 \end{array}
 \right)=\det(A)\cdot\det(D-CA^{-1}B)
\end{equation}
it follows that
\begin{equation*}
 \det(D-CA^{-1}B)=\frac{\det(M(G)-\lambda I)}{\det(A)}.
\end{equation*}

As $D-CA^{-1}B=R-\lambda I$ for some $R\in\mathbb{W}[\lambda]^{|S|\times |S|}$ then the claim is that the isospectral reduction $\mathcal{R}_S(G)$ is the graph with adjacency matrix $R$ (see proof of lemma \ref{lemma1.2}). Furthermore, as $A$ corresponds to the vertex set $\bar{S}$ then by equation (\ref{eq0.2})
$$\det(A)=\prod_{v_i\in\bar{S}}\big(\omega(e_{ii})-\lambda\big).$$ Our first task in this section is to make this precise.

\begin{definition}
Let $M\in\mathbb{W}[\lambda]^{n\times n}$ and $\mathcal{I}\subseteq\{1,\dots,n\}$ be nonempty. If the set $\mathcal{I}=\{i_1,\dots,i_m\}$ where $i_j<i_{j+1}$ then the matrix $A$ given by $A_{k\ell}=M_{i_ki_\ell}$ for $1\leq k,\ell\leq m\leq n$ is the \textit{principle submatrix} of $M$ indexed by $\mathcal{I}$.
\end{definition}

In what follows, we will use the convention that if $\mathcal{I}=\{i_1,\dots,i_m\}$ is an indexing set of a matrix then $i_j<i_{j+1}$. Moreover, if $D$ the principle submatrix of the matrix $M\in\mathbb{W}[\lambda]^{n\times n}$ indexed by $\mathcal{I}$ then there exists a unique permutation matrix $P$ such that
$$PMP^{-1}=\left( \begin{array}{cc}
A & B\\
C & D\\
\end{array}
\right)$$
where $A$ is the principle submatrix of $M$ indexed by $\bar{\mathcal{I}}=\{1,\dots,n\}-\mathcal{I}$.

Assuming $A$ is invertable, then by (\ref{eq1.2})
$$\det(M)=\det(PMP^{-1})=\det(A)\det(D-CA^{-1}B).$$

If this is the case we call the matrix $D-CA^{-1}B$ the \textit{reduction} of $M$ over $\mathcal{I}$ and write $r(M;\mathcal{I})=D-CA^{-1}B$. In the case that $\mathcal{I}=\{1,\dots,n\}$ we let the matrix $r(M;\mathcal{I})=M$. Moreover, if $r(M;\mathcal{I})$ can be reduced over the index set $\mathcal{J}$ we write the reduction of $r(M;\mathcal{I})$ over $\mathcal{J}$ by $r(M;\mathcal{I},\mathcal{J})$, and so on.

\begin{lemma}\label{lemma1.1}
Let $M\in\mathbb{W}[\lambda]^{n\times n}$ and $\mathcal{I}_m\subseteq\{1,\dots, n\}$ be nonempty. If the principle submatrix of $M$ indexed by $\bar{\mathcal{I}}_m$ is upper triangular with nonzero diagonal and the sets $\mathcal{I}_m\subseteq\mathcal{I}_{m-1}\subseteq\dots\subseteq\mathcal{I}_1\subseteq\{1,\dots,n\}$ then   $r(M;\mathcal{I}_1,\dots,\mathcal{I}_{m-1},\mathcal{I}_m)=r(M;\mathcal{I}_m)$.
\end{lemma}

\begin{proof}
Without loss in generality suppose $\mathcal{I}_m\subseteq\{1,\dots,n\}$ such that the matrix $M$ has block form
$$M=\left( \begin{array}{cc}
A & B\\
C & D\\
\end{array}
\right)$$
where $A$ is an upper triangular principle submatrix of $M$ indexed by $\bar{\mathcal{I}}_m$ with nonzero diagonal.
Then for $\mathcal{I}_m\subseteq\mathcal{I}_1\subseteq\{1,\dots,n\}$ there exists a unique permutation matrix $P$ such that
$$PMP^{-1}=\left( \begin{array}{ccc}
A_0 & B_0 & B_2\\
C_0 & A_1 & B_1\\
C_2 & C_1 & D
\end{array}
\right)$$
where $A_0$, $A_1$, and $D$ are the principle submatrices indexed by $\{1,\dots,n\}-\mathcal{I}_1,\mathcal{I}_1-\mathcal{I}_m$, and $\mathcal{I}_m$ respectively. Moreover, both $A_0$ and $A_1$ are upper triangular matrices with nonzero diagonal as they are principle submatrices of $A$ and are therefore invertable.

By partitioning the matrix $PMP^{-1}$ into the blocks
$$PMP^{-1}=\left( \begin{array}{cc}
\left[\begin{array}{c}
A_0
\end{array}\right]
&
\left[\begin{array}{cc}
B_0 & B_2
\end{array}\right]\\
\left[\begin{array}{c}
C_0\\
C_2
\end{array}\right]&
\left[\begin{array}{cc}
A_1 & B_1\\
C_1 & D
\end{array}\right]
\end{array}
\right)$$
it then follows that
\begin{equation}
r(M;\mathcal{I}_1)=
\left( \begin{array}{cc}
A_1-C_0A_0^{-1}B_0 & B_1-C_0A_0^{-1}B_2\\
C_1-C_2A_0^{-1}B_0 & D-C_2A_0^{-1}B_2\\
\end{array}
\right).
\end{equation}

The claim is then that the matrix $C_0A_0^{-1}B_0$ is an upper triangular matrix possessing a zero diagonal. To see this let the sets $\{1,\dots,n\}-\mathcal{I}_1=\{i_1,\dots,i_s\}$, $\mathcal{I}_1-\mathcal{I}_m=\{j_1,\dots,j_t\}$, and $(A_0^{-1})=\alpha_{ij}$. As the inverse of an upper triangular matrix is upper triangular $\alpha_{ij}=0$ for $i>j$. Hence,
$$(C_0A^{-1}_0B_0)_{\ell k}=\sum_{p=1}^{s}\sum_{q=1}^{p}\big(A_{j_{\ell}i_q}\alpha_{qp}A_{i_pj_k}\big) \ \ \ 1\leq \ell,k\leq t.$$
As $1\leq q\leq p$ then each $i_q\leq i_p$. For each $k\leq \ell$ note that similarly $j_k\leq j_\ell$. Moreover, if $j_k<i_p$ then $A_{i_pj_k}=0$ since $A$ is upper triangular. If $j_k>i_p$ then $i_q<j_\ell$ implying $A_{j_\ell i_q}=0$ for the same reason. As $j_k\neq i_p$, it then follows that
$$(C_0A^{-1}_0B_0)_{\ell k}=0 \ \text{for all} \ \ k\leq \ell$$
verifying the claim.

$A_1-C_0A_0^{-1}B_0$ is therefore an upper triangular matrix with nonzero diagonal implying $A_1-C_0A_0^{-1}B_0$ is invertible and $r(M;\mathcal{I}_1)$ can be reduced over $\mathcal{I}_m$. In particular,
\begin{equation}\label{eq2.1}
r(M;\mathcal{I}_1,\mathcal{I}_m)=D-C_2A_0^{-1}B_2-\big(C_1-C_2A_0^{-1}B_0\big)\Gamma\big(B_1-C_0A_0^{-1}B_2\big)
\end{equation}
where $\Gamma=(A_1-C_0A_0^{-1}B_0)^{-1}$.

To verify that $r(M;\mathcal{I}_1,\mathcal{I}_m)=r(M;\mathcal{I}_m)$ note that $M$ has block form
$$M=\left( \begin{array}{cc}
A & B\\
C & D\\
\end{array}
\right) \ \ \text{and} \ \
QAQ^{-1}=\left[ \begin{array}{cc}
A_0 & B_0\\
C_0 & A_1\\
\end{array}
\right]$$
where $Q$ is the principle submatrix of $P$ indexed by $\bar{\mathcal{I}}_m$. Therefore,
$$A^{-1}=Q^{-1}\left[ \begin{array}{cc}
A_0 & B_0\\
C_0 & A_1\\
\end{array}
\right]^{-1}Q=Q^{-1}\left[ \begin{array}{cc}
(A_0-B_0A_1^{-1}C_0)^{-1} & -A_0^{-1}B_0\Gamma^{-1}\\
-\Gamma^{-1}C_0A_0^{-1} & \Gamma^{-1}\\
\end{array}
\right]Q.$$
Note the matrix $(A_0-B_0A_1^{-1}C_0)^{-1}=A_0^{-1}-A_0^{-1}B_0\Gamma C_0A_0^{-1}$ by the Woodbury matrix identity ( see \cite{Hager89}) and is therefore well defined. From this
\begin{align*}
r(M;\mathcal{I}_m)=& D-CQ^{-1}
\left[\begin{array}{cc}
A_0 & B_0\\
C_0 & A_1\\
\end{array}
\right]^{-1}
QB\\
=& D-\left[ \begin{array}{cc}
C_2 & C_1
\end{array}
\right]
\left[ \begin{array}{cc}
(A_0-B_0A_1^{-1}C_0)^{-1} & -A_0^{-1}B_0\Gamma^{-1}\\
-\Gamma^{-1}C_0A_0^{-1} & \Gamma^{-1}\\
\end{array}
\right]
\left[ \begin{array}{c}
B_2\\ B_1
\end{array}
\right]\\
=& D-C_2(A_0^{-1}-A_0^{-1}B_0\Gamma C_0A_0^{-1})B_2-C_1\Gamma C_0A_0^{-1}B_2-\\
& \hspace{1.85in} C_2A_0^{-1}B_0\Gamma C_0A_0^{-1}B_1+C_1\Gamma B_1.
\end{align*}
From (\ref{eq2.1}) it follows that $r(M;\mathcal{I}_m)=r(M;\mathcal{I}_1,\mathcal{I}_m)$

For $\mathcal{I}_m\subseteq\mathcal{I}_2\subseteq\mathcal{I}_1$ the same argument can be repeated to show
$$r\big(r(M;\mathcal{I}_1);\mathcal{I}_m\big)=r\big(r(M;\mathcal{I}_1);\mathcal{I}_2;\mathcal{I}_m\big)$$
as $r(M;\mathcal{I}_1)$ has the same form as $M$, i.e. its upper left hand block is an upper triangular matrix with nonzero diagonal. Hence, $r(M;\mathcal{I}_m)=r(M;\mathcal{I}_1,\mathcal{I}_2;\mathcal{I}_m)$. The lemma follows by further extending $\mathcal{I}_m\subseteq\mathcal{I}_2\subseteq\mathcal{I}_1\subseteq\{1,\dots,n\}$ to the nested sequence $\mathcal{I}_m\subseteq\mathcal{I}_{m-1}\subseteq\dots\subseteq\mathcal{I}_1\subseteq\{1,\dots,n\}$ by repeated use of the above argument.
\end{proof}

To establish that graph reductions are analogous to matrix reductions we prove the following lemma.

\begin{lemma}\label{lemma1.2}
If $S=\{v_{m+1},\dots,v_{n}\}$ is a structural set of $G\in\mathbb{G}$ then $\mathcal{R}_S(G)$ is the graph with adjacency matrix $r\big(M(G)-\lambda I;\{m+1,\dots,n\}\big)+\lambda I$.
\end{lemma}

\begin{proof}
Let $\bar{S}=\{v_1,\dots,v_m\}$. From the discussion in the first paragraph of section 5, if $S$ is a structural set of $G=(V,E,\omega)$ then $M(G)-\lambda I$ has the block form
$$M(G)-\lambda I=\left( \begin{array}{cc}
A & B\\
C & D
\end{array}
\right)$$
where $A$ is the principle submatrix of $M(G)-\lambda I$ indexed by $\{1,\dots,m\}$ and is an upper triangular matrix with nonzero diagonal.

With this in mind, let $V_k=\{v_{k+1},\dots,v_n\}$ and  $\mathcal{I}_k=\{k+1,\dots,n\}$ for $1\leq k\leq m$.
If $M(G)_{ij}=\omega_{ij}$ for $1\leq i,j\leq n$ then, proceeding by induction, for $k=1$
$$r(M(G)-\lambda I;\mathcal{I}_1)=\big(D_1-[\omega_{21},\dots,\omega_{n1}]^T\frac{1}{\omega_{11}-\lambda}[\omega_{12},\dots,\omega_{1n}]\big)$$
where $D_1$ is the principle submatrix of $M(G)-\lambda I$ indexed by $\mathcal{I}_1$. Hence,
\begin{equation}\label{eq2.5}
r(M(G)-\lambda I;\mathcal{I}_1)_{ij}=
\begin{cases}
\omega_{ij}+\omega_{i1}\omega_{1j}/(\lambda-\omega_{11}) & i\neq j\\
\omega_{ij}+\omega_{i1}\omega_{1j}/(\lambda-\omega_{11})-\lambda & i=j
\end{cases}
\end{equation}
for $2\leq i,j\leq n$.

Conversely, note that the set $\mathcal{B}_{ij}(G;V_1)$ consists of at most the two branches $\beta^+=v_i,v_1,v_j$ and $\beta^-=v_i,v_j$ i.e. the branches from $v_i$ to $v_j$ with and without interior vertex $v_1$. As $\mathcal{P}_\omega(\beta^-)=\omega_{ij}$ and $\mathcal{P}_\omega(\beta^+)=\omega_{i1}\omega_{1j}/(\lambda-\omega_{11})$ if the branches $\beta^-$ and $\beta^+$ are respectively in $\mathcal{B}_{ij}(G;V_1)$ it then follows from (\ref{eq1.0}) and (\ref{eq2.5}) that $r\big(M(G)-\lambda I;\mathcal{I}_1\big)=M\big(\mathcal{R}_{V_1}(G)\big)-\lambda I$.

Suppose then that $r\big(M(G)-\lambda I;\mathcal{I}_{k-1}\big)=M\big(\mathcal{R}_{V_{k-1}}(G)\big)-\lambda I$ for some $k\leq m$. Since the principle submatrix of $M(G)-\lambda I$ indexed by $\mathcal{I}_{k}$ is upper triangular with nonzero diagonal and $\mathcal{I}_k\subseteq\mathcal{I}_{k-1}\subseteq\{1,\dots,n\}$ an application of lemma \ref{lemma1.1} implies that $r(M(G)-\lambda I;\mathcal{I}_k)=r(M(G)-\lambda I;\mathcal{I}_{k-1},\mathcal{I}_k)$. Hence,
$$r(M(G)-\lambda I;\mathcal{I}_k)=r\big(M(\mathcal{R}_{V_{k-1}}(G))-\lambda I;\mathcal{I}_k\big).$$

Letting $\mathcal{M}=M(\mathcal{R}_{V_{k-1}}(G))$ then by the argument above
$$r(M(G)-\lambda I;\mathcal{I}_k)_{ij}=
\begin{cases}
\mathcal{M}_{ij}+\mathcal{M}_{ik}\mathcal{M}_{kj}/(\lambda-\mathcal{M}_{kk}) & i\neq j\\
\mathcal{M}_{ij}+\mathcal{M}_{ik}\mathcal{M}_{kj}/(\lambda-\mathcal{M}_{kk})-\lambda & i=j
\end{cases}$$
for $k+1\leq i,j\leq n$. As $\mathcal{M}_{ij}=M(\mathcal{R}_{V_{k-1}}(G))_{ij}$ then the entries of $\mathcal{M}$ are given by
$\mathcal{M}_{ij}=\sum_{\beta\in\mathcal{B}_{ij}(G;V_{k-1})}\mathcal{P}_\omega(\beta)$.

Observe that
\begin{equation*}
\sum_{\beta\in\mathcal{B}_{ij}(G;V_k)}\mathcal{P}_\omega(\beta)=\sum_{\beta\in\mathcal{B}^-_{ij}(G;V_k)}\mathcal{P}_\omega(\beta)+\sum_{\beta\in\mathcal{B}^+_{ij}(G;V_k)}\mathcal{P}_\omega(\beta)
\end{equation*}
where $\mathcal{B}^+_{ij}(G;V_k)$ and $\mathcal{B}^-_{ij}(G;V_k)$ are the branches in $\mathcal{B}_{ij}(G;V_k)$ that contain and do not contain the interior vertex $v_k$ respectively. It then immediately follows that $\mathcal{B}^-_{ij}(G;V_k)=\mathcal{B}_{ij}(G;V_{k-1})$ implying $\sum_{\beta\in\mathcal{B}^-_{ij}(G;V_k)}\mathcal{P}_\omega(\beta)=\mathcal{M}_{ij}$.

On the other hand, any $\beta\in\mathcal{B}^+_{ij}(G;V_{k})$ can be written as $\beta=v_{i},\dots,v_k,\dots,v_{j}$ where $\beta_1=v_{i},\dots,v_k\in\mathcal{B}_{ik}(G;V_{k-1})$ and $\beta_2=v_k,\dots,v_{j}\in\mathcal{B}_{kj}(G;V_{k-1})$. Hence, equation (\ref{eq0.9}) implies
\begin{equation}\label{eq2.6}
\sum_{\beta\in\mathcal{B}^+_{ij}(G;V_k)}\mathcal{P}_\omega(\beta)=\sum_{\beta\in\mathcal{B}^+_{ij}(G;V_k)}\frac{\mathcal{P}_\omega(\beta_1)\mathcal{P}_\omega(\beta_2)}{\lambda-\omega_{kk}}.
\end{equation}

Conversely, if $\beta_1=v_{i},\dots,v_k\in\mathcal{B}_{ik}(G;V_{k-1})$ and $\beta_2=v_k,\dots,v_{j}\in\mathcal{B}_{kj}(G;V_{k-1})$ then $\beta=v_{i},\dots,v_k,\dots,v_{j}\in\mathcal{B}^+_{ij}(G;V_{k})$. This follows from the fact that $\beta_1$ and $\beta_2$ share no interior vertices since otherwise $G|_{\bar{V}_k}$ would contain a cycle. Note that $G|_{\bar{V}_k}$ cannot contain a cycle since $V_k\in st(G)$. Therefore,
\begin{equation}\label{eq2.7}
\sum_{\beta\in\mathcal{B}^+_{ij}(G;V_k)}\mathcal{P}_\omega(\beta_1)\mathcal{P}_\omega(\beta_2)=\sum_{\beta_1\in\mathcal{B}_{ik}(G;V_{k-1})}\mathcal{P}_\omega(\beta_1)\sum_{\beta_2\in\mathcal{B}_{kj}(G;V_{k-1})}\mathcal{P}_\omega(\beta_2).
\end{equation}

Moreover, $\omega_{kk}=\mathcal{M}_{kk}$ since $\mathcal{B}_{kk}(G;V_{k-1})$ contains at most the cycle $v_k,v_k$ given that $V_{k-1}\in st(G)$. From (\ref{eq2.6}) and (\ref{eq2.7}) it then follows that
$$\mathcal{M}_{ij}+\mathcal{M}_{ik}\mathcal{M}_{kj}/(\lambda-\mathcal{M}_{kk})=\sum_{\beta\in\mathcal{B}_{ij}(G;V_{k})}\mathcal{P}_\omega(\beta)$$
implying $r\big(M(G)-\lambda I;\mathcal{I}_k\big)=M\big(\mathcal{R}_{V_k}(G)\big)-\lambda I$.

By induction we have that $r\big(M(G)-\lambda I;\mathcal{I}_m\big)=M\big(\mathcal{R}_{S}(G)\big)-\lambda I$ by setting $k=m$.
\end{proof}

If $S=\{v_{i_1},\dots,v_{i_m}\}$ is a structural set of $G\in\mathbb{G}$ then let $\{i_1,\dots,i_m\}$ be the index set associated with $S$. Hence, if $S\in st(G)$ is indexed by $\mathcal{I}$ then $\bar{S}$ is indexed by $\bar{\mathcal{I}}$ and there is a unique permutation matrix $P$ such that
$$P\big(M(G)-\lambda I\big)P^{-1}=\left( \begin{array}{cc}
A & B\\
C & D
\end{array}\right)$$
where $A$ is the principle submatrix of $M(G)-\lambda I$ indexed by $\bar{\mathcal{I}}$. Therefore, $A=M(G|_{\bar{S}})-\lambda I$ and it follows from lemma \ref{lemma1.2} and (\ref{eq1.2}) that
\begin{equation}\label{eq3.3}
\det\big(M(G)-\lambda I\big)=\det\big(M(G|_{\bar{S}})-\lambda I\big)\det\big(M(\mathcal{R}_S(G))-\lambda I\big).
\end{equation}
Let $\det(M(G|_{\bar{S}})-\lambda I)=p/q\in\mathbb{W}[\lambda]$ where $p,q\in\mathbb{C}[\lambda]$. Then theorem \ref{maintheorem} follows by observing that $\sigma(G|_{\bar{S}})$ and $\sigma^{-1}(G|_{\bar{S}})$ are the solutions to $p=0$ and $q=0$ respectively.

We now give proofs of the theorems on isospectral transformations. For theorem \ref{theorem3} we give the following.

\begin{proof}
Let $S$ be a structural set of $G=(V,E,\omega)$. Following definition \ref{def1}, the graph $G|_{\bar{S}}$ contains no cycles, except possibly loops, and if $G|_{\bar{S}}$ has loops they do not have weight $\lambda$. If $S\subseteq T\subseteq V$ then the graph $G|_{\bar{T}}$ is a subgraph of $G|_{\bar{S}}$. Therefore, $T\in st(G)$ as $G|_{\bar{T}}$ similarly contains no cycles, except possibly loops, where loops do not have weight $\lambda$. By the same argument, $S\in st(\mathcal{R}_T(G))$ since $\mathcal{R}_T(G)|_{\bar{S}}$ is a subgraph of $G|_{\bar{S}}$. Hence, $T,S$ induces a sequence of reductions on $G$.

Suppose $S_m\subseteq S_{m-1}\subseteq\dots\subseteq S_1\subseteq V$ where $S_m$ is a structural set of $G$. It then follows that $S_1,\dots,S_m$ induce a sequence of reductions on $G$. Let $\mathcal{I}_i$ be the index set associated with $S_i$ for $1\leq i\leq m$. By lemma \ref{lemma1.2} the graph $\mathcal{R}(G;S_1)$ has adjacency matrix $r(M(G)-\lambda I,\mathcal{I}_1)+\lambda I$. Hence, by another application of lemma \ref{lemma1.2}, the graph $\mathcal{R}(G;S_1,S_2)$ has adjacency matrix
$$r\big(r(M(G)-\lambda I;\mathcal{I}_1)+(\lambda-\lambda)I;\mathcal{I}_2\big)+\lambda I=r(M(G)-\lambda I;\mathcal{I}_1,\mathcal{I}_2)+\lambda I.$$
Continuing in this manner, it follows that $\mathcal{R}(G;S_1,\dots,S_m)$ has adjacency matrix given by $r(M(G)-\lambda I;\mathcal{I}_1,\dots,\mathcal{I}_m)+\lambda I$ which by lemma \ref{lemma1.1} is equivalent to the matrix $r(M(G)-\lambda I;\mathcal{I}_m)+\lambda I$.

Given that $r(M(G)-\lambda I;\mathcal{I}_m)+\lambda I$ is the adjacency matrix of $\mathcal{R}_{S_m}(G)$ then lemma \ref{lemma1.2} implies that $\mathcal{R}_{S_m}(G)=\mathcal{R}(G;S_1,\dots,S_m)$.
\end{proof}

We now give a proof of lemma \ref{lemma1}.

\begin{proof}
Let $w_1=p_1/q_1$ and $w_2=p_2/q_2$ be in $\mathbb{W}_\pi[\lambda]$ for $p_1,p_2,q_1,q_2\in\mathbb{C}[\lambda]$. That is, $\pi(w_1),\pi(w_2)\leq 0$. Then
\begin{align*}
\pi(w_1+w_2)&=\pi\Big(\frac{p_1q_2+p_2q_1}{q_1q_2}\Big)\leq\max\{\pi(w_1),\pi(w_2)\}\leq 0,\\
\pi(w_1w_2)&=\pi\Big(\frac{p_1p_2}{q_1q_2}\Big)=\pi(w_1)+\pi(w_2)\leq 0, \ \text{and}\\
\pi\Big(\frac{1}{\lambda-w_1}\Big)&=\pi\Big(\frac{q_1}{(q_1\lambda-p_1)}\Big)=-1\leq 0.
\end{align*}
It then follows from equation (\ref{eq0.9}) and (\ref{eq1.0}) that $\mathcal{R}_S(G)\in\mathbb{G}_\pi$ if $G\in\mathbb{G}_\pi$.
\end{proof}

In order to prove theorem \ref{theorem-3} we require the following lemma that essentially implies that the order in which any two vertices are removed from a graph $G\in\mathbb{G}_\pi$ does not effect the final reduced graph.

\begin{lemma}\label{lemma-2}
Suppose $G\in\mathbb{G}_\pi$ with vertex set $V=\{v_1,\dots,v_n\}$ for $n>2$. If $T=\{v_1,v_2\}$ then $\mathcal{R}(G;V-\{v_1\},T)=\mathcal{R}(G;V-\{v_2\},T)$.
\end{lemma}

\begin{proof}
Let $G\in\mathbb{G}_\pi$ and $M(G)_{ij}=\omega_{ij}$ for $1\leq i,j\leq n$. Then $M(G)-\lambda I$ can be written in the block form
$$M(G)-\lambda I=\left( \begin{array}{ccc}
\omega_{11}-\lambda & \omega_{12} & B_1\\
\omega_{21} & \omega_{22}-\lambda & B_2\\
C_1    & C_2    & D
\end{array}
\right) \ \ \text{where} \ \
A=\left[ \begin{array}{cc}
\omega_{11}-\lambda & \omega_{12}\\
\omega_{21} & \omega_{22}-\lambda
\end{array}
\right]$$
is the principle submatrix of $M(G)-\lambda I$ indexed by $\bar{\mathcal{I}}=\{1,2\}$.

Note the assumption $G\in\mathbb{G}_\pi$ implies that both $\omega_{11}-\lambda$, $\omega_{22}-\lambda\neq 0$. Moreover, if $\det(A)=0$ then $\omega_{22}=\omega_{12}\omega_{21}(\omega_{11}-\lambda)^{-1}+\lambda.$ However, equation (\ref{eq2.5}) implies in this case that
$$M\big(\mathcal{R}_{V-v_1}(G)\big)_{22}-\lambda=\omega_{22}+\frac{\omega_{21}\omega_{12}}{\lambda-\omega_{11}}-\lambda=0.$$
Hence, $M\big(\mathcal{R}_{V-v_1}(G)\big)_{22}=\lambda$, which is not possible as it contradicts the conclusion of lemma \ref{lemma1}. Therefore, $\det(A)\neq 0$ implying
$$R_{ij}= \left(D-\frac{1}{\omega_{jj}-\lambda}C_jB_j-\frac{\omega_{jj}-\lambda}{\det(A)}(C_i-\frac{\omega_{ji}}{\omega_{jj}-\lambda}C_j)(B_i-\frac{\omega_{ij}}{\omega_{jj}-\lambda}B_j)\right)$$
is well defined.

Letting $\mathcal{I}_i=\mathcal{I}\cup\{i\}$ for $i\neq j$ and $i,j\in\bar{\mathcal{I}}$ then repeated use of (\ref{eq1.2}) implies $R_{ij}=r(M(G)-\lambda I;\mathcal{I}_i,\mathcal{I})$. Moreover, as $R_{12}=R_{21}$ then it follows that the graphs $\mathcal{R}(G;V-\{v_1\},T)=\mathcal{R}(G;V-\{v_2\},T)$ since these graphs have the same adjacency matrices.
\end{proof}

To simplify the proof of theorem \ref{theorem-3} we note the following. If $S_1,\dots,S_m$ induces a sequence of reductions on $G=(V,E,\omega)$ then $\mathcal{R}(G;S_1,\dots,S_m)$ can alternately be written as $\mathcal{R}em(G;S_0-S_1,\dots,S_{m-1}-S_{m})$ for $V=S_0$. This notation is meant to indicate that at the $i$th reduction we remove the vertices $S_{i-1}-S_{i}$ from the graph $\mathcal{R}em(G;S_0-S_1,\dots,S_{i-2}-S_{i-1})$ for each $1 \leq i\leq m$ where $S_{-1}-S_{0}=\emptyset$.

With this notation in place we give a proof of theorem \ref{theorem-3}.

\begin{proof}
If $G=(V,E,\omega)\in\mathbb{G}_\pi$ and $\{v_1,\dots,v_m\}\subset V$ then, by the discussion preceding theorem \ref{theorem-1}, the graph $\mathcal{R}em(G;\{v_1\},\dots,\{v_m\})$ is well defined. Moreover, for any $1\leq i<m$, let $G_i=\mathcal{R}em(G;\{v_1\},\dots,\{v_{i-1}\})$ where $G_1=G$. Then by lemma \ref{lemma-2} it follows that
$\mathcal{R}em\big(G_i;\{v_{i}\},\{v_{i+1}\}\big)=\mathcal{R}em\big(G_i;\{v_{i+1}\},\{v_{i}\}\big)$ which in turn implies that
\begin{align*}
\mathcal{R}em\big(G;\{v_1\},\dots,\{v_{i}\},\{v_{i+1}\},\dots\{v_m\}\big)&=\\ \mathcal{R}em\big(&G;\{v_1\},\dots,\{v_{i+1}\},\{v_{i}\},\dots\{v_m\}\big).
\end{align*}
By repeatedly switching the order of any two vertices as above it follows that for any bijection $b:\{v_1,\dots,v_m\}\rightarrow\{v_1,\dots,v_m\}$
\begin{equation}\label{eq2.8}
\mathcal{R}em\big(G;\{v_1\},\dots\{v_m\}\big)=\mathcal{R}em\big(G;\{b(v_{1})\},\dots\{b(v_{m})\}\big).
\end{equation}

Suppose $S_1,\dots,S_{m}$ induces a sequence of reductions on $G$. If $S_0=V$ and each $S_{i-1}-S_{i}=\{v^i_1,\dots,v^i_{i_n}\}$ then theorem \ref{theorem3} implies
\begin{align*}
\mathcal{R}em(G;S_{0}-S_{1},\dots,S_{m-1}-S_{m})&=\\ \mathcal{R}em(&G;\{v^1_1\},\dots,\{v^1_{n_1}\},\dots,\{v^{m}_1\},\dots,\{v^m_{n_{m}}\}).
\end{align*}
If $T_1,\dots,T_{p}$ also induces sequence of reductions on $G$ where $T_0=V$ and each $T_{i-1}-T_{i}=\{\tilde{v}^i_1,\dots,\tilde{v}^i_{i_q}\}$ then similarly,
\begin{align*}
\mathcal{R}em(G;T_{0}-T_{1},\dots,T_{p-1}-T_{p})&=\\ \mathcal{R}em(&G;\{\tilde{v}^1_1\},\dots,\{\tilde{v}^1_{q_1}\},\dots,\{\tilde{v}^{p}_1\},\dots,\{\tilde{v}^p_{q_{p}}\}).
\end{align*}
If $S_m=T_p$ then $\bar{S}_m=\bigcup_{i=1}^m\big(S_{i-1}-S_{i}\big)=\bigcup_{i=1}^p\big(T_{i-1}-T_{i}\big)$. There is then a bijection $\tilde{b}:\bar{S}_m\rightarrow\bar{S}_m$ such that
$\mathcal{R}em(G;\{v^1_1\},\dots,\{v^m_{n_{m}}\})=\mathcal{R}em(G;\{\tilde{b}(\tilde{v}^1_1)\},\dots,\{\tilde{b}(\tilde{v}^p_{q_{p}})\})$ implying
$\mathcal{R}(G;S_1,\dots,S_m)=\mathcal{R}(G;T_1,\dots,T_{p-1},S_m)$
by use of equation (\ref{eq2.8}).

This completes the proof.
\end{proof}

The notation $\mathcal{R}_{S_m}[G]=\mathcal{R}(G;S_1,\dots,S_m)$ is then well defined for any graph $G=(V,E,\omega)$ in $\mathbb{G}_\pi$ and nonempty $S_m\subseteq V$. Moreover, if for any $G\in\mathbb{G}_\pi-\emptyset$, $\tau(G)\subseteq V$ is nonempty then the relation $G\sim H$ if $\mathcal{R}_{\tau(G)}[G]\simeq\mathcal{R}_{\tau(H)}[H]$ is an equivalence relation on $\mathbb{G}_\pi$. To see this, note that as the rule $\tau$ specifies a unique nonempty graph $\mathcal{R}_{\tau(G)}[G]$ for any $G\in\mathbb{G}_\pi-\emptyset$, then $\sim$ is an equivalence relation given that $\simeq$ is clearly reflexive, symmetric, and transitive. Hence, theorem \ref{theorem0.1} holds.

We now give a proof of theorem \ref{theorem-2}.

\begin{proof}
Suppose $\mathcal{X}_S(G)=(\mathcal{V},\mathcal{E},\mu)$ is a branch expansion of the graph $G=(V,E,\omega)$. By assumption then each vertex of $G$ and $\mathcal{X}_S(G)$ belongs to a branch of $\mathcal{B}_S(G)$ and $\mathcal{B}_S(\mathcal{X}_S(G))$ respectively.

Suppose that $e_{ij}\in E$. If $i=j$ then, as the vertex $v_i$ belongs to some branch $\beta\in\mathcal{B}_S(G)$, then $\omega(e_{ij})$ belongs to the weight sequence $\Omega_G(\beta)$. If $i\neq j$ then there are branches $v_1,\dots,v_i,\dots,v_s$ and $u_1,\dots,u_j,\dots,u_t$ in $\mathcal{B}_S(G)$ containing $v_i$ and $v_j=u_j$ respectively. Hence, $\beta=v_1,\dots,v_i,v_j,\dots,u_t\in\mathcal{B}_S(G)$. Therefore, $\omega(e_{ij})$ belongs to the weight sequence $\Omega_G(\beta)$.

Hence, each edge weight of $G$ belongs to a weight sequence of some $\beta\in\mathcal{B}_S(G)$. By similar reasoning each edge weight of $\mathcal{X}_S(G)$ belongs to a weight sequence of some $\beta\in\mathcal{B}_S(\mathcal{X}_S(G))$. Since $\mathcal{B}_S(G)\simeq\mathcal{B}_S(\mathcal{X}_S(G))$ then $\omega(e_{ij})$ is an edge weight of $G$ if and only if $\omega(e_{ij})$ is an edge weight of $\mathcal{X}_S(G)$. Therefore, $G$ and $\mathcal{X}_S(G)$ have the same set of edge weights.

Suppose $\beta\in\mathcal{B}_S(\mathcal{X}_S(G))$. As $\mathcal{B}_S(G)\simeq\mathcal{B}_S(\mathcal{X}_S(G))$ there is a bijection
$$b:\mathcal{B}_S(G)\rightarrow\mathcal{B}_S(\mathcal{X}_S(G))$$
such that $\Omega_{\mathcal{X}_S(G)}(\beta)=\Omega_{G}(b^{-1}(\beta))$. If $\beta(i)=v\notin S$ define
$$\tilde{b}(v)=\big(b^{-1}(\beta)\big)(i).$$
To see that $\tilde{b}$ is well defined suppose $v\in\mathcal{V}-S$. As $v$ is assumed to belong to some branch $\beta\in\mathcal{B}_S(\mathcal{X}_S(G))$ and $v\notin S$ then $v$ is an interior vertex of $\beta$. Given that the branches of $\mathcal{B}_S(\mathcal{X}_S(G))$ are pairwise independent then the branch $\beta$ containing $v$ is unique. Hence, $\tilde{b}$ is well-defined.

Moreover, as $|\beta|=|b^{-1}(\beta)|$ then $\tilde{b}$ maps interior vertices of $\beta$ to interior vertices of $b^{-1}(\beta)$. Since each vertex of $\mathcal{V}-S$ and $V-S$ belong to some branch of $\mathcal{B}_S(\mathcal{X}_S(G))$ and $\mathcal{B}_S(G)$ respectively then $\tilde{b}:\mathcal{V}-S\rightarrow V-S$ is onto. Let $$\mathcal{V}_\ell=\{v\in\mathcal{V}-S:\tilde{b}(v)=v_\ell\in V\}.$$ Note that $|\mathcal{V}_\ell|=n_\ell$ is then the number of branches in $\mathcal{B}_S(G)$ containing $v_\ell$. Moreover, $\mathcal{V}-S$ is the disjoint union
$$\mathcal{V}-S=\bigcup_{v_\ell\in V-S}\mathcal{V}_\ell.$$

As $b:\mathcal{B}_S(G)\rightarrow\mathcal{B}_S(\mathcal{X}_S(G))$ preserves weight sequences then $\mu(e_{jj})=\omega(e_{ii})$ for each $v_j\in\mathcal{V}_i$. Hence,
\begin{equation}\label{newer}
\prod_{v_j\in \mathcal{V}-S}\big(\mu(e_{jj})-\lambda\big)=\prod_{v_i\in V-S}\Big(\prod_{v_j\in \mathcal{V}_i}\big(\mu(e_{jj})-\lambda\big)\Big)=\prod_{v_i\in V-S}\big(\omega(e_{ii})-\lambda\big)^{n_i}.
\end{equation}

If $S=\{v_1,\dots,v_m\}$ then by assumption $$\mathcal{B}_{ij}(G;S)\simeq\mathcal{B}_{ij}(\mathcal{X}_S(G);S) \ \ \text{for} \ \ 1\leq i,j\leq m.$$
Equation (\ref{eq1.0}) then implies that the edge $e_{ij}$ in both $\mathcal{R}_S(G)$ and $\mathcal{R}_S(\mathcal{X}_S(G))$ have the same weight. Therefore, $\mathcal{R}_S(G)=\mathcal{R}_S(\mathcal{X}_S(G))$. As
$$\det\big(M(\mathcal{R}_S(G))-\lambda I\big)=\det\big(M\big(\mathcal{R}_S(\mathcal{X}_S(G))\big)-\lambda I\big)$$
then equation (\ref{eq3.3}) implies
$$\frac{\det\big(M(G)-\lambda I\big)}{\det\big(M(G|_{\bar{S}})-\lambda I\big)}=\frac{\det\big(M(\mathcal{X}_S(G))-\lambda I\big)}{\det\big(M(\mathcal{X}_S(G)|_{\bar{S}})-\lambda I\big)}.$$
Via equation (\ref{eq0.2}) it then follows that
$$\det\big(M(\mathcal{X}_S(G))-\lambda I\big)=\det\big(M(G)-\lambda I\big)\frac{\prod_{v_j\in \mathcal{V}-S}\big(\mu(e_{jj})-\lambda\big)}{\prod_{v_i\in V-S}\big(\omega(e_{ii})-\lambda\big)}.$$
From equation (\ref{newer}) we then have that
$$\det\big(M(\mathcal{X}_S(G))-\lambda I \big)=\det\big(M(G)-\lambda I\big)\prod_{v_i\in V-S}\big(\omega(e_{ii})-\lambda\big)^{n_i-1}.$$
\end{proof}

For theorem \ref{theoremunital} we give the following proof.

\begin{proof}
Let $G=(V,E,\omega)$ and $S\in st_0(G)$. Suppose the edge weights of $G$ are in the semiring $\mathbb{U}\subset\mathbb{W}[\lambda]$. By theorem \ref{theorem-2} both $G$ and $\mathcal{X}_S(G)=(\mathcal{V},\mathcal{E},\mu)$ have the same set of edge weights. Hence, $\mathcal{X}_S(G)$ has edge weights in $\mathbb{U}$.

Let $\mathcal{Y}_S(G)=(\mathcal{V},\mathcal{E},\nu)$ and suppose $\beta=v_1,\dots,v_\ell\in\mathcal{B}_S(\mathcal{X}_S(G))$. Then equation (\ref{eq.new}) gives the weights of the edges belonging to $\beta$ in $\mathcal{Y}_S(G)$ as
\begin{equation}\label{redy}
\nu(e_{i,i+1})=
\begin{cases}
\prod_{k=1}^{\ell-1}\mu(e_{k,k+1}) \ \ &\text{for} \ \ i=1\\
1 & 2\leq i\leq \ell-1
\end{cases}
\end{equation}
for $\ell\neq 1$. If $\ell=1$ then $\nu(e_{11})=\mu(e_{11})$. Since $1\in\mathbb{U}$ and $\mathbb{U}$ is closed under multiplication then the edges of $\beta$ in $\mathcal{Y}_S(G)$ have weights in $\mathbb{U}$. Given that each edge of $\mathcal{Y}_S(G)$ belongs to a branch $\beta\in\mathcal{B}_S(\mathcal{Y}_S(G))$ then this implies that $\mathcal{Y}_S(G)$ has edge weights in $\mathbb{U}$. As $\mathcal{Z}^\prime_S(G)$ is a restriction of the graph $\mathcal{Y}_S(G)$ it follows that $\mathcal{Z}^\prime_S(G)$ also has edge weights in $\mathbb{U}$.

To construct $\mathcal{Z}_S^{\prime\prime}(G)$ from the graph $\mathcal{Z}^\prime_S(G)$ we add edges with weights from the weight set of $\mathcal{Y}_S(G)$. Hence, $\mathcal{Z}_S^{\prime\prime}(G)$ has weights in $\mathbb{U}$. Suppose $\mathcal{Z}_S^{\prime\prime}(G)$ has parallel edges $e_1,\dots,e_N$ of the form $(e_i,e_j)$ with weights $w_1,\dots,w_N\in\mathbb{U}$ respectively. As $\mathbb{U}$ is closed under addition then $w_1+\dots+w_N\in\mathbb{U}$. Hence, replacing $e_1,\dots,e_N$ by $e_{ij}$ where $e_{ij}$ is given the weight $w_1+\dots+w_N$ implies that the graph $\mathcal{Z}_S(G)$ has edge weights in $\mathbb{U}$.

Observe that as $\mathcal{B}_S(\mathcal{X}_S(G))=\mathcal{B}_S(\mathcal{Y}_S(G))$ then equation (\ref{redy}) implies $$\mathcal{P}_\mu(\beta)=\prod_{k=1}^{|\beta|-1}\mu(e_{k,k+1})/\lambda^{|\beta|-2}=\mathcal{P}_\nu(\beta)$$
for all $\beta\in\mathcal{B}_S(\mathcal{X}_S(G)),\mathcal{B}_S(\mathcal{Y}_S(G))$. Hence, $\mathcal{R}_S(\mathcal{X}_S(G))=\mathcal{R}_S(\mathcal{Y}_S(G))$. The claim then is that $\mathcal{R}_S(\mathcal{Y}_S(G))=\mathcal{R}_S(\mathcal{Z}_S(G))$.

To verify this let
$$B_{ij}^\ell=\big\{\beta\in\mathcal{B}_{ij}(\mathcal{Y}_S(G);S):|\beta|=\ell\big\}.$$
Moreover, for $\beta\in\mathcal{B}_S(\mathcal{Y}_S(G))$ let $\mathcal{P}^1_\nu(\beta)$ be the first edge weight of $\beta$. Hence, if $\mathcal{L}=\max\{|\beta|:\beta\in\mathcal{B}_S(\mathcal{Y}_S(G))\}$ then
\begin{equation}\label{nextnew}
\sum_{\beta\in\mathcal{B}_{ij}(\mathcal{Y}_S(G);S)}\mathcal{P}_{\nu}(\beta)=\sum_{\ell=2}^{\mathcal{L}}\Big(\sum_{\beta\in B_{ij}^\ell}\mathcal{P}_{\nu}(\beta)\Big)=\sum_{\ell=2}^{\mathcal{L}}\Big(\big(\sum_{\beta\in B_{ij}^\ell}\mathcal{P}^1_{\nu}(\beta)\big)/\lambda^{\ell-2}\Big).
\end{equation}

Suppose $S=\{v_1,\dots,v_m\}$ and $\beta^j=v_i^j,\dots,v_k^j$ for $1\leq j\leq m$. Following the $\mathcal{Z}$-construction we add an edge of the form $(v_i,v_{k-|\beta|+2})$ to the graph $\mathcal{Z}^\prime_S(G)$ with weight $\mathcal{P}^1_\nu(\beta)$ for each branch in $B_{ij}^\ell$ not equal to $\beta^j$. Hence, the graph $\mathcal{Z}_S(G)|_S$ consists of the interior vertices of $\beta^1,\dots,\beta^m$. As these vertices have no loops and $\beta^1,\dots,\beta^m$ are pairwise independent then $S\in st_0(\mathcal{Z}_S(G)).$

Suppose that $B_{ij}^\ell=\{\beta_1,\dots,\beta_N\}$. We then add an edge of the form $(v_i,v_{k-\ell+2})$ to the graph $\mathcal{Z}^\prime_S(G)$ with weight $\mathcal{P}^1_\nu(\beta_t)$ for each branch $\beta_t\in\mathcal{B}_{ij}(\mathcal{Y}_S(G))$ not equal to $\beta^j$. Hence, in $\mathcal{Z}_S^{\prime\prime}(G)$ there are $N$ edges $e_1,\dots,e_N$ of the form $(v_i,v^j_{k-\ell+2})$ having weights $\mathcal{P}_\nu^1(\beta_1),\dots,\mathcal{P}_\nu^1(\beta_m)$ respectively.

Define the branch set
$$\tilde{B}_{ij}^\ell=\big\{\beta\in\mathcal{B}_{ij}(\mathcal{Z}_S(G);S):|\beta|=\ell\big\}.$$
If $\tilde{\beta}=v_i,v^j_{k-\ell+2},\dots,v_m^j$ then by construction $\tilde{B}_{ij}^\ell=\tilde{\beta}$. Moreover, for $\mathcal{Z}_S(G)=(\mathcal{V}-\mathcal{U},\tilde{\mathcal{E}},\tilde{\nu})$, the branch $\beta$ has weight sequence
$$\Omega_{\mathcal{Z}_S(G)}(\tilde{\beta})=\sum_{\beta\in B_{ij}^\ell}\mathcal{P}^1_\nu(\beta),0,\dots,0,1,0,\dots,1$$
as the weights of the edges $e_1,\dots,e_N$ are summed in the construction of $\mathcal{Z}_S(G)$. Therefore,
\begin{equation}\label{newnext}
\sum_{\beta\in \tilde{B}_{ij}^\ell}\mathcal{P}^1_{\tilde{\nu}}(\beta)=\sum_{\beta\in B_{ij}^\ell}\mathcal{P}^1_\nu(\beta).
\end{equation}

Note that the sum
$$\sum_{\beta\in\mathcal{B}_{ij}(\mathcal{Z}_S(G);S)}\mathcal{P}_{\tilde{\nu}}(\beta)=\sum_{\ell=2}^{\mathcal{L}}\Big(\sum_{\beta\in \tilde{B}_{ij}^\ell}\mathcal{P}_{\nu}(\beta)\Big)=\sum_{\ell=2}^{\mathcal{L}}\Big(\big(\sum_{\beta\in \tilde{B}_{ij}^\ell}\mathcal{P}^1_{\tilde{\nu}}(\beta)\big)/\lambda^{\ell-2}\Big).$$
Combining this with equations (\ref{nextnew}) and (\ref{newnext}) implies
$$\sum_{\beta\in\mathcal{B}_{ij}(\mathcal{Y}_S(G);S)}\mathcal{P}_{\nu}(\beta)=\sum_{\beta\in\mathcal{B}_{ij}(\mathcal{Z}_S(G);S)}\mathcal{P}_{\tilde{\nu}}(\beta).$$
Therefore, $\mathcal{R}_S(\mathcal{Y}_S(G))=\mathcal{R}_S(\mathcal{Z}_S(G))$ verifying the claim.

As $\mathcal{R}_S(G)=\mathcal{R}_S(\mathcal{Y}_S(G))$ then $\mathcal{R}_S(G)=\mathcal{R}_S(\mathcal{Z}_S(G))$. As $S$ is a complete structural set, corollary \ref{cor100} implies that $G$ and $\mathcal{Z}_S(G)$ have the same nonzero spectrum.
\end{proof}

\section{Improving Stability Estimates via Dynamical Network Expansions}
A major obstacle in understanding the dynamics of a high dimensional system is that the information needed to do so is spread throughout the various system components. By extending the notion of an isospectral branch expansion to a dynamical network it is possible to modify the structure of a dynamical network $(\mathcal{F},X)$ in a way that preserves the system's dynamics but concentrates this information. 

In 6.1 we introduce the concept of a \emph{dynamical network expansion} $(\mathcal{X}_S\mathcal{F},X_S)$ of the network $(\mathcal{F},X)$. We then show that if $(\mathcal{X}_S\mathcal{F},X_S)$ is globally stable then so is $(\mathcal{F},X)$ (see theorem \ref{gafp}). Using this connection we show in theorem \ref{last} that the global stability of a system $(\mathcal{F},X)$ is better understood by applying theorem \ref{stability} to $(\mathcal{X}_S\mathcal{F},X_S)$ rather than $(\mathcal{F},X)$.

\subsection{Dynamical Network Expansions}
Suppose $(\mathcal{F},X)$ is a dynamical network. Following the discussion in section 2.3, here we consider $(\mathcal{F},X)$ as a dynamical network with no local dynamics. That is, the component
$$\mathcal{F}_j:\bigoplus_{i\in\mathcal{I}_j}X_i\rightarrow X \ \ \text{for} \ \ j\in\mathcal{I}=\{1,\dots,n\}.$$
Alternatively we write
$$\mathcal{F}_j(\{x_i\})=\mathcal{F}_j(x_{j_1},\dots,x_{j_m}), \ \ \text{where} \ \ \mathcal{I}_j=\{j_1,\dots,j_m\}.$$
Note that if we \textit{replace} the variable $x_{j_i}$ of $\mathcal{F}_j$ by the function $f(y_1,\dots,y_k)$ the result is the function
$$\mathcal{F}_j(x_{j_1},\dots,x_{j_{i-1}},f(y_1,\dots,y_k),x_{j_{i+1}},\dots,x_{j_m})$$
having variables $x_{j_1},\dots,x_{j_{i-1}},y_1,\dots,y_k,x_{j_{i+1}},\dots,x_{j_m}$. Additionally, if the sequence $\underline{\gamma}=\ell_1,\dots,\ell_N$ let
$$\mathcal{F}_{j;\underline{\gamma}}(x_{j_1},\dots,x_{j_m})=\mathcal{F}_j(x_{j_1,\underline{\gamma}},\dots,x_{j_m,\underline{\gamma}}).$$
That is, the variables of the function $\mathcal{F}_{j;\underline{\gamma}}$ are indexed by the sequences $j_i,\ell_1,\dots,\ell_N$ for $1\leq i\leq m$.

As defined in section 2 the graph of interactions $\Gamma_\mathcal{F}=(V,E)$ is an unweighted graph. However, $\Gamma_\mathcal{F}$ can be considered to be the weighted graph $\Gamma_\mathcal{F}=(V,E,\omega)\in\mathbb{G}$ by setting $\omega(e_{ij})=1$ for all $e_{ij}\in E$. This will allow us, in particular, to consider structural sets $S\in st(\Gamma_\mathcal{F})$. Recall that if $S=\{v_1,\dots,v_m\}$ then $\mathcal{I}_S=\{1,\dots,m\}$ is the \textit{index set} of $S$.

\begin{definition}\label{admissable}
For $S\in st_0(\Gamma_{\mathcal{F}})$ the set
\begin{equation}\label{eq.add}
\mathcal{A}_S(\mathcal{F})=\{\ell_1,\dots,\ell_N:v_{\ell_1},\dots,v_{\ell_N}\in\mathcal{B}_S(\Gamma_\mathcal{F}), \ N>2\}
\end{equation}
is called the set of \textit{admissible sequences} of $\mathcal{F}$ with respect to $S$.
\end{definition}

Let $(\mathcal{F},X)$ be a dynamical network and suppose $S\in st_0(\Gamma_\mathcal{F})$. For $j\in\mathcal{I}$ let $\mathcal{F}_{\left<j,1\right>}$ be the function
$$\mathcal{F}_j=\mathcal{F}_j(x_{j_i},\dots,x_{j_m})$$
in which each variable $x_{j_\ell}$ is replaced by $x_{j_\ell,j}$ if $j_\ell\notin\mathcal{I}_S$.

For $i>1$ let $\mathcal{F}_{\left<j,i\right>}$ be the function
$$\mathcal{F}_{\left<j,i-1\right>}=\mathcal{F}_{\left<j,i-1\right>}(x_{\underline{\gamma}_1},\dots,x_{\underline{\gamma}_t})$$
in which each $x_{\underline{\gamma}_\ell}=x_{\ell_1,\dots,\ell_N}$ is replaced by the function $\mathcal{F}_{\ell_1;\underline{\gamma}_\ell}$ if $\ell_1\notin\mathcal{I}_S$. If $\ell_1\in\mathcal{I}_S$ for each $1\leq\ell\leq t$ then define $\mathcal{X}_S\mathcal{F}_j=\mathcal{F}_{\left<j,{i-1}\right>}$.

Let $\underline{\gamma}=\ell_1,\dots,\ell_N\in\mathcal{A}_S(\mathcal{F})$. For $1<i<|\underline{\gamma}|=N$ define the $N-2$ spaces
$$X_{i;\underline{\gamma}}=X_{\ell_1}.$$
Additionally, define the functions
$$\mathcal{X}_S\mathcal{F}_{i;\underline{\gamma}}(x_{i-1,\underline{\gamma}})=x_{i-1,\underline{\gamma}}.$$
By way of notation we let
\begin{equation}\label{eq2.40}
X_{N-1,\underline{\gamma}}=X_{\underline{\gamma}}, \  \mathcal{X}_S\mathcal{F}_{N-1,\underline{\gamma}}=\mathcal{X}_S\mathcal{F}_{\underline{\gamma}}, \ \text{and} \ x_{1,\underline{\gamma}}=x_{\ell_1}.
\end{equation}

\begin{definition}\label{expansion}
Suppose $S\in st_0(\Gamma_\mathcal{F})$ such that each vertex of $\Gamma_\mathcal{F}$ belongs to a branch of $\mathcal{B}_S(\Gamma_\mathcal{F})$. Let
$$
\mathcal{X}_S\mathcal{F}=\Big(\bigoplus_{j\in\mathcal{I}_S}\mathcal{X}_S\mathcal{F}_j\Big) \oplus\Big(\bigoplus_{
\begin{smallmatrix}
\underline{\gamma}\in\mathcal{A}_S(\mathcal{F})\\
1<i<|\underline{\gamma}|
\end{smallmatrix}} \mathcal{X}_S\mathcal{F}_{i;\underline{\gamma}}\Big).$$
and
$$X_S=\Big(\bigoplus_{j\in\mathcal{I}_S}X_j\Big)\oplus\Big(\bigoplus_{
\begin{smallmatrix}
\underline{\gamma}\in\mathcal{A}_S(\mathcal{F})\\
1<i<|\underline{\gamma}|
\end{smallmatrix}}X_{i;\underline{\gamma}}\Big).$$
The dynamical network $(\mathcal{X}_S\mathcal{F},X_S)$ is called the \emph{dynamical network expansion} of $(\mathcal{F},X)$ with respect to $S$.
\end{definition}

To see that $(\mathcal{X}_S\mathcal{F})_j$ is well defined for each $j\in\mathcal{I}_S$ suppose
$$\mathcal{F}_{\left<j,{i-1}\right>}=\mathcal{F}_{\left<j,{i-1}\right>}(x_{\underline{\gamma}_1},\dots,x_{\underline{\gamma}_t}).$$ If $\underline{\gamma}_\ell=\ell_1,\dots,\ell_N$ and $N>1$ then by construction $\ell_N=j\in\mathcal{I}_S$, $\ell_k\notin\mathcal{I}_S$, and each $\mathcal{F}_{\ell_k}$ is a function of $x_{\ell_{k-1}}$ for $1< k\leq N$. Therefore, if $\ell_1\in\mathcal{I}_S$ then the vertex sequence $v_{\ell_1},\dots,v_{\ell_N}\in\mathcal{B}_j(\Gamma_\mathcal{F};S)$.

If $\ell_1\notin \mathcal{I}_S$ then $v_{\ell_1},\dots,v_{\ell_N}$ is part of a branch of $\mathcal{B}_j(\Gamma_\mathcal{F};S)$ in the following sense. There is a branch of the form $v_1,\dots,v_{r-1},v_{\ell_1},v_{r+1},\dots,v_s\in\mathcal{B}_j(\Gamma_\mathcal{F};S)$ since $v_{\ell_1}$ is assumed to belong to some branch of $\mathcal{B}_j(\Gamma_\mathcal{F};S)$. Hence,
$$v_1,\dots,v_{r-1},v_{\ell_1},v_{\ell_2},\dots,v_{\ell_N}\in\mathcal{B}_j(\Gamma_\mathcal{F};S)$$
which contains the vertex sequence $v_{\ell_1},\dots,v_{\ell_N}$.

As each branch of $\mathcal{B}_S(\Gamma_\mathcal{F})$ consists of a finite sequence of vertices then for large enough $i$, $\ell_1\in\mathcal{I}_S$. As this holds for each $1\leq \ell\leq t$ then $\mathcal{X}_S\mathcal{F}_j$ is well defined for each $j\in\mathcal{I}_S$. Moreover, this implies that if $x_{\ell_1,\dots,\ell_N}$ is a variable of $\mathcal{X}_S\mathcal{F}_j$ and $N>2$ then $v_{\ell_1},\dots,v_{\ell_N}\in\mathcal{B}_j(\Gamma_\mathcal{F};S)$. Similarly, one can show that each $\mathcal{X}_S\mathcal{F}_j$ is well defined for $j\notin\mathcal{I}_S$.

\begin{lemma}\label{last1}
Suppose $S\in st_0(\Gamma_\mathcal{F})$ and $j\in\mathcal{I}_S$. Then $x_{\ell_1,\dots,\ell_N}$ is a variable of $\mathcal{X}_S\mathcal{F}_j$ if and only if either $N=1$, $\ell_1\in\mathcal{I}_S$, and $x_{\ell_1}$ is a variable of $\mathcal{F}_j$ or $N>2$ and $v_{\ell_1},\dots,v_{\ell_N}\in\mathcal{B}_j(\Gamma_\mathcal{F};S)$.
\end{lemma}

\begin{proof}
Suppose $j\in\mathcal{I}_S$ and that $x_{\ell_1,\dots,\ell_N}$ is a variable of $\mathcal{X}_S\mathcal{F}_j$. A variable of $\mathcal{X}_S\mathcal{F}_j$ is indexed by a sequence of length $N=1$ if it is not replaced by a function in the first step of the construction of $\mathcal{X}_S\mathcal{F}_j$. Hence, $x_{\ell_1}$ is a variable of $\mathcal{X}_S\mathcal{F}_j$ if $x_{\ell_1}$ is a variable of $\mathcal{F}_j$ and $\ell_1\in\mathcal{I}_S$.

Assuming $N=2$, or $x_{\ell_1,\ell_2}$ is a variable of $(\mathcal{X}_S\mathcal{F})_j$, then $\ell_1\in\mathcal{I}_S$ as $x_{\ell_1,\ell_2}$ is not replaced by the function $\mathcal{F}_{\ell_1;\ell_1,\ell_2}$ in the next step of the construction. However, this implies that $x_{\ell_1}$ is a variable of $\mathcal{F}_{\ell_2}$. Hence, in constructing $(\mathcal{X}_S\mathcal{F})_j$ the variable $x_{\ell_1}$ is not replaced by $x_{\ell_1,\ell_2}$ contradicting the assumption that $N=2$. Therefore, $N\neq 2$. However, if $N>2$ then as observed $v_{\ell_1},\dots,v_{\ell_N}\in\mathcal{B}_j(\Gamma_\mathcal{F};S)$.

Conversely, suppose $\beta=v_{\ell_1},\dots,v_{\ell_N}\in\mathcal{B}_j(\Gamma_\mathcal{F};S)$. Then $x_{\ell_{i-1}}$ is a variable of $\mathcal{F}_{\ell_i}$ for $1<i\leq N$ where  $\ell_1,\ell_N\in\mathcal{I}_S$, and $\ell_2,\dots,\ell_{N-1}\notin\mathcal{I}_S$. Hence, $x_{\ell_1,\dots,\ell_N}$ is a variable of $\mathcal{X}_S\mathcal{F}_{\ell_1}$. If $N=2$ then $v_{\ell_1},v_{\ell_2}\in\mathcal{B}_j(\Gamma_\mathcal{F};S)$. As $\ell_1,\ell_2\in\mathcal{I}_S$ then $x_{\ell_1}$ is a variable of $\mathcal{F}_j$.
\end{proof}

\begin{figure}
  \begin{center}
    \begin{overpic}[scale=.3]{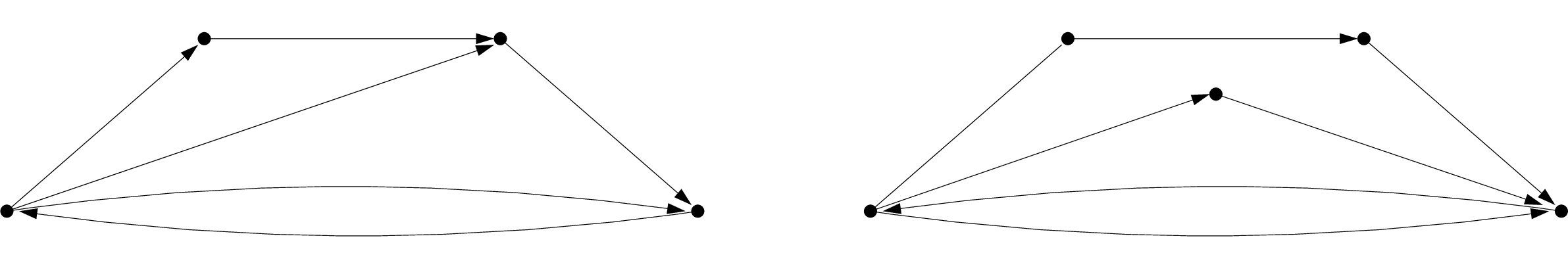}
    \put(0,0){$v_1$}
    \put(44,0){$v_2$}
    \put(12,15){$v_3$}
    \put(30,15){$v_4$}
    \put(21,-2){$\Gamma_{\mathcal{F}}$}

    \put(55,0){$v_1$}
    \put(99,0){$v_2$}
    \put(64,15){$v_{2;1342}$}
    \put(83,15){$v_{1342}$}
    \put(74,11.5){$v_{142}$}
    \put(76,-2){$\Gamma_{\mathcal{X}_S\mathcal{F}}$}
    \end{overpic}
  \end{center}
\caption{The graph of interactions $\Gamma_\mathcal{F}$ and $\Gamma_{\mathcal{X}_S\mathcal{F}}$ of $(\mathcal{F},X)$ and $(\mathcal{X}_S\mathcal{F},X_S)$ in example \ref{ex13}.}\label{fig2001}
\end{figure}
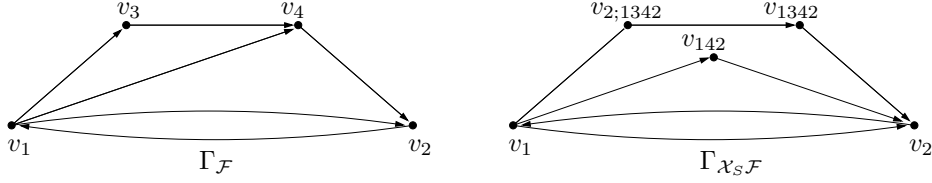

\begin{example}\label{ex13}
Consider the dynamical network $(\mathcal{F},X)$ where $\mathcal{F}$ is given by
$$\mathcal{F}(\mathbf{x})=\left[\begin{array}{l}
\mathcal{F}_1(x_2)\\
\mathcal{F}_2(x_1,x_4)\\
\mathcal{F}_3(x_1)\\
\mathcal{F}_4(x_1,x_3)\\
\end{array}
\right].$$
The graph of interactions $\Gamma_\mathcal{F}=(V,E,\omega)$ is shown in figure \ref{fig2001} (left) where each edge is given unit weight.

Note that $S=\{v_1,v_2\}$ is a complete structural set of $\Gamma_\mathcal{F}$. Moreover, as
$$\mathcal{B}_S(\Gamma_\mathcal{F})=\{\beta_1,\beta_2,\beta_3,\beta_4\}$$
where $\beta_1=v_1,v_4,v_2;$ $\beta_2=v_1,v_3,v_4,v_2;$ $\beta_3=v_1,v_2;$ and $\beta_4=v_2,v_1$ then each vertex of $\Gamma_\mathcal{F}$ belongs to a branch of $\mathcal{B}_S(\Gamma_\mathcal{F})$. Hence, $\mathcal{X}_S\mathcal{F}$ is defined.

To construct this expansion we first consider the components of $\mathcal{F}$ indexed by $\mathcal{I}_S=\{1,2\}$. For $\mathcal{F}_1=\mathcal{F}_1(x_2)$ note that $x_2$ is indexed by an element of $\mathcal{I}_S$. Hence, $\mathcal{F}_{\left<1,1\right>}=\mathcal{F}_1(x_2)$. As each variable of $\mathcal{F}_{\left<1,1\right>}$ is indexed by a sequence beginning with an element of $\mathcal{I}_S$ then
$$\mathcal{X}_S\mathcal{F}_1=\mathcal{F}_1(x_2).$$

For $\mathcal{F}_2=\mathcal{F}_2(x_1,x_4)$ the variable $x_4$ is not indexed by an element of $\mathcal{I}_S$. Hence, $\mathcal{F}_{\left<2,1\right>}=\mathcal{F}_2(x_1,x_{42})$. Replacing $x_{42}$ by the function $\mathcal{F}_{4;42}=\mathcal{F}_4(x_{142},x_{342})$ yields $\mathcal{F}_{\left<2,2\right>}=\mathcal{F}_4(x_1,\mathcal{F}_4(x_{142},x_{342}))$. Continuing we replace $x_{342}$ by $\mathcal{F}_{3;342}$ which gives us $\mathcal{F}_{\left<2,3\right>}$. Since each variable of $\mathcal{F}_{\left<2,3\right>}$ is indexed by a sequence beginning with an element of $\mathcal{I}_S$ then $\mathcal{X}_S\mathcal{F}_2=\mathcal{F}_{\left<2,3\right>}$ which is given by
$$\mathcal{X}_S\mathcal{F}_2=\mathcal{F}_2\big(x_1,\mathcal{F}_4(x_{142},\mathcal{F}_3(x_{1342}))\big).$$

Note that $\mathcal{A}_S(\mathcal{F})=\{142,1342\}$. For $142\in\mathcal{A}_S(\mathcal{F})$ there is a single corresponding function $\mathcal{X}_S\mathcal{F}_{2;142}$ given by
$\mathcal{X}_S\mathcal{F}_{2;142}(x_{1,142})=x_{1;142}$. Similarly, the functions $$\mathcal{X}_S\mathcal{F}_{2;1342}(x_{1;1342})=x_{1;1342} \ \ \text{and} \ \ \mathcal{X}_S\mathcal{F}_{3;1342}(x_{2;1342})=x_{2;1342}$$
correspond to $1342\in\mathcal{A}_S(\mathcal{F})$

By the identification in (\ref{eq2.40}) the expansion $(\mathcal{X}_S\mathcal{F}, X_S)$ is given by
\begin{equation}\label{eqnext}
\mathcal{X}_S\mathcal{F}(\textbf{x})=
\left[
\begin{array}{l}
\mathcal{X}_{S}\mathcal{F}_{1}(x_2)\\
\mathcal{X}_{S}\mathcal{F}_{2}(x_1,x_{142},x_{1342})\\
\mathcal{X}_{S}\mathcal{F}_{142}(x_1)\\
\mathcal{X}_{S}\mathcal{F}_{1342}(x_{2;1342})\\
\mathcal{X}_{S}\mathcal{F}_{2;1342}(x_1)\\
\end{array}
\right]=
\left[
\begin{array}{l}
\mathcal{F}_1(x_2)\\
\mathcal{F}_2\big(x_1,\mathcal{F}_4(x_{142},\mathcal{F}_3(x_{1342}))\big)\\
x_1\\
x_{2;1342}\\
x_1\\
\end{array}
\right]
\end{equation}
where $X_S=\big(X_1\oplus X_2\big)\oplus\big(X_{142}\oplus X_{1342}\oplus X_{2;1342}\big)$. The graph of interactions $\Gamma_{\mathcal{X}_S\mathcal{F}}$ is shown in figure \ref{fig2001} (right). Note that the branch expansion $\mathcal{X}_S(\Gamma_\mathcal{F})=\Gamma_{\mathcal{X}_S\mathcal{F}}$ (see lemma \ref{expand}).
\end{example}

A natural question to ask is whether the expansion $(\mathcal{X}_S\mathcal{F},X_S)$ and $(\mathcal{F},X)$ have similar dynamics. To address this let $\mathcal{L}=\max\{|\beta|:\beta\in\mathcal{B}_S(\Gamma_\mathcal{F})\}$. For $\tilde{\mathbf{x}}\in X$ let $\tilde{\mathbf{x}}_S\in X_S$ be given by
$$
(\tilde{\mathbf{x}}_S)_j=
\begin{cases}
\mathcal{F}^\mathcal{L}_j(\tilde{\mathbf{x}}) & \text{for} \ \ j\in\mathcal{I}_S\\
\mathcal{F}^{\mathcal{L}-i+1}_{\ell_1}(\tilde{\mathbf{x}}) & \text{for} \ \ j=i;\ell_1,\dots,\ell_N
\end{cases}.
$$

\begin{lemma}\label{identical}
Let $(\mathcal{X}_S\mathcal{F},X_S)$ be a dynamical network expansion of $(\mathcal{F},X)$. Then
$$\mathcal{X}_S\mathcal{F}_j^{k}(\tilde{\mathbf{x}}_S)=\mathcal{F}^{\mathcal{L}+k}_j(\tilde{\mathbf{x}})$$
for any $\tilde{\mathbf{x}}\in X$, $j\in\mathcal{I}$, and $k>0$.
\end{lemma}

\begin{proof}
Let $j\in\mathcal{I}_S$ and suppose $\mathcal{F}_j=\mathcal{F}_j(x_{j_1},\dots,x_{j_m})$ is a type-1 function. Then, $\mathcal{X}_S\mathcal{F}_j=\mathcal{F}_j$ and $j_\ell\in\mathcal{I}_S$ for $1\leq\ell\leq m$. As $(\tilde{x}_S)_{j_\ell}=\mathcal{F}^\mathcal{L}_{j_\ell}(\tilde{x})$ for $j_\ell\in\mathcal{I}_S$ then
\begin{equation}\label{eq123}
\mathcal{X}_S\mathcal{F}_j(\tilde{\mathbf{x}}_S)=\mathcal{F}_j\big(\mathcal{F}^\mathcal{L}_{j_1}(\tilde{\mathbf{x}}),\dots,\mathcal{F}^\mathcal{L}_{j_m}(\tilde{\mathbf{x}}))\big)=\mathcal{F}_j^{\mathcal{L}+1}(\tilde{\mathbf{x}})
\end{equation}
Suppose $\mathcal{F}_j=\mathcal{F}_j(x_{j_1},\dots,x_{j_m})$ is a type-2 function. Then either $j_\ell\in\mathcal{I}_S$ or $\mathcal{F}_{\left<j,2\right>}$ is the function $\mathcal{F}_j$ in which
the variable $x_{j_\ell}$ is replaced by the function $\mathcal{F}_{j_\ell;j_\ell,j}$ for all ${j_\ell}\notin\mathcal{I}_S$. This implies that
\begin{equation}\label{1}
(\tilde{\mathbf{x}}_S)_{j_\ell}=\mathcal{F}^\mathcal{L}_{j_\ell}(\tilde{\mathbf{x}}) \ \ \text{if} \ \ j_\ell\in\mathcal{I}_S.
\end{equation}

If on the other hand $j_\ell\notin\mathcal{I}_S$ then $\mathcal{F}_{j_\ell}$ is a type-1 function. By lemma \ref{last1}, if $x_{i,j_\ell,j}$ is a variable of $\mathcal{F}_{j_\ell;j_\ell,j}$ then $v_i,v_{j_\ell},v_{j}\in\mathcal{B}_j(\Gamma_\mathcal{F})$ implying $i,j_\ell,j\in\mathcal{A}_S(\mathcal{F})$. As per our notation, if $\ell_1,\dots,\ell_N\in\mathcal{A}_S(\mathcal{F})$
then the variable $x_{\ell_1,\dots,\ell_N}=x_{N-1;\ell_1,\dots,\ell_N}$. Therefore, each variable of  $\mathcal{F}_{j_\ell;j_\ell,j}$ has the form $x_{i,j_\ell,j}=x_{(2;i,j_\ell,j)}$. Hence,
\begin{equation}\label{2}
\mathcal{F}_{j_\ell;j_\ell,j}(\tilde{\mathbf{x}}_S)=\mathcal{F}_{j_\ell}(\mathcal{F}^{\mathcal{L}-2+1}\big(\tilde{\mathbf{x}})\big)=\mathcal{F}_{j_\ell}^{\mathcal{L}}(\tilde{\mathbf{x}}) \ \ \text{for} \ \ j_\ell\notin\mathcal{I}_S.
\end{equation}

As the variable $x_{\ell_j}$ is replaced by the function $\mathcal{F}_{j_\ell;j_\ell,j}$ for all $j_\ell\notin\mathcal{I}_S$ then (\ref{1}) and (\ref{2}) together imply that
$$\mathcal{X}_S\mathcal{F}_j(\tilde{\mathbf{x}}_S)=\mathcal{F}_j \big(\mathcal{F}^\mathcal{L}_{j_1}(\tilde{\mathbf{x}}),\dots,\mathcal{F}^\mathcal{L}_{j_m} (\tilde{\mathbf{x}})\big)=\mathcal{F}_j^{\mathcal{L}+1}(\tilde{\mathbf{x}})$$
as in (\ref{eq123}). Continuing in this manner it follows that
\begin{equation}\label{3}
\mathcal{X}_S\mathcal{F}_j(\tilde{\mathbf{x}}_S)=\mathcal{X}_S\mathcal{F}_j(\tilde{\mathbf{x}})=\mathcal{F}_j^{\mathcal{L}+1}(\tilde{\mathbf{x}}) \ \ \text{for} \ \ j\in\mathcal{I}_S.
\end{equation}
Hence, lemma \ref{identical} holds for $k=1$.

Note that each
$$\mathcal{X}_S\mathcal{F}_{i,\underline{\gamma}}(\tilde{\mathbf{x}}_S)=\mathcal{X}_S \mathcal{F}_{i;\ell_1,\dots,\ell_N}(\mathcal{F}_{\ell_1}^{\mathcal{L}-(i-1)+1}(\tilde{\mathbf{x}}))=\mathcal{F}_{\ell_1}^{\mathcal{L}-i+2}(\tilde{\mathbf{x}}).$$ This together with (\ref{3}) implies that
$$
\mathcal{X}_S\mathcal{F}_j(\tilde{\mathbf{x}}_S)=
\begin{cases}
\mathcal{F}^{\mathcal{L}+1}_j(\tilde{\mathbf{x}}) & \text{for} \ \ j\in\mathcal{I}_S\\
\mathcal{F}^{\mathcal{L}-i+2}_{\ell_1}(\tilde{\mathbf{x}}) & \text{for} \ \ j=i;\ell_1,\dots,\ell_N
\end{cases}.
$$

As $\mathcal{X}_S\mathcal{F}(\tilde{\mathbf{x}}_S)$ has the same form as $\tilde{\mathbf{x}}_S$ then replacing $\tilde{\mathbf{x}}_S$ by $\mathcal{X}_S\mathcal{F}(\tilde{\mathbf{x}}_S)$ and applying the same argument yields
$$
\mathcal{X}_S\mathcal{F}^2_j(\tilde{\mathbf{x}}_S)_j=
\begin{cases}
\mathcal{F}^{\mathcal{L}+1}_j(\tilde{\mathbf{x}}) & \text{for} \ \ j\in\mathcal{I}_S\\
\mathcal{F}^{\mathcal{L}-i+2}_{\ell_1}(\tilde{\mathbf{x}}) & \text{for} \ \ j=i;\ell_1,\dots,\ell_N
\end{cases}.
$$
Continuing in this manner it follows that  $\mathcal{X}_S\mathcal{F}_j^{k}(\tilde{\mathbf{x}}_S)=\mathcal{F}^{\mathcal{L}+k}_j(\tilde{\mathbf{x}})$
for all $j\in\mathcal{I}_S$, and $k>0$.
\end{proof}

\begin{theorem}\label{gafp}
Suppose $({\mathcal{X}_S\mathcal{F}},X_S)$ is a dynamical network expansion of $(\mathcal{F},X)$. If $({\mathcal{X}_S\mathcal{F}},X_S)$ has a globally attracting fixed point then $(\mathcal{F},X)$ also has a globally attracting fixed point.
\end{theorem}

\begin{proof}
Suppose $\mathbf{y}\in X_S$ is a globally attracting fixed point of the dynamical network $(\mathcal{X}_S\mathcal{F},X_S)$. For $\tilde{\mathbf{x}}\in X$ and $j\in\mathcal{I}_S$ lemma \ref{identical} implies that $\mathcal{X_S\mathcal{F}}_j^{k}(\tilde{\mathbf{x}})=\mathcal{F}^{\mathcal{L}+k}_j(\tilde{\mathbf{x}})$ for all $k>0$. Hence, for $j\in\mathcal{I}_S$
\begin{equation}\label{fin}
d\big(\mathcal{F}^k_j(\tilde{\mathbf{x}}),\mathbf{y}_j\big)\rightarrow 0 \ \ \text{as} \ \ k\rightarrow\infty.
\end{equation}

Proceeding as in the proof of lemma \ref{identical} suppose $j\notin\mathcal{I}_S$ and that the function $\mathcal{F}_j$ is a type-1 function. Hence, $\mathcal{F}_j=\mathcal{F}_j(x_{j_1},\dots,x_{j_m})$ where $j_\ell\in\mathcal{I}_S$ for all $1\leq\ell\leq m$. Let $\epsilon>0$. As the function $\mathcal{F}:X\rightarrow X$ is continuous and $$\mathcal{F}^{k}_j(\tilde{\mathbf{x}})=\mathcal{F}_j(\mathcal{F}^{k-1}_{j_1}(\tilde{\mathbf{x}}),\dots,\mathcal{F}_{j_m}^{k-1}(\tilde{\mathbf{x}}))$$
then (\ref{fin}) implies, for large enough $k$, that $d\big(\mathcal{F}^k_j(\tilde{\mathbf{x}}),\mathcal{F}_j(\mathbf{y})\big)<\epsilon.$ Therefore,
\begin{equation}\label{fin2}
d\big(\mathcal{F}^k_j(\tilde{\mathbf{x}}),\mathcal{F}_j(\mathbf{y})\big)\rightarrow 0 \ \ \text{as} \ \ k\rightarrow\infty
\end{equation}
if $\mathcal{F}_j$ is a type-1 function.

Suppose $j\notin\mathcal{I}_S$ and that the function $\mathcal{F}_j$ is a type-2 function. If $x_{j_\ell}$ is a variable of $\mathcal{F}_j$ then either $j_\ell\in\mathcal{I}_S$ or $\mathcal{F}_{j_\ell}$ is a type-1 function. If $j_\ell\in\mathcal{I}_S$ then (\ref{fin}) implies that $\mathcal{F}_{\ell_j}^k(\tilde{\mathbf{x}})\rightarrow\mathbf{y}_{j_\ell}$ as $k\rightarrow\infty$. If $j_\ell\notin\mathcal{I}_S$ then $\mathcal{F}_{j_\ell}$ is a type-1 function and $\mathcal{F}_{j_\ell}^k(\tilde{\mathbf{x}})\rightarrow\mathcal{F}_{j_\ell}(\mathbf{y})$ by (\ref{fin2}). Hence, $\mathcal{F}_j^k(\tilde{\mathbf{x}})\rightarrow\mathcal{F}_j^\infty(\tilde{\mathbf{x}})\in X$ as $k\rightarrow\infty$ where the component $\mathcal{F}_j^\infty(\tilde{\mathbf{x}})$ depends only on $\mathbf{y}$.

As each $\mathcal{F}_j$ for $j\in\mathcal{I}$ is a type-$i$ function for some $i<\infty$ then continuing in this manner implies that $(\mathcal{F},X)$ has a globally attracting fixed point.
\end{proof}

The following is an immediate corollary to theorem \ref{gafp}.

\begin{corollary}\label{cor30}
Let $({\mathcal{X}_S\mathcal{F}},X_S)$ be a dynamical network expansion of $(\mathcal{F},X)$. Suppose the constants $\tilde{\Lambda}_{ij}$ satisfy (\ref{eq2.3}) for $\mathcal{X}_S\mathcal{F}$. If $\rho(\tilde{\Lambda})<1$ then $(\mathcal{F},X)$ has a globally attracting fixed point.
\end{corollary}

The question then is whether there is any advantage in considering a dynamical network expansion over the original (unexpanded) network. As it turns out, dynamical network expansions always allow for better estimates (or at least no worse) of a dynamical network's global stability. To demonstrate this we first require the following lemma.

\begin{lemma}\label{expand}
Suppose that $(\mathcal{X}_S\mathcal{F},X_S)$ is a dynamical network expansion of $(\mathcal{F},X)$. Then $\mathcal{X}_S(\Gamma_\mathcal{F})=\Gamma_{\mathcal{X}_S\mathcal{F}}$.
\end{lemma}

\begin{proof}
Let $\Gamma_{\mathcal{X}_S\mathcal{F}}=(\mathcal{V},\mathcal{E},\mu)$. By construction $$\mathcal{V}=S\cup\{v_{i;\underline{\gamma}}:\underline{\gamma}\in\mathcal{A}_S(\mathcal{F}), \ 1<i<|\underline{\gamma}|\}$$
where the vertex $v_j\in S$ corresponds to the component $\mathcal{X}_S\mathcal{F}_j$ and $v_{i;\underline{\gamma}}\in\mathcal{V}-S$ the component $\mathcal{X}_S\mathcal{F}_{i;\underline{\gamma}}$.

Suppose $\underline{\gamma}=\ell_1,\dots,\ell_N\in\mathcal{A}_S(\mathcal{F})$. Then $v_{\ell_1},\dots,v_{\ell_N}\in\mathcal{B}_S(\Gamma_\mathcal{F})$ where $N>2$. Hence, lemma \ref{last1} implies that $x_{\ell_1,\dots,\ell_N}=x_{N-1,\underline{\gamma}}$ is a variable of the function $\mathcal{X}_S\mathcal{F}_{\ell_N}$. Moreover, as $\mathcal{X}_S\mathcal{F}_{i;\underline{\gamma}}=\mathcal{X}_S\mathcal{F}_{i;\underline{\gamma}} (x_{i-1;\underline{\gamma}})$ for $1<i<\underline{\gamma}$ where $x_{1,\underline{\gamma}}=x_{\ell_1}$ then the vertices $v_{\ell_1},v_{2;\underline\gamma},\dots,v_{N-1,\underline\gamma},v_{\ell_N}$ form a path in $\Gamma_{\mathcal{X}_S\mathcal{F}}$ such that $v_{\ell_1},v_{\ell_N}\in S$.

As any two admissible sequences of $\mathcal{A}_S(\mathcal{F})$ are distinct then $\Gamma_{\mathcal{X}_S\mathcal{F}}|_{\bar{S}}$ consists of $|\mathcal{A}_S(\mathcal{F})|$ paths that share no vertices of the form $v_{2,\underline{\gamma}},\dots,v_{N-1,\underline{\gamma}}$. Since the vertex $v_{i;\underline{\gamma}}\in\mathcal{V}$ has no loop in $\Gamma_{\mathcal{X}_S\mathcal{F}}|_{\bar{S}}$ then $\Gamma_{\mathcal{X}_S\mathcal{F}}|_{\bar{S}}$ contains no cycles. Hence, $S\in st_0(\Gamma_{\mathcal{X}_S\mathcal{F}})$ and for each $\underline{\gamma}=\ell_1,\dots,\ell_N$,
\begin{equation}\label{nexteq}
\mathcal{B}_{\ell_1,\ell_N}(\Gamma_{\mathcal{X}_S\mathcal{F}};S)=\{v_{\ell_1},v_{2,\underline\gamma}, \dots,v_{N-1,\underline\gamma},v_{\ell_N}:\underline{\gamma}=\ell_1,\dots,\ell_N\in\mathcal{A}_S(\mathcal{F})\}
\end{equation}
if $N>2$.

For $\beta=v_{\ell_1},\dots,v_{\ell_N}\in\mathcal{B}_{\ell_1,\ell_N}(\Gamma_\mathcal{F};S)$ and $\underline{\gamma}=\ell_1,\dots,\ell_N$ let
\begin{equation}\label{beta}
b(\beta)=
\begin{cases}
\beta & \text{if} \ \ N=2\\
v_{\ell_1},v_{2,\underline\gamma},\dots,v_{N-1,\underline\gamma},v_{\ell_N} & \text{otherwise}
\end{cases}.
\end{equation}
The claim is that
$$b:\mathcal{B}_{\ell_1,\ell_N}(\Gamma_{\mathcal{F}};S)\rightarrow\mathcal{B}_{\ell_1,\ell_N}(\Gamma_{\mathcal{X}_S\mathcal{F}};S)$$ and that this map is a bijection.

To see this suppose $N=2$. Then, $\beta=v_{\ell_1},v_{\ell_2}\in\mathcal{B}_{\ell_1,\ell_N}(\Gamma_{\mathcal{F}};S)$ implying $x_{\ell_1}$ is a variable of $\mathcal{F}_{\ell_2}$. As $\ell_1\in\mathcal{I}_S$ then by construction $x_{\ell_1}$ is a variable of $\mathcal{X}_S\mathcal{F}_{\ell_2}$. Since $v_{\ell_1},v_{\ell_2}\in S$ then $\beta=v_{\ell_1},v_{\ell_2}\in\mathcal{B}_{\ell_1,\ell_2}(\Gamma_{\mathcal{X}_S\mathcal{F}};S)$. Conversely, suppose $v_{\ell_1},v_{\ell_2}\in\mathcal{B}_{\ell_1,\ell_2}(\Gamma_{\mathcal{X}_S\mathcal{F}};S)$ i.e. $x_{\ell_1}$ is a variable of the function $\mathcal{X}_S\mathcal{F}_{\ell_2}$. Then $x_{\ell_1}$ is a variable of $\mathcal{F}_{\ell_2}$ by lemma \ref{last1} implying $\beta\in\mathcal{B}_{\ell_1,\ell_2}(\Gamma_{\mathcal{F}};S)$. Hence, $b$ restricted to branches of length $2$ is a bijection.

For $N>2$, the sequence $\underline{\gamma}=\ell_1,\dots,\ell_N\in\mathcal{A}_S(\mathcal{F})$. Hence,
$$b(\beta)=v_{\ell_1},v_{2;\underline{\gamma}},\dots,v_{N-1;\underline{\gamma}},v_{\ell_N}$$ which is a branch in $\mathcal{B}_{\ell_1,\ell_N}(\Gamma_{\mathcal{X}_S\mathcal{F}};S)$ by (\ref{nexteq}). Conversely, if
$$v_{\ell_1},v_{2;\underline{\gamma}},\dots,v_{N-1;\underline{\gamma}},v_{\ell_N}\in\mathcal{B}_{\ell_1,\ell_N}(\Gamma_{\mathcal{X}_S\mathcal{F}};S)$$
then $\underline{\gamma}\in\mathcal{A}_S(\mathcal{F})$ implying $v_{\ell_1},\dots,v_{\ell_N}\in\mathcal{B}_{\ell_1,\ell_N}(\Gamma_{\mathcal{F}};S)$. Therefore, the function $b:\mathcal{B}_{\ell_1,\ell_N}(\Gamma_{\mathcal{F}};S)\rightarrow\mathcal{B}_{\ell_1,\ell_N}(\Gamma_{\mathcal{X}_S\mathcal{F}};S)$ is a bijection.

As $\ell_1,\ell_N\in\mathcal{I}_S$ were arbitrary and each edge of $\Gamma_\mathcal{F}$ and $\Gamma_{\mathcal{X}_S\mathcal{F}}$ have unit weight then
$$\mathcal{B}_S(\Gamma_\mathcal{F})\simeq\mathcal{B}_S(\Gamma_{\mathcal{X}_S\mathcal{F}}).$$
Since the branches of $\Gamma_{\mathcal{X}_S\mathcal{F}}$ are independent it follows that $\mathcal{X}_S(\Gamma_\mathcal{F})=\Gamma_{\mathcal{X}_S\mathcal{F}}$.
\end{proof}

\begin{theorem}\label{last} \textbf{(Improved Stability Estimates for Dynamical Networks)}
Let $({\mathcal{X}_S\mathcal{F}},X_S)$ be a dynamical network expansion of $(\mathcal{F},X)$. Suppose there exist constants $\Lambda_{ij}$ satisfying (\ref{eq2.3}) for $\mathcal{F}$. Then there are constants $\tilde{\Lambda}_{ij}$ satisfying (\ref{eq2.3}) for $\mathcal{X}_S\mathcal{F}$ such that $\rho(\tilde{\Lambda})\leq\rho(\Lambda)$.
\end{theorem}

Before proving theorem \ref{last} suppose the constants $\Lambda_{ij}$ satisfy (\ref{eq2.3}) for the map $\mathcal{F}$. Define $\Gamma_\mathcal{F}(\Lambda)=(V,E,\omega)$ to be the graph with adjacency matrix $\Lambda$ where $\Lambda_{ij}$ satisfy (\ref{eq2.3}) for $\mathcal{F}$. Note that as $\Lambda_{ij}\neq 0$ if and only if $\mathcal{F}_j$ is a function of $x_i$ then $\Gamma_\mathcal{F}$ and $\Gamma_\mathcal{F}(\Lambda)$ are identical as unweighted graphs. Therefore, if $S\in st_0(\Gamma_\mathcal{F})$ then $S\in st_0(\Gamma_\mathcal{F}(\Lambda))$ and $\mathcal{B}_S\big(\Gamma_\mathcal{F}(\Lambda)\big)=\mathcal{B}_S(\Gamma_\mathcal{F})$.

Moreover, the branch $\beta=v_{i_1},\dots,v_{i_m}\in\mathcal{B}_S(\Gamma_\mathcal{F}(\Lambda))$ has weight sequence
$$\Omega_{\Gamma_\mathcal{F}(\Lambda)}(\beta)=\Lambda_{i_1i_2},\dots,0,\Lambda_{i_l,i_{\ell+1}},0,\dots,\Lambda_{i_{m-1},i_m}$$
implying that the branch product
\begin{equation}\label{beq}
\mathcal{P}_\omega(\beta)=\frac{\prod_{\ell=1}^{|\beta|-1}\Lambda_{i_\ell,i_{\ell+1}}}{\lambda^{|\beta|-2}}.
\end{equation} We now give a proof of theorem \ref{last}.

\begin{proof}
Let $\mathbf{x},\mathbf{y}\in X_S$. For $\underline{\gamma}\in\mathcal{A}_S(\mathcal{F})$ and $1<i<|\underline{\gamma}|$,
$$d\big(\mathcal{X}_S\mathcal{F}_{i;\underline{\gamma}}(\mathbf{x}), \mathcal{X}_S\mathcal{F}_{i;\underline{\gamma}}(\mathbf{y})\big)=
d(x_{i-1;\underline{\gamma}},y_{i-1;\underline{\gamma}}).
$$
Therefore, for any $\underline{\gamma}\in\mathcal{A}_S(\mathcal{F})$ and $1<i<|\underline{\gamma}|$ the constants
\begin{equation}\label{const1}
\tilde{\Lambda}_{i-1;\underline{\gamma},j}=
\begin{cases}
1 & \ \ \text{if} \ \ j=i;\underline{\gamma}\\
0 & \ \ \text{otherwise}
\end{cases}
\end{equation}
satisfy equation (\ref{eq2.3}) for $\mathcal{X}_S\mathcal{F}_j$.

For $j\in\mathcal{I}$ let $\mathcal{I}^+_j=\mathcal{I}_j\cup\mathcal{I}_S$ and $\mathcal{I}^-_j=\mathcal{I}_j-\mathcal{I}_S$. For $i_1\in\mathcal{I}_S$
\begin{align}\label{align1}
d\big(\mathcal{X}_S\mathcal{F}_{i_1}(\mathbf{x}),\mathcal{X}_S\mathcal{F}_{i_1}(\mathbf{y})\big)&\leq\\
\sum_{i_2\in\mathcal{I}^+_{i_1}}\Lambda_{i_2i_1}d(x_{i_2},y_{i_2})&+\sum_{i_2\in\mathcal{I}^-_{i_1}}\Lambda_{i_2i_1}
d\big(\mathcal{F}_{i_2;i_1}(\mathbf{x}),\mathcal{F}_{i_2;i_1}(\mathbf{y})\big).\notag
\end{align}
Similarly, the sum
\begin{align}
\sum_{i_2\in\mathcal{I}^-_{i_1}}\Lambda_{i_2i_1}
d\big(\mathcal{F}_{i_2;i_1}(\mathbf{x}),\mathcal{F}_{i_2;i_1}(\mathbf{y})\big)&\leq \label{align2}\\
\sum_{i_2\in\mathcal{I}^-_{i_1}}\Lambda_{i_2i_1}\Big(\sum_{i_3\in\mathcal{I}^+_{i_2}}\Lambda_{i_3i_2}d(x_{i_3,i_2,i_1},y_{i_3,i_2,i_1})&+\sum_{i_3\in\mathcal{I}^-_{i_2}}\Lambda_{i_3i_2}
d\big(\mathcal{F}_{i_3;i_2,i_1}(\mathbf{x}),\mathcal{F}_{i_3;i_2,i_1}(\mathbf{y})\big)\Big).\notag
\end{align}

Note that if $i_2\in\mathcal{I}_{i_1}^-$ and $i_3\in\mathcal{I}_{i_2}^+$ then $i_2\notin\mathcal{I}_S$ and $i_3\in\mathcal{I}_S$ implying the sequence $v_{i_3},v_{i_2},v_{i_1}\in\mathcal{B}_S\big(\Gamma_\mathcal{F}(\Lambda)\big)$. In particular, $\mathcal{P}_{\omega}(v_{i_3},v_{i_2},v_{i_1})\lambda=\Lambda_{i_3i_2}\Lambda_{i_2i_1}$ by (\ref{beq}).

Let $B^\ell_{i_1}=\{\beta\in\mathcal{B}_{i_1}(\Gamma_{\mathcal{F}}(\Lambda);S):|\beta|=\ell\}$. Moreover, if the branch
$\beta=v_1\dots,v_m$ let $x_{\beta}=x_{1,\dots,m}$ and $\gamma(\beta)=1,\dots,m$. If $\Gamma_\mathcal{F}(\Lambda)=(V,E,\omega)$ then combining inequalities in (\ref{align1}) and (\ref{align2}) yields
\begin{align*}
d\big(\mathcal{X}_S\mathcal{F}_{i_1}(&\mathbf{x}),\mathcal{X}_S\mathcal{F}_{i_1}(\mathbf{y})\big)\leq\\ \sum_{i_2\in\mathcal{I}^+_{i_1}}\Lambda_{i_2i_1}d(x_{i_2}&,y_{i_2})+
\sum_{i_2\in\mathcal{I}^-_{i_1}}\sum_{i_3\in\mathcal{I}^+_{i_2}}\Lambda_{i_3i_2}\Lambda_{i_2i_1}d(d(x_{i_3,i_2,i_1},y_{i_3,i_2,i_1}))+\\
&\sum_{i_2\in\mathcal{I}^-_{i_1}}\sum_{i_3\in\mathcal{I}^-_{i_2}}\Lambda_{i_3i_2}\Lambda_{i_2i_1}
d\big(\mathcal{F}_{i_3;i_2,i_1}(\mathbf{x}),\mathcal{F}_{i_3;i_2,i_1}(\mathbf{y})\big)=\\
\sum_{i_2\in\mathcal{I}^+_{i_1}}\Lambda_{i_2i_1}d(x_{i_2}&,y_{i_2})+
\sum_{\beta\in B^3_{i_1}}\big(\mathcal{P}_{\omega}(\beta)\lambda\big) d(x_\beta,y_\beta)+\\
&\sum_{i_2\in\mathcal{I}^-_{i_1}}\sum_{i_3\in\mathcal{I}^-_{i_2}}\Lambda_{i_3i_2}\Lambda_{i_2i_1}
d\big(\mathcal{F}_{i_3;i_2,i_1}(\mathbf{x}),\mathcal{F}_{i_3;i_2,i_1}(\mathbf{y})\big).
\end{align*}
Continuing in this manner it follows that
\begin{align*}
d\big(\mathcal{X}_S\mathcal{F}_{i_1}(\mathbf{x}),\mathcal{X}_S&\mathcal{F}_{i_1}(\mathbf{y})\big)\leq\\ \sum_{i_2\in\mathcal{I}^+_{i_1}}\Lambda_{i_2i_1}d(x_{i_2}&,y_{i_2})+
\sum_{\ell=3}^\mathcal{L}\Big(\sum_{\beta\in B^\ell_{i_1}}\big(\mathcal{P}_{\omega}(\beta)\lambda^{|\beta|-2}\big) d(x_\beta,y_\beta)\Big)=\\
\sum_{i_2\in\mathcal{I}^+_{i_1}}\Lambda_{i_2i_1}d(x_{i_2}&,y_{i_2})+
\sum_{\beta\in \mathcal{B}_{i_1}(\Gamma_{\mathcal{F}}(\Lambda);S)}\big(\mathcal{P}_{\omega}(\beta)\lambda^{|\beta|-2}\big) d(x_\beta,y_\beta)
\end{align*}
where $\mathcal{L}=\max\{|\beta|:\beta\in\mathcal{B}_S(\Gamma_{\mathcal{F}}(\Lambda))\}$.

By lemma \ref{last1} the variables of $\mathcal{X}_S\mathcal{F}_{i_1}$ are indexed by either $j\in\mathcal{I}_S$ or by the sequence $\gamma(\beta)$ for some $\beta\in\mathcal{B}_S(\Gamma_\mathcal{F})$. Hence,
\begin{equation}\label{const2}
\tilde{\Lambda}_{ij}=
\begin{cases}
\Lambda_{ij} & \text{if} \ \ i,j\in\mathcal{I}_S\\
\mathcal{P}_\omega(\beta)\lambda^{|\beta|-2} & \text{if} \ \ j\in\mathcal{I}_S \ \ \text{and} \ \ i=\gamma(\beta)\\
0 & \text{otherwise}
\end{cases}
\end{equation}
satisfy condition (\ref{eq2.3}) for $\mathcal{X}_S\mathcal{F}_j$. The constants in (\ref{const1}) and (\ref{const2}) therefore give a complete set of the constants $\tilde{\Lambda}_{ij}$ satisfying (\ref{eq2.3}) for $\mathcal{X}_SF$.

Suppose $S=\{v_1,\dots,v_m\}$. Note that the bijection defined in (\ref{beta}) also gives a bijection
\begin{equation}\label{biject}
b:\mathcal{B}_{ij}(\Gamma_\mathcal{F}(\Lambda);S)\rightarrow\mathcal{B}_{ij}(\Gamma_{\mathcal{X}_S\mathcal{F}}(\tilde{\Lambda});S) \ \ \text{for} \ \ 1\leq i,j\leq m.
\end{equation}
In particular, if $\beta\in\mathcal{B}_S(\Gamma_\mathcal{F}(\Lambda))$ then equations (\ref{const1}) and (\ref{const2}) imply that
$$\Omega_{\Gamma_{\mathcal{X}_S\mathcal{F}}(\tilde{\Lambda})}(b(\beta))=1,0,1,\dots,0,\mathcal{P}_\omega(\beta)\lambda^{|\beta|-2}.$$
Let $\Gamma_{\mathcal{X}_S\mathcal{F}}(\tilde{\Lambda})=(\mathcal{V},\mathcal{E},\mu)$. As $|\beta|=|b(\beta)|$ then $$\mathcal{P}_\mu\big(b(\beta)\big)=\frac{\mathcal{P}_\omega(\beta)\lambda^{|\beta|-2}}{\lambda^{|\beta|-2}}=\mathcal{P}_\omega(\beta).$$

Given the one-to-one correspondence noted in equation (\ref{biject}) it follows that $$\mathcal{R}_S(\Gamma_\mathcal{F}(\Lambda))=\mathcal{R}_S(\Gamma_{\mathcal{X}_S\mathcal{F}}(\tilde{\Lambda})).$$
As $S\in st_0(\Gamma_\mathcal{F}(\Lambda))$ and $S\in st_0(\Gamma_{\mathcal{X}_S\mathcal{F}}(\tilde{\Lambda}))$ then corollary \ref{cor100} implies that $\Gamma_\mathcal{F}(\Lambda)$ and $\Gamma_{\mathcal{X}_S\mathcal{F}}(\tilde{\Lambda})$ have the same nonzero spectrum. Hence, $\rho(\tilde{\Lambda})=\rho(\Lambda)$ for this choice of constants completing the proof.
\end{proof}

\begin{example}\label{ex1500}
Let $\mathcal{L}(x)=4x(1-x)$ be the standard logistic map and let the function $\mathcal{Q}(x)=1-x^2$ where both $\mathcal{L}:[0,1]\rightarrow[0,1]$ and $\mathcal{Q}:[0,1]\rightarrow[0,1]$. Suppose the local systems
$$
T_i(x_i)=
\begin{cases}
\mathcal{L}(x_i) \ \ \text{for} \ \ i=2,4\\
\mathcal{Q}(x_i) \ \ \text{for} \ \ i=1,3
\end{cases}.
$$
Let $F:[0,1]^4\rightarrow [0,1]^4$ be the interaction given by
$$F(\mathbf{x})=\left[\begin{array}{l}
\frac{1}{4}x_2\\
\frac{1}{4}x_1+\frac{1}{4}x_4\\
\frac{1}{4}x_1\\
\frac{1}{4}x_1+\frac{1}{4}x_3\\
\end{array}
\right].$$

Consider the dynamical network $(\mathcal{F},X)$ where $\mathcal{F}=F\circ T$ and $X=[0,1]^4$. Note that the constants $L_i=4$ for $i=2,4$ and $L_i=2$ for $i=1,3$ satisfy equation (\ref{eq0.1}) for the local systems $(T,X)$. For $\Lambda_{ij}=\max_{\textbf{x}\in X}|(DF)_{ji}(\textbf{x})|$ the matrix
$$
M_\mathcal{F}=
\left[\begin{array}{cccc}
0&\frac{1}{4}&\frac{1}{4}&\frac{1}{4}\\
\frac{1}{4}&0&0&0\\
0&0&0&\frac{1}{4}\\
0&\frac{1}{4}&0&0\\
\end{array}
\right]
\cdot diag[2,4,2,4]
=
\left[\begin{array}{cccc}
0&1&\frac{1}{2}&1\\
\frac{1}{2}&0&0&0\\
0&0&0&1\\
0&1&0&0\\
\end{array}
\right].
$$
As the spectral radius $\rho(M_\mathcal{F})=1.08>1$ then theorem \ref{stability} does not directly apply to the dynamical network $\mathcal{F}$.

Moreover, suppose $\tilde{\mathcal{F}}=(F\circ T)\circ id$ generates the dynamical network $(\tilde{\mathcal{F}},X)$ having no local dynamics. For $\tilde{\Lambda}_{ij}=\max_{\textbf{x}\in X}|(D\tilde{\mathcal{F}})_{ji}(\textbf{x})|$ one can compute that $\tilde{\Lambda}=M_\mathcal{F}$. As $\tilde{\Lambda}>1$ then considering $(\mathcal{F},X)$ as a network without local dynamics does not improve our estimate of the network's stability as it did in example \ref{local}.

However, note that
$$\mathcal{F}(\mathbf{x})=\left[\begin{array}{l}
\mathcal{F}_1(x_2)\\
\mathcal{F}_1(x_1,x_4)\\
\mathcal{F}_1(x_1)\\
\mathcal{F}_1(x_1,x_3)\\
\end{array}
\right]=\left[\begin{array}{l}
\frac{1}{4}\mathcal{L}(x_2)\\
\frac{1}{4}\mathcal{Q}(x_1)+\frac{1}{4}\mathcal{L}(x_4)\\
\frac{1}{4}\mathcal{Q}(x_1)\\
\frac{1}{4}\mathcal{Q}(x_1)+\frac{1}{4}\mathcal{Q}(x_3)\\
\end{array}
\right].$$

As $\mathcal{F}$ has the form of the map given in example \ref{ex13} then following this example one can compute
$$\mathcal{X}_S\mathcal{F}(\textbf{x})=
\left[
\begin{array}{l}
\mathcal{X}_{S}\mathcal{F}_{1}\\
\mathcal{X}_{S}\mathcal{F}_{2}\\
\mathcal{X}_{S}\mathcal{F}_{142}\\
\mathcal{X}_{S}\mathcal{F}_{1342}\\
\mathcal{X}_{S}\mathcal{F}_{2;1342}\\
\end{array}
\right]=\left[\begin{array}{l}
\frac{1}{4}\mathcal{L}(x_2)\\
\frac{1}{4}\mathcal{Q}(x_1)+\frac{1}{4}\mathcal{L}\big(\mathcal{Q}(x_{142})+\frac{1}{4}\mathcal{Q}(\frac{1}{4}\mathcal{Q}(x_{1342}))\big)\\
x_1\\
x_{2;1342}\\
x_1
\end{array}
\right]$$
for $S=\{v_1,v_2\}$. Here, $X_{142}$, $X_{2;1342}$, $X_{1342}=[0,1]$ implying $X_S=[0,1]^5$.

Letting $\bar{\Lambda}_{ij}=\max_{\textbf{x}\in X_S}|(D\mathcal{X}_S\mathcal{F})_{ji}(\mathbf{x})|$ it follows that
$$\bar{\Lambda}=\left[\begin{array}{ccccc}
0 & 1/2 & 1 & 1 & 0\\
1 & 0 & 0 & 0 & 0\\
0 & 0.265 & 0 & 0 & 0\\
0 & 0 & 0 & 0 & 1\\
0 & 0.012 & 0 & 0 & 0\\
\end{array}
\right].$$
As $\rho(\bar{\Lambda})=0.90<1$ corollary \ref{cor30} implies that the dynamical network $(\mathcal{F},X)$ does in fact have a globally attracting fixed point.
\end{example}

Importantly, the method of using network expansions generalizes the method used in \cite{Afriamovich07,Afriamovich10} for determining whether a network has a unique global attractor. As a demonstration we give the following examples.

\begin{example}
For the dynamical network $(\mathcal{F},X)$ let
$$L=\max_{1\leq i\leq n}L_i \ \ \text{and} \ \ \Lambda_{max}=\max_{1\leq i,j\leq n}\Lambda_{ij}.$$ Recall the spectral radius of the graph $\Gamma_{\mathcal{F}}$ (considered as an unweighted graph) is $\rho(\Gamma_{\mathcal{F}})$. The condition in \cite{Afriamovich07} guaranteeing the global stability of the dynamical network $(\mathcal{F},X)$ is the following: If the product
\begin{equation}\label{suff1}
L\Lambda_{max}\rho(\Gamma_{\mathcal{F}})<1
\end{equation}
then $(\mathcal{F},X)$ has a globally attracting fixed point.

Consider the dynamical network $(\mathcal{F},X)$ given in example \ref{ex2}. In this case $L=4$, $\Lambda_{max}=2/9$, and $\rho(\Gamma_{\mathcal{F}})=2$. Hence, $$L\Lambda_{max}\rho(\Gamma_{\mathcal{F}})=16/9>1.$$ Therefore, $(\mathcal{F},X)$ does not satisfy condition (\ref{suff1}). However, as the spectral radius $\rho(\mathcal{M}_{\mathcal{F}})<1$ then the dynamical network $(\mathcal{F},X)$ has a global attractor by theorem \ref{stability} (see example 2).
\end{example}

\begin{example}
Suppose $(\mathcal{F},X)$ is the dynamical network given in example \ref{ex1500}. Using the mathematical framework established in \cite{Afriamovich10} one can calculate that the topological pressure $P^{(k)}(\varphi)$ of the system is given by
$$P^{(k)}(\varphi)=\ln(1.08).$$
The condition given in \cite{Afriamovich10} that guarantees the global stability of the dynamical network $(\mathcal{F},X)$ is if the systems topological pressure $$P^{(k)}(\varphi)<0$$ (see theorem 2, \cite{Afriamovich10}). As $\ln(1.08)>0$ then the method introduced in \cite{Afriamovich10} does not imply the global stability of $(\mathcal{F},X)$. However, by expanding the network it is possible to demonstrate that the system does have a globally attracting fixed point (see example \ref{ex1500}).
\end{example}

As a final observation in this section, we note that it is possible to sequentially expand a dynamical network and thereby improve ones estimate of whether the original unexpanded system has a globally attracting fixed point.



\section{Concluding Remarks}
The goal of this paper is twofold; introduce isospectral graph transformations and use these transformations to investigate the dynamic properties of dynamical networks with arbitrary topologies (graph of interactions). For the first, we described the general process of isospectral graph reductions in which a graph is collapsed around a specific set of vertices. We then considered sequences of isospectral reductions and showed that under mild assumptions a typical graph could be uniquely reduced to a graph on any subset of its vertex set.

In \cite{BW09}, such reductions and sequences of reductions are used to improve the classical eigenvalue estimates found in \cite{Brauer47,Brualdi82,Hadamard03,Horn85,Varga09}. These estimates of Gershgorin et al. can be traced back to work done by L\'{e}vy, Desplanques, Minkowski, and Hadamard
\cite{Levy81,Desplanques87,Minkowski00,Hadamard03} on diagonally dominant matrices. The main result of \cite{BW09} is that such eigenvalue estimates improve as the graph associated to a matrix is reduced via the method of isospectral reduction given in this paper.

Isospectral reductions are also relevant to the theory of networks as they provide a flexible means of coarse graining networks (i.e. viewing the network at some appropriate scale) while simultaneously introducing new equivalence relations on the space of all networks. From the point of view of applications, isospectral graph reductions are a tool which allow experimentalists to modify the networks with which they are working. However, the particular transformation requires the expertise of the experimentalist (biologist, ecologist, etc.) to determine an appropriate set of network elements (vertices) over which to reduce the network.

In this paper we additionally considered isospectral graph transformations over fixed weight sets. Such transformations allow one to transform an arbitrary graph while preserving the set of edge weights of the graph along with its spectrum. Motivated by these procedures we also introduced the notion of a dynamical network expansion.

Dynamical network expansions modify a dynamical network (i.e. a dynamical system with a graph structure) in a way that preserves the system's dynamics but alters its associated graph structure (graph of interactions). In this paper we demonstrated that this procedure allows one to establish global stability of a more general class of dynamical networks than considered in \cite{Afriamovich07,Afriamovich10}. That is, such expansions provide more information and can be used to obtain stronger results regarding network dynamics than the direct analysis of the network. It is worth mentioning that, as any finite dimension dynamical system with discrete time can be considered a dynamical network, this technique can be applied to a very large class of systems.

Lastly, the results of the present paper introduce various approaches to simplifying a graph's structure while maintaining the graph's (network's) spectrum. Therefore, these techniques can be used for optimal design, in the sense of structure simplicity of dynamical networks with prescribed dynamical properties ranging from synchronizability to chaoticity \cite{Afriamovich07, Blank06}.

\section{Acknowledgments}
We would like to thank C. Kirst and M. Timme for valuable comments and discussions. This work was partially supported by the NSF grants DMS-0900945, BECS-1024868, and the Humboldt Foundation. 


\end{document}